\documentclass[a4paper,10pt]{article}
\usepackage[english]{babel}
\usepackage[utf8]{inputenc}
\usepackage{amsmath}
\usepackage{indentfirst}
\usepackage{fancyhdr}
\usepackage{amssymb}
\usepackage{amsthm}
\usepackage[dvips]{graphicx}
\usepackage{color}
\usepackage{cancel} 
\usepackage[dvipsnames]{xcolor}
\usepackage{eucal}
\usepackage{latexsym}
\usepackage{chngcntr}
\usepackage{caption}

\usepackage{subcaption}
\usepackage{faktor}
\usepackage{microtype}
\usepackage{newlfont}
\usepackage{hyperref,geometry,mathtools} 
\usepackage{todonotes}
\usepackage{amsmath, amsfonts, amsthm,amssymb,  bbm, dsfont, lmodern, mathrsfs}
\usepackage{amsmath}
\usepackage{amsthm}
\usepackage{tikz-cd}
\usepackage{wela}
\usepackage{tikz-cd}
\usepackage{multicol}
\usepackage{cancel}
\usepackage{mathtools}
\usepackage{graphicx}

\newcommand{\white}[1]{{\textcolor{white}{#1}}}

\newcommand{\cV}{{\mathcal V}}
\newcommand{\tf}{{\mathtt f}}

\newcommand{\umu}{\underline{\mu}}
\newcommand{\ent}[6]{\begingroup 
\setlength\arraycolsep{-2pt}\begin{matrix}{\tB_{#1}^{[#2]}} &\white{{|}^|}_{#3,#4}^{#5,#6}\end{matrix}\endgroup}
\newcommand{\sent}[6]{ {{\tB_{#1}^{[#2]}}_{#3,#4}^{#5,#6}}}

\newcommand{\Res}[2]{\mathtt{R}\begingroup 
\setlength\arraycolsep{0pt}{\scriptsize \begin{matrix} #2 \\[0.7mm] #1 \end{matrix} }\endgroup}

\newcommand{\e}{\epsilon}

\newcommand{\im}{\mathrm{i}\,}

\newcommand{\und}[1]{{\underline{#1}}}

\newcommand{\molt}[2]{\left(#1\,,\,#2\right)}

\newcommand\restr[2]{{% we make the whole thing an ordinary symbol
  \left.\kern-\nulldelimiterspace % automatically resize the bar with \right
  #1 % the function
  \vphantom{\big|} % pretend it's a little taller at normal size
  \right|_{#2} % this is the delimiter
  }}

\allowdisplaybreaks
\theoremstyle{plain}
\newtheorem{lem}{Lemma}
\newtheorem{teo}[lem]{Theorem}
\newtheorem{prop}[lem]{Proposition}
\newtheorem{ass}[lem]{Assumption}

\theoremstyle{definition}

\newtheorem{rmk}[lem]{Remark}

\renewcommand{\bar}{\overline}

\newcommand{\sgn}{\mathrm{sgn}}

\newcommand{\vet}[2]{\begin{bmatrix}#1 \\ #2 \end{bmatrix}}
\newcommand{\uno}{\mathrm{Id}}

\newcommand{\bR}{\mathbb{R}}
\newcommand{\bT}{\mathbb{T}}
\newcommand{\bZ}{\mathbb{Z}}
\newcommand{\bN}{\mathbb{N}}

\newcommand{\bC}{\mathbb{C}}

\newcommand{\cL}{\mathcal{L}}
\newcommand{\cO}{\mathcal{O}}

\newcommand{\cJ}{\mathcal{J}}
\newcommand{\kB}{\mathfrak{B}}
\newcommand{\cB}{{\cal B}}

\newcommand{\tB}{\mathtt{B}}

\newcommand{\tJ}{\mathtt{J}}
\newcommand{\tL}{\mathtt{L}}
\newcommand{\de}{\mathrm{d}}
\newcommand{\pa}{\partial}

\newcommand{\tR}{\mathtt{R}}

\newcommand{\bro}{\bar\rho}

\numberwithin{equation}{section}
\counterwithin{lem}{section}
\title{\bf First isola of  modulational instability of Stokes waves in deep water}
\begin{document}

\author{Massimiliano Berti, Alberto Maspero\thanks{
International School for Advanced Studies (SISSA), Via Bonomea 265, 34136, Trieste, Italy.\\ 
 \textit{Emails: } \texttt{berti@sissa.it},  \texttt{alberto.maspero@sissa.it}
 }, Paolo Ventura\thanks{ University of Milan, Via C. Saldini, 50, 20133 Milan, Italy. \textit{Email: } \texttt{paolo.ventura@unimi.it}}}

\maketitle

\begin{abstract}
We prove high-frequency modulational instability of small-amplitude Stokes waves in deep water under  longitudinal perturbations, providing the first  isola of unstable eigenvalues  branching off  from $\im \tfrac34$. 
Unlike the finite depth case this is a degenerate problem
and the real part of the unstable eigenvalues has a much smaller size than in finite depth. 
By a symplectic version of Kato theory 
we reduce to search the eigenvalues of a $2\times 2$ 
Hamiltonian and  reversible matrix
which has eigenvalues with non-zero real part if and only if a certain analytic function is not identically zero.
In deep water we prove that the Taylor coefficients up to order three of this function vanish, but not the  fourth-order one. 
\end{abstract}

\maketitle

\begin{flushright}
{\em To Thomas, who taught us beautiful mathematics\\ and to   be always  grateful for the good things happening in  life}
\end{flushright}

\tableofcontents

\section{Introduction and main result} 

Stokes waves are periodic  
solutions 
of the 
pure gravity water waves equations, traveling at constant speed. Since their discovery 
by Stokes \cite{stokes} in 1847 and their rigorous mathematical existence 
proof in 
\cite{Nek, LC, Struik}, 
they have been the 
object of intense studies, regarded as a key first step toward better understanding the complicated flow evolution of the water waves equations. 
Pioneering  experimental and formal works  by Benjamin and Feir \cite{Benjamin, BF},  Lighthill \cite{Ligh},  Zakharov \cite{Zak1}, Whitham \cite{Wh} highlighted, 
more than fifty years ago, that 
Stokes waves are 
 unstable under long wave perturbations, a phenomenon that 
nowadays  goes  by the name of Benjamin-Feir or {\em modulational instability}.
Rigorous mathematical proofs  were later given by  Bridges-Mielke \cite{BrM} in finite depth and by Nguyen-Strauss \cite{NS} in deep water. 
These works prove  existence of unstable spectrum near the origin of the complex plane of 
the linearized water waves operator  $\cL_\e$ (see \eqref{cLepsilon}) at a  
 Stokes wave of small amplitude $ \epsilon $, when it is 
 regarded as an unbounded operator  on $L^2(\bR)$. We also mention the recent nonlinear modulational instability result 
 in Chen-Su \cite{ChenSu}.

In a recent series of works we gave the  complete description of the portion of the spectrum $\sigma_{L^2(\bR)}(\cL_\e) $ near the origin of the complex plane both in  deep water \cite{BMV1} and finite depth \cite{BMV2, BMV_ed},  proving the existence of  an unstable spectral branch of eigenvalues outside the imaginary axis forming a  figure ``8'', as first observed  numerically by  Deconinck-Oliveras \cite{DO} (see also the formal computations in \cite{CD}).

\smallskip
The behaviour of the spectrum  $\sigma_{L^2(\bR)}(\cL_\e) $ away from zero is still not sufficiently  investigated, 
although highly relevant for the 
stability/instability of the Stokes waves.
Numerical results by Deconinck-Oliveras \cite{DO}  shed light on the presence of ``\emph{isolas} of instability'' --isolated elliptic
shaped islands of unstable eigenvalues  away from the origin-- which suggests the appearance of 
unstable spectrum along the whole imaginary axis, with  
real part decreasing exponentially  at infinity.  
Because of the Hamiltonian character of the operator $\cL_\e$, unstable eigenvalues can occur only as perturbations of multiple purely imaginary eigenvalues of $\cL_0$.  
The nonzero multiple  eigenvalues of $\cL_0$ 
are enumerated in literature by an integer $p \geq 2$ --it turns out they are all double-- and we shall follow this convention. 
Formal expansions  describing  the first two   isolas ($p=2,3$) 
of unstable eigenvalues were obtained 
by 
Creedon-Deconinck-Tritchenko \cite{CDT} in both finite and infinite depth, see also  
\cite{DT,CDT2, CDT3}  for other asymptotic models.
A rigorous analytical result about the first isola of instability $p=2$  was given recently  by 
Hur and Yang \cite{HY} in finite depth 
for pure gravity waves 
and in \cite{HY2} for gravity-capillary waves.
Unfortunately the approach in \cite{HY} does not apply in  infinite depth, as it relies on spatial dynamics which fails in deep water (similarly to \cite{BrM}). 
We finally mention the very  recent paper \cite{CNS} which,  relying on the spectral approach in 
\cite{BMV1},   proved  the 
instability of the Stokes waves under 
transversal perturbations. 

\smallskip

The goal of this paper is to rigorously  prove  the  existence of the first isola of  instability ($p=2$) 
in the deep-water case, under longitudinal perturbations. As we explain below this problem is considerably more difficult than in finite depth, because it is {\it degenerate}. 
In Theorem \ref{thm:main} we show that the spectrum $\sigma_{L^2(\bR)}(\cL_\e) $ has a branch of eigenvalues   near  $\im \frac34$ with non zero real part, 
 forming a {\it very narrow} elliptic shaped curve in the complex plane, see Figures \ref{comeBernard} and \ref{figure0}. Formula \eqref{isola.intro} provides the ellipse which asymptotically approximate the $p=2$-isola of unstable eigenvalues. 
 
  \begin{figure}[h] 
 \centering
 \includegraphics[width=9cm]{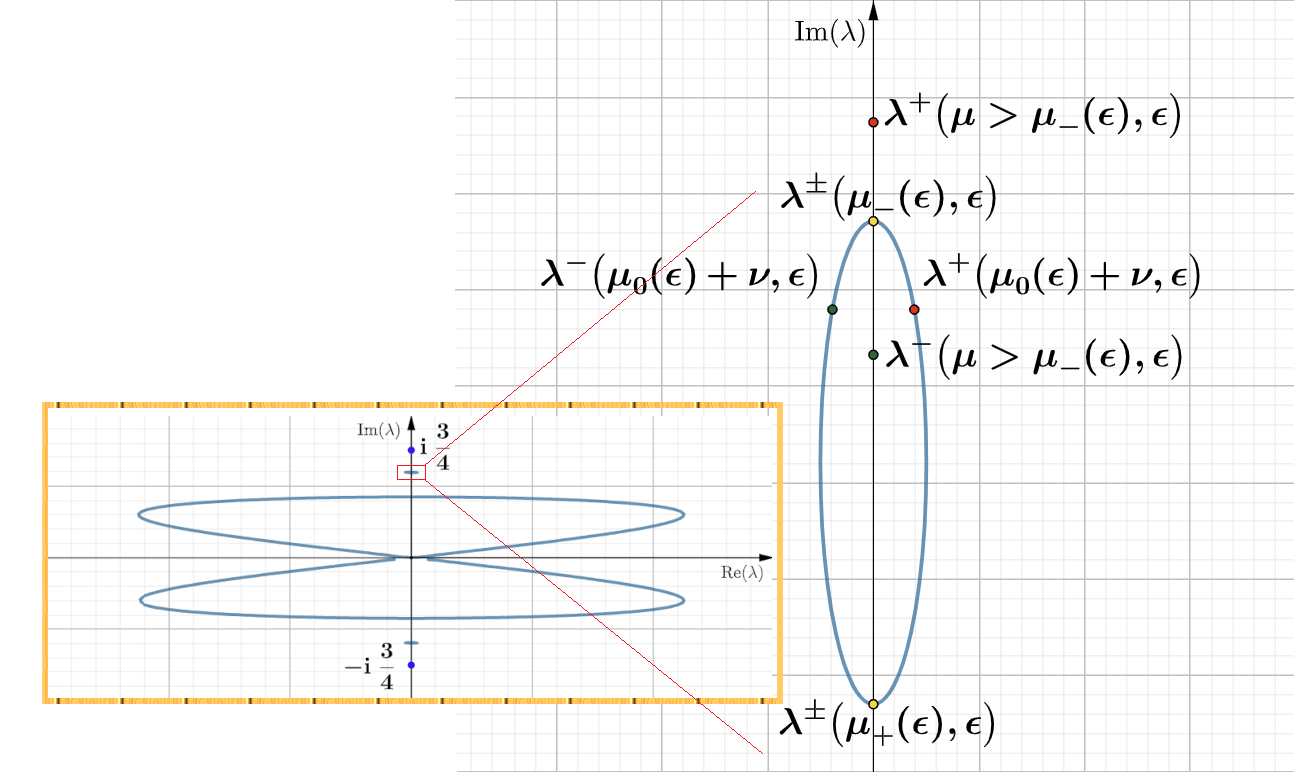}
\caption{\label{comeBernard} Part of the spectrum of $\cL_\e $ 
in the deep-water case. The figure  ``8'' was proved in \cite{BMV1}. The two small symmetric isolas  (detailed in Figure \ref{figure0})  are the subject of this article.  }
\end{figure}

Let us now state precisely our main result. Since the operator  $\cL_{\e}$  has   $2\pi$-periodic coefficients, arising by the linearization of the water waves equations at a $2\pi$-periodic Stokes wave, 
its spectrum is conveniently described via Bloch-Floquet theory, which ensures  that 
$\sigma_{L^2(\bR)}(\cL_\e) = \bigcup_{\mu \in [0,\frac12)} \sigma_{L^2(\bT)}(\cL_{\mu,\e})$ where $\cL_{\mu,\e}:= e^{- \im \mu x} \cL_\e e^{\im \mu x}$. 
Our main result describes the  spectrum  near $\im \tfrac34$ of the operator $ \cL_{\mu,\e}$,   
for any $(\mu, \e)$ sufficiently close to $(\tfrac14,0)$, value  at which
$\cL_{\frac14,0}$ has a double eigenvalue (the  
closest one to $0$).
%(actually $\cL_{\frac14,0}$ has infinitely many isolated double eigenvalues along the imaginary axis, see Lemma \ref{thm:unpert.coll} below). 
\\[1mm]
$ \bullet $ {\bf Notation:}  Along the paper we denote by  $r (y^{m_1}\e^{n_1},\dots,y^{m_p}\e^{n_p}) $ (the variable $y $ is called $ \mu,\delta,\nu $ depending on the context) 
a real analytic function  satisfying,  for some $ C > 0 $ and for any small value of $(y, \e)$,
 $ |r (y^{m_1}\e^{n_1},\dots,y^{m_p}\e^{n_p}) |  \leq C \sum_{j = 1}^p |y|^{m_j}|\e|^{n_j}\, 
 . $
\begin{teo}\label{thm:main}
 There exist $\e_1,\delta_0>0$ and real  analytic functions $\mu_0,\mu_\pm \colon [0, \e_1) \to B_{\delta_0}(\tfrac14)$ with 
$ \mu_+(\e)\leq \mu_0(\e) \leq \mu_-(\e)$ of the form % and satisfying 
% $\mu_0(0)=\mu_+(0)=\mu_-(0) = \frac14$
\begin{equation}\label{mupm}
\mu_0 (\e) = \frac14 - \frac{57}{64} \e^2 + r(\e^3)\, , \quad
\mu_\pm (\e) = \mu_0(\e)  \mp \frac{111 \sqrt{3}}{1024}\e^4 +r(\e^5) \, ,
\end{equation}
%, $\mu_0'(0)=\mu_\wedge'(0)=\mu_\vee'(0)=0$ and $\mu_0''(0)=-\frac{57}{32}$
 such that for  any $(\mu, \e) \in B_{\delta_0}(\tfrac14)\times [0,\e_1)$ the 
following holds true. Defining $ \nu_\pm (\epsilon) := \mu_\pm (\epsilon) - \mu_0 (\epsilon) $, 
the operator $ \cL_{\mu,\e}$ at  $(\mu,\e) = (\mu_0(\e) + \nu, \e)$ possesses two eigenvalues  of the form 
\begin{align}\label{final.eig}
& \lambda^\pm \big(\mu_0(\e) + \nu, \e \big) = \\ \notag &\quad\begin{cases} \im \dfrac{3}4 +\im {\dfrac{4}3}\nu - \im{\dfrac{55}{32}}  \e^2+ \im r(\e^{{3}}, \nu\e^2, \nu^2)   \pm \dfrac12 \sqrt{ D(\mu_0(\e)+\nu,\e)}  &\mbox{if }  \nu \in \big( \nu_+(\e) , \nu_-(\e)\big) \, , \\[3mm]
\im \dfrac{3}4 +\im {\dfrac{4}3\nu} - \im{\dfrac{55}{32} } \e^2+ \im r(\e^{{3}}, \nu\e^2, \nu^2)
 \pm \dfrac{\im\!}2 \sqrt{ |D(\mu_0(\e)+\nu,\e)| } &\mbox{if } \nu \notin \big( \nu_+(\e) , \nu_-(\e)\big)\, ,
\end{cases} 
\end{align}
where
\begin{equation}\label{degenerateexpD0}
D(\mu_0(\e)+\nu,\e) = \frac{4107}{65536} \e^8 - \frac{16}{9} \nu^2 - \frac{37}{128} \nu\e^6 +r(\e^{9},\nu\e^7,\nu^2\e^2,\nu^3)\, .
\end{equation} 
For any  fixed $\e \in (0,\e_1)$,  
 the pair of unstable eigenvalues $\lambda^\pm(\mu,\e)$ depicts, as $\mu$ varies in $\big(\mu_+(\e),\mu_-(\e)\big)$, an ellipse-like curve in the complex plane, i.e. a closed analytic  curve that
  intersects orthogonally the imaginary axis 
  and encircles a convex region.
 \end{teo}
Let us make some comment on the result.\smallskip

1. According to \eqref{final.eig}, for any Floquet parameter $\mu \in \big(\mu_+(\e) , \mu_-(\e)\big)$ the eigenvalues $\lambda^{\pm}(\mu,\e)$ have opposite real part.
As $\mu \to \mu_{\pm}(\e)$, the eigenvalues collide on the imaginary axis. For $\mu<\mu_{+}(\e)$ or $\mu> \mu_-(\e)$ the two eigenvalues are purely imaginary. Being $\cL_\e$ a real operator,  also the mirror spectral branch of eigenvalues
$\bar{\lambda^\pm(-\mu,\e)} $ is present. 
% for $(\mu, \e) \in B_{\delta_0}(\tfrac14)\times [0,\e_1)$.
%$\mu \in \big(\mu_+(\e) , \mu_-(\e)\big)$.

\smallskip

2. By fixing $\e >0$ and letting  $\mu$ vary in the interval $\big(\mu_+(\e), \mu_-(\e)\big)$, the two eigenvalues  $\lambda^{\pm}(\mu,\e)$ depict an ellipse-like curve. By
discarding the remainders $r(\nu^\alpha)$, $\alpha=2,3 $, 
in the expressions \eqref{final.eig} and
\eqref{degenerateexpD0}, the  real and imaginary part of the approximate eigenvalues parameterize the ellipse
\begin{equation}\label{isola.intro}
x^2  + \frac14 \big(y-y_0(\e)\big)^2\big(1+r(\e^2)\big) =  \frac{4107}{262144} \e^8(1+r(\e)) \, ,\quad y_0(\e) = \frac34-\frac{55}{32}\e^2+r(\e^4)\,.
\end{equation}

\begin{figure}[h]
 \centering
 \includegraphics[width=7cm]{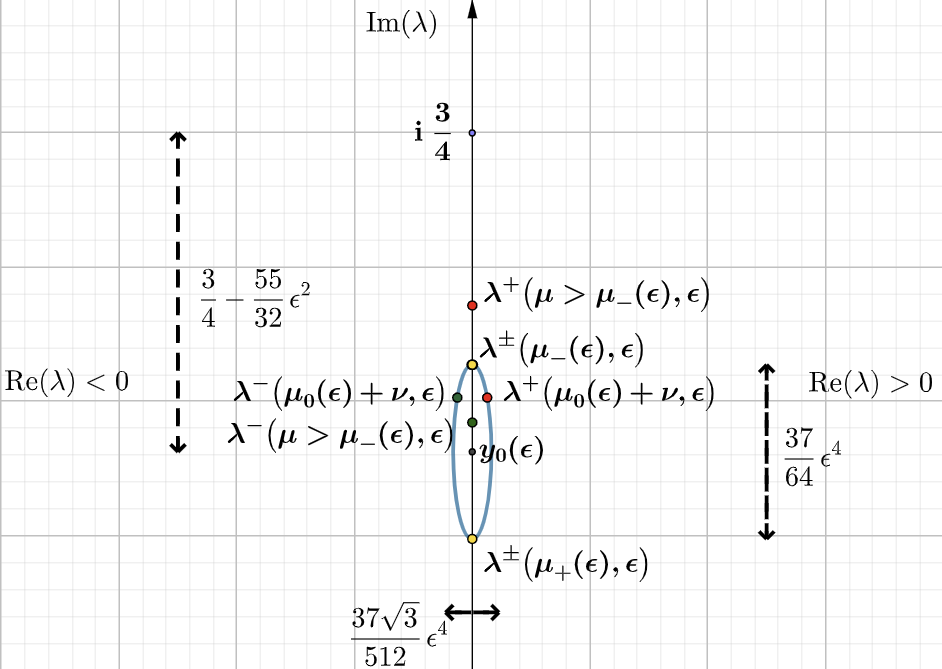}
\caption{First isola depicted by the two symmetric non-purely imaginary eigenvalues of $\cL_{\mu,\e} $ close to $\im\tfrac34 $. The isola does not encircle $\im \frac34 $. 
\label{figure0}}
\end{figure}

3. Theorem \ref{thm:main} actually provides the expansion of the two eigenvalues of $\cL_{\mu,\e}$ close to $\im \frac34$ for all values of $(\mu,\e)$ in  $(\tfrac14 - \delta_0, \tfrac14 + \delta_0) \times [0,\e_1)$. The 
analytic curves $\mu_\pm(\e) $ in 
\eqref{mupm} divide such rectangle in two separated regions: 
one  where  $\cL_{\mu,\e}$ has two eigenvalues 
with non-zero real part, and another one  where the eigenvalues are purely imaginary. 
\begin{figure}[h]\centering
\includegraphics[scale=0.5]{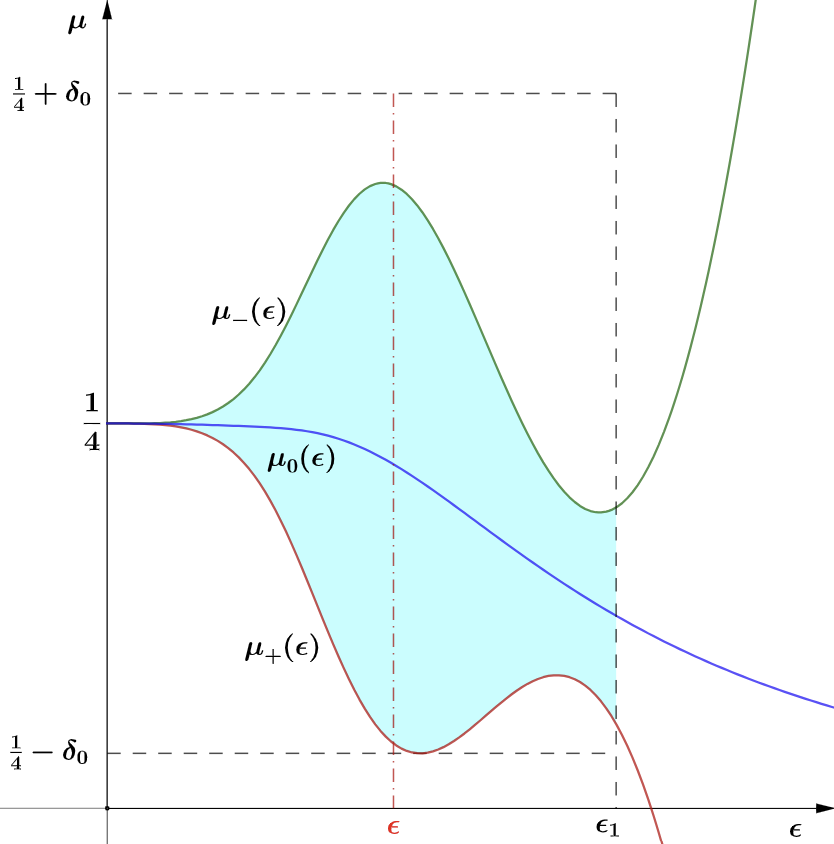}
\caption{\label{actualinstabilityregion} The degenerate instability region. We boxed with black dashed lines the validity box of Theorem \ref{thm:main}. At any  fixed  $\e$, for $(\mu,\e) $ in the colored cusp-shape region, one has the formation of the isola of instability in Figure \ref{figure0}.}
\end{figure}
\smallskip

4. The work \cite{HY} describes the first isola in finite depth. Such a  case is non-degenerate (see
\eqref{cases}) for almost any depth. On the contrary, the 
infinite depth  case is {\it degenerate}, and its analysis requires 
 the fourth-order  expansion of the Stokes waves, 
 as we now explain. 

\smallskip

We briefly describe the proof and its difficulties. 
Exploiting that $\cL_{\mu,\e}$ is a Hamiltonian and reversible operator, 
we use  a symplectic version of Kato's similarity transformation theory to compute a symplectic basis  of the two-dimensional invariant subspace $\cV_{\mu,\e}$ associated with the two eigenvalues of $\cL_{\mu,\e}$ close to $\im \frac34 $.
  The action of $\cL_{\mu,\e}\vert_{\cV_{\mu,\e}}$  is then represented by a $2\times 2$ Hamiltonian and reversible matrix of the form 
$ \tL(\mu,\e)={\scriptsize \begin{pmatrix} -\im \alpha(\mu,\e) &  \beta(\mu,\e) \\ \beta(\mu,\e) & \im \gamma(\mu,\e) \end{pmatrix}}$, where $\alpha(\mu,\e)$, $\beta(\mu,\e)$ and $\gamma(\mu,\e)$ are real analytic functions.

We like to remember that this method, that we initiated in \cite{BMV1, BMV2}, was inspired to us by the approach that  Thomas Kappeler developed in   \cite{Kappeler}, in the context of selfadjoint operators, to study the spectral gaps of the Lax-operator of the KdV equation (this line of research was  further developed in   \cite{KP,toda}).

After our symplectic reduction one needs only to investigate the eigenvalues of the matrix $ \tL(\mu,\e)$.  In Section \ref{sec:mr} we prove 
% an abstract criterion  ensuring 
a necessary and sufficient condition for a $2\times 2$ matrix of this form to have non-purely imaginary eigenvalues.  Theorem \ref{instacrit}
proves that  $\tL(\mu,\e)$ 
 has two unstable eigenvalues if and only if the real analytic function 
$\e \mapsto \beta(\mu_0(\e), \e)$ is not identically zero, where $\mu_0(\e)$ is the analytic function such that
$\alpha(\mu_0(\e),\e) + \gamma(\mu_0(\e),\e) \equiv 0$.
This instability criterion amounts to 
prove that $\beta(\mu_0(\e),\e) $ has a nonzero Taylor coefficient at $ \e = 0 $.
The first Taylor coefficients of  $\beta(\mu_0(\e),\e) $  are computed  in    \eqref{beta(mu,e)}  in terms of the Taylor coefficients of  $\alpha(\mu,\e)$, $\beta(\mu,\e)$ and $\gamma(\mu,\e)$. 
In deep water, it turns out that the coefficient  $\beta_1$ in    \eqref{beta(mu,e)} vanishes, cfr. \eqref{epsilonspento2}; this is the degeneracy  we mentioned above (whereas 
$ \beta_1 $ is not zero for almost any finite depth). 
We are then led to  compute  the coefficient of $\e^4$ in  \eqref{beta(mu,e)}, which 
 depends on  a higher-order Taylor expansion of  
$\alpha(\mu,\e)$, $\beta(\mu,\e)$ and $\gamma(\mu,\e)$  given in Theorem \ref{thm:expansion} and proved throughout Sections \ref{sec:5} and \ref{section:conti}. 
Via a careful analysis to isolate only the relevant terms, we finally prove that the fourth-order  coefficient of \eqref{beta(mu,e)}
is
$\frac{37\sqrt{3}}{512} \neq 0 $. We thus conclude that the   eigenvalues of $ \tL(\mu,\e)$, i.e. the ones of $\cL_{\mu,\e}$ close to $ \im \tfrac34$, are unstable in a certain region of $(\mu,\e)$ as stated in Theorem \ref{thm:main}.

 Theorem \ref{thm:main} is a direct consequence of Theorem \ref{thm:expansion} that we now rigorously present. 

\smallskip 
\noindent {\bf The water waves equations.}
We consider the Euler equations for a 2-dimensional  
incompressible 
and irrotational fluid under the  action of  gravity % in deep water
 filling the 
region 
$ { \mathcal D}_\eta := \big\{ (x,y)\in \bT\times \bR\;:\;  y< \eta(t,x) \big\}$, 
$   \bT :=\bR/2\pi\bZ  $,  
with  space periodic boundary conditions. 
The irrotational velocity field is the gradient  
of a harmonic scalar potential $\Phi=\Phi(t,x,y) $  
determined by its trace $ \psi(t,x)=\Phi(t,x,\eta(t,x)) $ at the free surface
$ y = \eta (t, x ) $.
Actually $\Phi$ is the unique solution of 
$    \Delta \Phi = 0 $ in $ {\mathcal D}_\eta $ satisfying 
 the Dirichlet condition $   \Phi(t,x,\eta(t,x)) = \psi(t,x) $ and 
$  \nabla \Phi(t,x,y) \to 0 $ as $ y\to - \infty $.
The time evolution of the fluid is determined by two boundary conditions at the free surface. 
The first is that the fluid particles  remain, along the evolution, on the free surface   
(kinematic boundary condition), 
and the second one is that the pressure of the fluid  
is equal, at the free surface, to the constant atmospheric pressure 
 (dynamic boundary condition). 
Then, as shown in \cite{Zak1,CS}, 
the evolution of the fluid is determined by the 
following equations for the unknowns $ (\eta (t,x), \psi (t,x)) $,  
\begin{equation}\label{WWeq}
 \eta_t  = G(\eta)\psi \, , \quad 
  \psi_t  =  
- g \eta - \dfrac{\psi_x^2}{2} + \dfrac{1}{2(1+\eta_x^2)} \big( G(\eta) \psi + \eta_x \psi_x \big)^2 \, , 
\end{equation}
where $g > 0 $ is the gravity constant and $G(\eta)$ denotes 
the Dirichlet-Neumann operator $
 [G(\eta)\psi](x) := \Phi_y(x,\eta(x)) -  \Phi_x(x,\eta(x)) \eta _x(x)$. 
 With no loss of generality we set the gravity constant $g=1$.
The equations \eqref{WWeq} are the Hamiltonian system
\begin{equation}\label{HSsympl}
\pa_t \vet{\eta}{\psi} = \cJ \vet{\nabla_\eta \mathcal{H}}{\nabla_\psi \mathcal{H}}, \quad \quad \cJ:=\begin{bmatrix} 0 & \uno \\ -\uno & 0 \end{bmatrix},
\end{equation}
 where $ \nabla $ denote the $ L^2$-gradient, and the Hamiltonian
$  \mathcal{H}(\eta,\psi) :=  \frac12 \int_{\mathbb{T}} \left( \psi \,G(\eta)\psi +\eta^2 \right) \de x
$
is the sum of the kinetic and potential energy of the fluid. 
In addition of being Hamiltonian, the water waves 
system \eqref{WWeq} is time reversible with respect to the involution 
\begin{equation}\label{revrho}
\rho\vet{\eta(x)}{\psi(x)} := \vet{\eta(-x)}{-\psi(-x)}, \quad \text{i.e. } \ 
\mathcal{H} \circ \rho = \mathcal{H} \, .
\end{equation}
%Moreover, the equation \eqref{WWeq} 
% is invariant under space translation. \\[1mm]
{\bf Stokes waves, linearization and Bloch-Floquet expansion. }
The equations \eqref{WWeq} admit an analytic family of traveling periodic Stokes waves % of the form 
 $ (\eta_\e (t,x), \psi_\e(t,x)) = (\breve \eta_\e (x-c_\e t), \breve \psi_\e (x-c_\e t ))  $
 % for some real speed $c_\e = 1 + \frac12 \e^2 * O(\e^4) $  
 where   $(\breve \eta_ \e (x), \breve \psi_\e (x)) $ are  $2\pi$-periodic functions of the form 
$$ 
   \breve \eta_\e (x) =  \e \cos(x) + \cO(\e^2) \,,\quad \breve \psi_\e (x)  = \e \sin(x) +\cO(\e^2) \, ,\quad c_\e = 1+\cO(\e^2) \, .
$$
In a reference frame moving
with the speed $ c_\e $, 
the linearized  water waves equations at the Stokes waves 
  turn out to be\footnote{After conjugating with the ``good unknown of Alinhac" and the ``Levi-Civita" transformations, we refer to  \cite{BMV1,BMV2,BMV_ed} for details.}  the linear system 
$ h_t = \cL_\e h   $
where 
$\cL_\e:  H^1(\bT,\bR^2) 
\to L^2(\bT, \bR^2) $ 
is the Hamiltonian and reversible real operator
 \begin{equation}\label{cLepsilon} 
\cL_\e  =  \begingroup 
\setlength\arraycolsep{1pt}
\begin{bmatrix} \pa_x \circ (1+p_\e(x)) &  |D| \\ - (1+a_\e(x)) &   (1+p_\e(x))\pa_x \end{bmatrix} \endgroup = \cJ \begingroup 
\setlength\arraycolsep{1
pt}\begin{bmatrix}   1+a_\e(x) &   -(1+p_\e(x)) \pa_x \\ 
\pa_x \circ (1+p_\e(x)) &  |D|  \end{bmatrix} \endgroup
\end{equation} 
and
$p_\e (x), a_\e (x) $ are real  even analytic functions. 
We shall need their fourth order Taylor expansion which is 
\begin{subequations}\label{apexp}
\begin{align}
p_\e (x) & =  \sum_{n \geq 1} \e^n p_n (x) =  -2\e\cos(x) + \e^2 \Big( \frac32-2\cos(2x) \Big)  \label{pfunction} \\ \notag
&  \ \ + 3 \e^3 \big( \cos(x)-\cos(3x) \big)+ \e^4 \Big( \frac18 +4 \cos (2 x)-\frac{16}{3} \cos (4 x) \Big)+\cO(\e^5) \, ,  \\
a_\e (x) & =  \sum_{n \geq 1} \e^n a_n (x) = -2\e\cos(x)+2\e^2\big(1-\cos(2x)\big) \label{afunction}  \\ \notag
& \ \ +\e^3\big(4\cos(x)-3\cos(3x)\big)+\e^4\Big(-1+4 \cos (2 x)-\frac{16}{3} \cos (4 x)\Big)+\cO(\e^5) \, ,
\end{align}  
\end{subequations}
as shown taking the infinite depth limit  in \cite[(A.59)--(A.60)]{BMV_ed} (it follows
by the fourth-order expansion of the Stokes waves in \cite[(A.1)]{BMV_ed} or 
\cite[Proposition 2.2]{CNS}). 

By Bloch-Floquet theory, by introducing the Floquet exponent $\mu$, the spectrum
$$
\sigma_{L^2(\bR)} (\cL_\e ) = \bigcup_{\mu\in [- \frac12, \frac12)} \sigma_{L^2(\bT)}  (\cL_{\mu, \e})  
\quad \text{where} \quad 
\cL_{\mu,\e}:= e^{- \im \mu x} \, \cL_\e \, e^{\im \mu x}  \, ,  
$$
and, if $\lambda$ is an eigenvalue of $\cL_{\mu,\e}$ with a $ 2 \pi $-periodic 
eigenvector $v(x)$, then  $h (t,x) = e^{\lambda t} e^{\im \mu x} v(x)$ is a solution of 
$h_t = \cL_{\e} h$ whose growth in time is determined by Re$\,\lambda$.
\begin{rmk}\label{Brillouin}
Being
$ \sigma ({\mathcal L}_{-\mu,\e}) = \overline{  \sigma ({\mathcal L}_{\mu,\e}) } $ and 
$\sigma({\mathcal L}_{\mu,\e})$ a 1-periodic set with respect to $ \mu $, 
 we study the $ L^2(\bT) $-spectrum of  $  \cL_{\mu,\e} $  for $ \mu $ in the  ``first zone of Brillouin" 
$ 0 \leq \mu < \tfrac12 $.
\end{rmk}

The operator $ \cL_{\mu,\e}$ 
is  the complex  \emph{Hamiltonian} and \emph{reversible} operator
\begin{align}\label{WW}
 \cL_{\mu,\e} &:= \begin{bmatrix} (\pa_x+\im\mu)\circ (1+p_\e(x)) & 
 |D+\mu|  \\ -(1+a_\e(x)) & (1+p_\e(x))(\pa_x+\im \mu) \end{bmatrix} \\ 
 &= \underbrace{\begin{bmatrix} 0 & \uno\\ -\uno & 0 \end{bmatrix}}_{\displaystyle{=\cJ}} \underbrace{\begin{bmatrix} 1+a_\e(x) & -(1+p_\e(x))(\pa_x+\im \mu) \\ (\pa_x+\im\mu)\circ (1+p_\e(x)) & |D+\mu | 
 \end{bmatrix}}_{\displaystyle{=:\cB (\mu,\e) = \cB^* (\mu,\e) } } \, ,   \notag 
\end{align} 
that we regard  as an operator with 
domain $H^1(\bT):= H^1(\mathbb{T},\bC^2)$ and range $L^2(\bT):=L^2(\mathbb{T},\bC^2)$, 
equipped with  
the complex scalar product\footnote{The operator  $\cB^* (\mu,\e) $ in  \eqref{WW} 
is the adjoint with respect to the complex scalar product \eqref{scalar}.} 
\begin{equation}\label{scalar}
(f,g) := \frac{1}{2\pi} \int_\bT \left( f_1 \bar{g_1} + f_2 \bar{g_2} \right) \, \text{d} x  \, , 
\quad
\forall f= \vet{f_1}{f_2}, \ \  g= \vet{g_1}{g_2} \in  L^2(\bT, \bC^2) \, .
\end{equation} 
We also  denote $ \| f \|^2 = (f,f) $.  
The complex Hilbert space $ L^2 (\bT, \bC^2) $ is also equipped with the sesquilinear,  skew-Hermitian and non-degenerate  complex symplectic form  
\begin{equation}\label{ses}
{\cal W}_c  \, \colon L^2 (\bT, \bC^2) \times L^2 (\bT, \bC^2) \to \bC \, , \quad  {\cal W}_c(f, g) := (\cJ f,g) \, ,
\end{equation}
where $\cJ$ is defined in \eqref{HSsympl}. 
The complex operator $\cL_{\mu,\e}$ in  \eqref{WW}
is also reversible, namely  
\begin{equation}\label{revLB}
 \cL_{\mu,\e} \circ \bro =- \bro \circ \cL_{\mu,\e} \, , \quad \text{equivalently} \quad  
\cB (\mu,\e) \circ \bro = \bro \circ \cB (\mu,\e)   \, ,
\end{equation}
where $\bro$ is the complex involution (cfr. \eqref{revrho})
\begin{equation}\label{reversibilityappears}
 \bro \vet{\eta(x)}{\psi(x)} := \vet{\bar\eta(-x)}{-\bar\psi(-x)} \, .
\end{equation}
In addition $(\mu, \e) \to \cL_{\mu,\e} \in \cL(H^1(\bT), L^2(\bT))$ is analytic, 
%since  the functions $\e \mapsto a_\e$, $p_\e$  defined in \eqref{apexp} are analytic as maps $B(\e_0) \to H^1(\bT)$ 
and  ${\mathcal L}_{\mu,\e}$ is linear in $\mu$, being 
\begin{equation} \label{grazieStrauss}
|D+\mu| =  |D| + \mu \, 
\sgn^+(D) \,, \quad \forall \mu > 0 \,,
\end{equation}
where 
$ \sgn^+(j):= 1 $ if $ j\geq 0 $ and  $ \sgn^+(j):= - 1  $ for any  
$ j < 0 $. 
 
\smallskip

We aim to describe a {\em far-from-the-origin} spectral branching of eigenvalues of $  \cL_{\mu,\e}$  out of the imaginary axis.
The Hamiltonian structure of $\cL_{\mu,\e}$ allows such a branching to form only as perturbation of a \emph{multiple} purely imaginary eigenvalue
% implies that 
%  purely imaginary simple eigenvalues 
 of $\cL_{\mu,0}$.
 % remain on the imaginary axis under perturbation.
% We now carefully describe the spectrum of $\cL_{\mu,0}$.
\\[1mm]{\bf The spectrum of $\cL_{\mu,0}$.}\label{initialspectrum} 
The spectrum of the Fourier multiplier matrix operator 
\begin{equation}\label{cLmu}
 \cL_{\mu,0} = 
 \begin{bmatrix} 
\pa_x+\im\mu  & |D+\mu| \\ -1 & \pa_x+\im\mu \end{bmatrix} 
\end{equation}
on $L^2(\bT,\bC^2)$ is given by 
\begin{align}
& \lambda_j^\sigma(\mu)= \im \omega_j^\sigma (\mu) \, , \quad 
\omega_j^\sigma(\mu) := j+\mu -\sigma \sqrt{|j+\mu|} \, , 
\quad  j \in \bZ \, , \  \sigma=\pm \, , \ \text{and we write} \notag  \\
&  \omega_j^\sigma (\mu) = \omega^\sigma(j+\mu):= j+\mu -\sigma \Omega_j(\mu) \, ,   \quad  
 \Omega_j(\mu) := \Omega(j+\mu) :=  \sqrt{|j+\mu|}   \,  . 
 \label{omeghino}
\end{align}
 For any $ j + \mu \neq 0  $,  $ \sigma = \pm $, we associate to the  
 eigenvalue $\im\omega^\sigma_j(\mu) $ the eigenvector 
 \begin{equation}\label{def:fsigmaj}
 f^\sigma_j(\mu) := \frac{1}{\sqrt{2\Omega_j(\mu)}} \vet{-\sqrt{\sigma}\,\Omega_j(\mu)}{\sqrt{-\sigma}}e^{\im j x}\, , \qquad
 \cL_{\mu, 0} f^\sigma_j(\mu) = \im \omega^\sigma_j(\mu) f^\sigma_j(\mu) \ ,
\end{equation}
which  satisfies,
recalling \eqref{reversibilityappears}, the reversibility property   
\begin{equation}\label{baserev}
\bar \rho f_j^+ (\mu ) = f_j^+ (\mu )  \, , \quad 
\bar \rho f_j^- (\mu ) = - f_j^- (\mu ) \, . 
\end{equation}
For any $\mu \notin \bZ$ the family of  eigenvectors \eqref{def:fsigmaj} forms a {\it complex symplectic basis} of $L^2(\bT,\bC^2)$ with respect to the complex symplectic form ${\cal W}_c$ in \eqref{ses}, namely its elements are linearly independent, span densely $L^2(\bT,\bC^2) $ and satisfy, 
for any $ j \in \bZ $,  
\begin{equation}\label{sympbas}
{\cal W}_c \big( f_j^\sigma(\mu ), f_{j'}^{\sigma'}(\mu ) \big) = \begin{cases}  -\im & \text{if } j=j'\text{ and }\sigma=\sigma'=+ \, , \\ 
\ \; \im & \text{if } j=j'\text{ and }\sigma=\sigma'=- \, , \\
\ \; 0 & \text{otherwise} \, . \end{cases} 
\end{equation}
The choice of the normalization constant in \eqref{def:fsigmaj} 
implies \eqref{sympbas}. 

All the multiple nonzero eigenvalues of $\cL_{\mu,0} $  are given by the 
following lemma, cfr.
\cite{Nic}, where, 
in view of Remark \ref{Brillouin},  
we consider $ \mu \in [0,\frac12) $.

\begin{lem}[{\bf Multiple eigenvalues of $\cL_{\mu,0} $ away from $0$}]
\label{thm:unpert.coll} 
For any $\mu \in [0,\frac12)$, the spectrum of $\cL_{\mu,0}$ away from $0$ contains only simple or double eigenvalues:
\begin{enumerate}
    \item for any $\mu$ in $ \big(0,\tfrac14\big)\cup\big(\tfrac14,\tfrac12\big)$ the eigenvalues of $\cL_{\mu,0}$ are all simple;
    \item for $ \mu = \tfrac14 $  the double eigenvalues of $\cL_{\frac14,0}$ form the set $\big\{\im \omega_*^{(p)}\big\}_{\!\!\substack{p\geq 2\\ p \textup{ even}}} $ where
\begin{equation}\label{intcollision}
         \omega_*^{(p)} :=  \frac{p^2-1}{4}\, , \qquad p \in \bN ;
    \end{equation} 
    \item for $\mu = 0$ the double eigenvalues of $\cL_{0,0}$ form the set $\big\{\pm \im \omega_*^{(p)} \big\}_{\!\!\substack{p\geq 3\\ p \textup{ odd}}} $.
\end{enumerate}
Let $ p \in \bN$, $p\geq 2$.
The eigenspace associated with the double eigenvalue $\im \omega_*^{(p)} $ is spanned  by the eigenvectors  $f_k^-(\umu), \, f_{k'}^+(\umu) $ in \eqref{def:fsigmaj} where
\begin{equation}
 \begin{cases}
  \underline{\mu} = \frac14 \, , \quad k = n^2-n \, , \quad k'=k+p=n^2+n\,,  & \text{ if } p=2n \text{ is even} , \\
  \underline{\mu} = 0 \, , 
  \quad  k = n^2 \, , \quad k'=k+p=(n+1)^2\, ,
  & \text{ if } p=2n+1 \text{ is odd} .
 \end{cases}
\end{equation}
\end{lem}

In view of Remark \ref{Brillouin} the splitting of the double eigenvalue $ -\im \omega_*^{(p)}$ for $\umu=0$ can be obtained by complex conjugation.
This paper aims to study the splitting of the closest-to-zero double eigenvalue of
$  \cL_{\frac14,0} $.  
{\bf Hence, in Section \ref{section:conti}, in view of Lemma \ref{thm:unpert.coll}, we shall fix}  
\begin{equation}
\label{choicek} p=2\, , \quad k=0\, , \quad k'=2\, , \quad  \umu:= \tfrac14 \quad \text{and denote} \ 
\omega_*:= \omega_*^{(2)}= 
% \umu +\Omega_0(\umu ) = 
\tfrac34 \, .
\end{equation}
We shall consider Floquet parameters $ \mu $ 
close to $ \und{\mu} $, so that $ \mu \notin \bZ $ and 
the family of  eigenvectors \eqref{def:fsigmaj} forms a
complex symplectic basis according to \eqref{sympbas}. 

The spectrum   $\sigma (\cL_{\underline \mu ,0})  $ decomposes in two disjoint parts: 
\begin{equation}
\label{spettrodiviso0}
\sigma (\cL_{\underline{\mu},0}) = \sigma' (\cL_{\underline{\mu},0}) \cup \sigma'' (\cL_{\underline{\mu},0})
\quad \text{where} \quad  \sigma'(\cL_{\underline{\mu},0}):= \big\{\im\omega_*^{(p)} \big\}
\end{equation}
is the double eigenvalue  in  \eqref{choicek}  and 
$ \sigma''(\cL_{\underline{\mu},0}):= \big\{ \lambda_j^\sigma(\umu ),\ (j,\sigma) \notin \Sigma \big\} $, where   
$\Sigma := \big\{ (k',+),\,(k,-) \big\}$,  
collects the other eigenvalues  $\lambda_j^\sigma(\umu )$   in \eqref{omeghino}.

By Kato's perturbation theory (Lemma \ref{lem:Kato1})
for any $ (\mu, \e) $  sufficiently close to $ (\umu, 0) $, the perturbed spectrum
$\sigma\left(\cL_{\mu,\e}\right) $ admits a disjoint decomposition 
$
\sigma\left(\cL_{\mu,\e}\right) = \sigma'\left(\cL_{\mu,\e}\right) \cup \sigma''\left(\cL_{\mu,\e}\right) $
where $ \sigma'\left(\cL_{\mu,\e}\right) 
$  consists of 2 eigenvalues close to the double eigenvalue  $\im \omega_*^{(p)} $ 
 of $\cL_{\umu,0}$. 
We denote by $\cV_{\mu, \e}$   the spectral subspace associated to  $\sigma'\left(\cL_{\mu,\e}\right) $, which   has  dimension 2 and it is  invariant under $\cL_{\mu, \e}$.
The next result provides the expansion of the matrix representing the action of the operator $\cL_{\mu,\e} \colon \cV_{\mu, \e} \to \cV_{\mu, \e}$.   
We denote $B(r):= \{ y \in \bR \colon \  |y| < r\}$ 
the real interval of length $2r$ centered in 0.
\begin{teo}\label{thm:expansion}
Assume \eqref{choicek}.
There exist $\e_0, \delta_0 >0$ such that
% , for any $(\mu, \e) \in B_{\delta_0}(\tfrac14) \times B_{\e_0}(0)$, 
the operator $\cL_{\mu,\e} \colon \cV_{\mu, \e} \to \cV_{\mu, \e}$ for any  $(\mu,\e) = (\tfrac14+\delta, \e) \in B_{\delta_0}(\tfrac14) \times B_{\e_0}(0)$ is  represented %on a suitable basis 
by a $2\times 2$ 
 matrix with identical real off-diagonal entries and purely imaginary diagonal entries  of the form
\begin{equation} \label{final.mat}
{\footnotesize \begin{pmatrix} \im \frac34 + \im\frac23 \delta
-\im   \frac98 \e^2  
 +\im r_1(\e^3,\delta\e^2,\delta^2)\ \ \ \ \  - \frac{\sqrt{3}}{6} \delta\e^2  
- \frac{39 \sqrt{3}}{512} \e^4 + r_2 (\e^{5}, \delta\e^4, \delta^2\e^2, \delta^4\e)  \\[2mm]
\!\!\!- \frac{\sqrt{3}}{6} \delta\e^2  
- \frac{39 \sqrt{3}}{512} \e^4 + r_2 (\e^{5}, \delta\e^4, \delta^2\e^2, \delta^4\e)  \ \  \ \ \ \im\frac34 + \im 2\delta+   \frac{\im}{16}\e^2 +  
  \im r_3 (\e^3,\delta\e^2,\delta^2) 
\end{pmatrix}}\, .
\end{equation}
\end{teo}

Theorem  \ref{thm:main} follows directly from Theorem \ref{thm:expansion} 
%and  the description of the
%eigenvalues of the $ 2 \times 2 $ reversible and Hamiltonian %matrix $ \tL(\mu,\e)$ in \eqref{eigenvalues}. 
as shown in the beginning of Section \ref{section:conti}.
% below 
% Proposition \ref{expbT}.

\section{Perturbation of separated eigenvalues and instability criteria}\label{Katoapp}

We briefly recall  Kato's similarity transformation 
theory  as developed in \cite[Section 3]{BMV1}. 
The following result is proved as in
\cite[Lemmata 3.1, 3.2]{BMV1}  
with the only difference that concerns perturbations of 
a non zero eigenvalue $\im \omega_*^{(p)} $ of $\cL_{\umu,0} $. We remind that the operators $ \cL_{\mu,\e}  : Y \subset X \to X $   
has domain $Y:=H^1(\mathbb{T}):=H^1(\mathbb{T},\bC^2)$ and  range $X:=L^2(\mathbb{T}):=L^2(\mathbb{T},\bC^2)$. 

\begin{lem}\label{lem:Kato1} Fix $p\in\bN$, $p\geq 2$.
Let $\Gamma$ be a closed  counter-clocked wise oriented curve winding around
the double eigenvalue 
 $ 
 \im \omega_*^{(p)} $  of $ \cL_{\umu, 0} $ given by Lemma \ref{thm:unpert.coll}
 in the complex plane, separating $\sigma'
 (\cL_{\umu, 0}) $
  and the other part of the spectrum $\sigma'' (\cL_{\umu,0})$ in \eqref{spettrodiviso0}.
Then there exist $\e_0, \delta_0>0$  such that for any $(\mu, \e) \in B_{\delta_0}(\umu)\times B_{\e_0}(0)$  the following hold:
\\[1mm]
1. 
The curve $\Gamma$ belongs to the resolvent set of 
the operator $\cL_{\mu,\e} : Y \subset X \to X $ defined in \eqref{WW}. 
The operators
\begin{equation}\label{Pproj} 
% P({\mu,\e}) :=
  P (\mu,\e) :=  -\frac{1}{2\pi\im}\oint_\Gamma (\cL_{\mu,\e}-\lambda)^{-1} \de\lambda : X \to Y 
\end{equation}  
are  projectors commuting  with $\cL_{\mu,\e}$,  i.e. 
$ P({\mu,\e})^2 = P({\mu,\e}) $ and 
$ P({\mu,\e})\cL_{\mu,\e} = \cL_{\mu,\e} P({\mu,\e}) $.
The map $(\mu, \epsilon)\mapsto P({\mu,\e})$ is  analytic from 
$B_{\delta_0}(\umu)\times B_{\e_0}(0)$
 to $ \cL(X, Y)$.  The projectors $P({\mu,\e})$ 
are skew-Hamiltonian, namely $\cJ P({\mu,\e})=P({\mu,\e})^*\cJ $,  and reversibility preserving, i.e. 
$ \bro P({\mu,\e}) = P({\mu,\e})  \bro $.
 \\[1mm]
2. 
The domain $Y$  of  $\cL_{\mu,\e}$ decomposes as  the direct sum
$ Y= \mathcal{V}_{\mu,\e} \oplus \text{Ker}(P({\mu,\e})) $ 
of the  closed  subspaces $\mathcal{V}_{\mu,\e} 
:=\text{Rg}(P({\mu,\e})) $, $ \text{Ker}(P({\mu,\e})) $, 
which are invariant under $\cL_{\mu,\e}$, and  
$$
\sigma(\cL_{\mu,\e})\cap \{ z \in \bC \mbox{ inside } \Gamma \} = \sigma(\cL_{\mu,\e}\vert_{{\mathcal V}_{\mu,\e}} )  = \sigma'(\cL_{\mu, \e}) \, . 
$$
3.   The projector $P({\mu,\e})$ 
is conjugated to $P(\umu,0) $ through an operator $U(\mu,\e)$,  bounded and  invertible in $ Y $ and in $ X $, and
\begin{equation} \label{OperatorU} 
U({\mu,\e}) P(\umu,0)  =  P({\mu,\e}) U({\mu,\e})=
\big( \uno-(P({\mu,\e})-P (\umu,0))^2 \big)^{-1/2}
% \big[ 
P({\mu,\e})P (\umu,0) \, ,
% + (\uno - P({\mu,\e}))(\uno-P (\umu,0)) \big] 
\end{equation}
%with inverse
%\begin{equation}
% \label{Uinv}
%U({\mu,\e})^{-1}  = 
% \big[ 
%P (\umu,0) P({\mu,\e})+(\uno-P (\umu,0)) (\uno - P({\mu,\e})) \big] \big( \uno-(P({\mu,\e})-P (\umu,0))^2 \big)^{-1/2} \, , 
%\end{equation}
%  and such that
% $ U({\mu,\e}) P (\umu,0)U({\mu,\e})^{-1} =  P({\mu,\e})  $. 
The map $(\mu, \epsilon)\mapsto  U({\mu,\e})$ is analytic from  $B_{\delta_0}(\umu)\times B_{\e_0}(0)$ to $\cL(Y)$. 
The transformation operator $U({\mu,\e})$  is symplectic, i.e.
 $ U({\mu,\e})^* \cJ U({\mu,\e})= \cJ $, and reversibility preserving.
One has
$
\mathcal{V}_{\mu,\e}=  U({\mu,\e})\mathcal{V}_{\umu,0}
$ and $\dim \mathcal{V}_{\mu,\e} = \dim \mathcal{V}_{\umu,0}= 2 $.
\end{lem}

\begin{rmk}
The proof that $ P_{\mu,\e} $ is skew-Hamiltonian and reversibility preserving holds as in 
\cite[Lemma 3.1]{BMV1}  
choosing $ \gamma (t) = \im \omega_*^{(p)} + r e^{\im t}$ so that  $ - \bar \gamma (t) $ winds around 
$ \im \omega_*^{(p)} $ clockwise.
\end{rmk}

 We consider the basis   
\begin{equation}\label{basisF}
{\cal F} := \{ 
f^+(\mu,\e),   f^- (\mu,\e)\} \, , \quad 
f^+(\mu,\e) := U({\mu,\e}) f_{k'}^+ , 
\ \ 
f^-(\mu,\e) := U({\mu,\e}) f_{k}^- \, , 
\end{equation}
of the subspace $ {\mathcal V}_{\mu,\e} $,  
obtained applying  the transformation operators $ U({\mu,\e}) $ of Lemma \ref{lem:Kato1}
to the eigenvectors  $f_{k'}^+:=f_{k'}^+(\umu )$, $f_{k}^-:=f_{k}^-(\umu )$  in \eqref{def:fsigmaj}
of $ \cL_{\umu,0}$, 
which, by Lemma  \ref{thm:unpert.coll}, form a basis of the eigenspace
$ \mathcal{V}_{\underline \mu,0}  $ associated with $\im \omega_*^{(p)}\!\!\! $, for any fixed integer $p\geq 2$.

\begin{lem} {\bf (Matrix representation of $ \cL_{\mu,\e}$ on $ \mathcal{V}_{\mu,\e}$)}
Fix $p\in \bN $, $p\geq 2$. The operator $\cL_{\mu,\e}: \mathcal{V}_{\mu,\e}\to\mathcal{V}_{\mu,\e} $ in 
\eqref{WW} 
% defined for any  $(\mu, \e) \in B_{\delta_0}(\umu)\times B_{\e_0}(0)$
 is  represented on the basis ${\cal F}$ in \eqref{basisF} 
 by the $2\times 2$ Hamiltonian and reversible matrix
 $$
 \tL(\mu,\e)=\tJ\tB(\mu,\e), \quad  \tJ :=  
 \begin{pmatrix} 
-\im & 0  \\
0  & \im 
\end{pmatrix} \, , \qquad \forall (\mu, \e) \in B_{\delta_0}(\umu)\times B_{\e_0}(0)\, ,
 $$
 where
\begin{equation} \label{tocomputematrix}
 \tB(\mu,\e)  := \begin{pmatrix}
( \mathfrak{B}({\mu,\e}) f^+_{k'}, f^+_{k'}) 
& 
(\mathfrak{B}({\mu,\e})  f^-_k, f^+_{k'}) \\ 
(\mathfrak{B}({\mu,\e}) f^+_{k'}, f^-_k) 
&
(\mathfrak{B}({\mu,\e}) f^-_k, f^-_k)
\end{pmatrix}  
:= \begin{pmatrix} \alpha(\mu,\e) & \im \beta(\mu,\e) \\ -\im \beta(\mu,\e) & \gamma(\mu,\e) \end{pmatrix}  \, ,
\end{equation}
the functions  $ \alpha(\mu,\e),  \beta (\mu,\e), \gamma (\mu,\e) $ are real analytic, and  
\begin{equation}\label{Bgotico}
\mathfrak{B} ({\mu,\e}) := P({\umu,0})^* \, U({\mu,\e})^* \, {\cal B}({\mu,\e}) \, U({\mu,\e}) \, P({\umu,0})
\, . \end{equation}  
Furthermore 
\begin{equation}\label{bmuzero} 
\tB(\mu,0) = 
\begin{pmatrix}  \alpha (\mu,0)  & \im \beta (\mu,0)  \\ 
- \im \beta (\mu,0) &  \gamma (\mu,0)  \end{pmatrix} = 
\begin{pmatrix} -\omega_{k'}^+(\mu)  & 0 \\ 0 &  \omega_k^-(\mu)  \end{pmatrix}  
\end{equation} 
with $\omega_j^\sigma (\mu) $ in \eqref{omeghino}, in particular
$ \tB(\umu,0)={\footnotesize \begingroup 
\setlength\arraycolsep{2pt}\begin{pmatrix} -\omega_*^{(p)}  & 0 \\ 0 &  \omega_*^{(p)}\end{pmatrix}\endgroup} $. 
\end{lem}

\begin{proof}
In view of \eqref{sympbas} and \eqref{baserev},
the basis $ \{ f^+_{k'}, f^-_{k}\} $  of $ \mathcal{V}_{\underline \mu,0} $ 
is complex  symplectic and reversible
and, since $U({\mu,\e})$ is symplectic and  reversibility preserving,  
the 
  basis $  {\cal F}  $ in \eqref{basisF} is a complex symplectic and reversible 
  basis of $\mathcal{V}_{\mu,\e} $. 
Given a complex symplectic basis  $\{\tf^+,\tf^-\} $, any vector $\tf $ in $ \langle \mathtt{f}^+, \mathtt{f}^- \rangle $ verifies
$
\tf = \im  \molt{ \cJ  \tf}{\tf^+} \tf^+  
 - \im \molt{\cJ  \tf}{\tf^-} \tf^- $ whence we obtain 
 \eqref{tocomputematrix}-\eqref{Bgotico}. The function $ \beta$ is real by the reversibility property  
\eqref{revLB} and \eqref{baserev}. 

Let us prove \eqref{bmuzero}. 
The operator  $\kB ({\mu,0}) $ in \eqref{Bgotico} is a Fourier multiplier, since $\cL_{\mu,0}$ in \eqref{WW} is a Fourier multiplier, and so is $\cB (\mu,0)$ in  \eqref{WW}, $P(\mu,0)$ in \eqref{Pproj}, and finally $U(\mu,0)$ in \eqref{OperatorU}. 
As a consequence
$\beta(\mu,0) = \left( \kB ({\mu,0}) f_k^-, f_{k'}^+ \right)  = 0 $.
Then we exploit that 
 $\cL_{\mu,0} $ has eigenvectors 
$$
\cL_{\mu,0} f_{k'}^\pm(\mu) = \im \omega_{k'}^\pm(\mu) f_{2}^\pm(\mu)\, ,\quad \cL_{\mu,0} f_{k}^\pm(\mu) = \im \omega_{k}^\pm(\mu) f_{k}^\pm(\mu) \, , 
$$
in \eqref{def:fsigmaj}
and $U(\mu,0)f_{k'}^+ $ and $U(\mu,0)f_k^- $ are also eigenvectors of $\cL_{\mu,0} $ which, by continuity, are multiples respectively of $f_{k'}^+(\mu) $ and $f_k^-(\mu) $. 
\end{proof}

The eigenvalues of the matrix $\tL(\mu,\e)$ 
in \eqref{tocomputematrix} are 
\begin{equation}\label{eigenvalues}
\lambda^\pm (\mu,\e) = \tfrac\im2 S(\mu,\e)\pm\tfrac12\sqrt{
D(\mu,\e) 
}
\end{equation}
where
 \begin{align}\label{traceB}
 S(\mu,\e)  & :=  \gamma(\mu,\e) - \alpha(\mu,\e)\, ,  \\  \label{discriminant}
 D(\mu,\e)  & := 4 \beta^2 (\mu,\e) - T^2 (\mu,\e)  \, , \quad
 T(\mu,\e):= \alpha(\mu,\e) + \gamma(\mu,\e)  \, .
 \end{align}
The goal of the next sections is to prove, for $(\mu,\e) $ close to  $(\umu,0) $,
the existence of  eigenvalues of the matrix $\tL(\mu,\e)$ with non zero real part.
We now formulate abstract instability criteria to completely describe 
the spectrum of a  $ 2 \times 2 $  
 matrix of the form \eqref{tocomputematrix}.

 \paragraph{Abstract instability criteria.}\label{sec:mr}
We consider a  $ 2 \times 2 $  Hamiltonian and reversible 
 matrix 
 \begin{equation} \label{tocomputematrix1}
\tL(\mu,\e)=\tJ\tB(\mu,\e), \quad  
\tB(\mu,\e) = \begin{pmatrix} \alpha(\mu,\e) & \im \beta(\mu,\e) \\ -\im \beta(\mu,\e) & \gamma(\mu,\e) \end{pmatrix}   \, , 
\end{equation}
 where  $\alpha(\mu,\e)$, $\beta(\mu,\e)$, $\gamma(\mu,\e)$ are 
 real analytic functions 
defined 
in a neighborhood $ B_{\delta_0}(\umu)\times B_{\e_0}(0) $ of $(\umu,0)$, $\umu \in \bR$. 
%Later we will set  $\umu = \umu$, but in view of  future applications we 
% Here $ \umu $ is generic value (in  the application to Theorem \ref{thm:expansion} we shall take $\umu = \frac14$). 
We make the following 
\begin{ass}\label{H}  The real analytic entries $\alpha(\mu,\e),\beta(\mu,\e),\gamma(\mu,\e) $ of the matrix \eqref{tocomputematrix1} admit  
for  $(\mu,\e) = (\umu+\delta,\e)\in  B_{\delta_0}(\umu)\times B_{\e_0}(0)$ 
an  expansion of the form 
 \begin{subequations}\label{matrixentries}
\begin{align}\label{expa}
 \alpha(\umu+\delta,\e) & = -\alpha_0(\delta)
 + \alpha_2\e^2  +  r(\e^3,\delta\e^2,\delta^2\e)\, ,\\
\label{expb}
  \beta(\umu+\delta,\e) & = \beta_1 \e^2+\beta_2 \delta\e^2 +  \beta_3 \e^4 + r (\e^{5},\delta\e^3,\delta^2\e^2,\delta^3\e)\, , \\
    \label{expc}
  \gamma(\umu+\delta,\e) & = \gamma_0(\delta)  + \gamma_2 \e^2  + r(\e^3,\delta\e^2,\delta^2\e) \, , 
\end{align}
\end{subequations}
where 
\begin{equation}\label{epsilonspento}
\alpha_0(\delta)  = \omega_*^{(p)} - \alpha_1 \delta + r(\delta^2)\, , \qquad 
\gamma_0(\delta)= \omega_*^{(p)}  + \gamma_1 \delta  + r(\delta^2)\, ,
\end{equation}
and 
\begin{equation}\label{alpha1gamma1}
\alpha_1 + \gamma_1 > 0  \ .
\end{equation}
\end{ass}
In view of  Assumption \ref{H} the trace   
 $ T(\mu,\e):= \text{Tr}\,\tB(\mu,\e) = \alpha(\mu,\e) + \gamma(\mu,\e) $
  expands as 
 \begin{equation}\label{expT}
 \begin{aligned}
T(\umu + \delta,\e)  & = T(\umu + \delta,0) +T_2\e^2 + r(\e^3,\delta\e^2)  \ ,  \qquad \ T_2 :=  \alpha_2+\gamma_2  \, ,  \\
T(\umu + \delta,0)&  = T_1 \delta + r( \delta^2) \ , \qquad
T_1 := \alpha_1 + \gamma_1  >0  \  .
\end{aligned}
\end{equation}
By \eqref{expT} and the analytic implicit function theorem there exists $\e_1 \in (0, \e_0) $ such that
the zero set of the trace $T(\mu,\e)$ 
and of the functions 
  \begin{equation}\label{def-discri}
 d_\pm  (\mu,\e) :=  T (\mu,\e) \pm 2\beta (\mu,\e) 
 \end{equation}
 are, in the set 
 $ B_{\delta_0}(\umu)\times  B_{\e_1 }(0)$, 
graphs of analytic functions 
\begin{equation}\label{impfunc}
\mu_0 \, , \mu_\pm : 
(-\e_1 ,\e_1 ) \mapsto \mu_0 (\e) \, , \ \mu_\pm  (\e) \, ,\ \text{i.e.}\ 
 T (\mu_0 (\e),\e) \equiv 0\ \text{ and }\ d_\pm  (\mu_\pm (\e),\e  ) \equiv 0 \, , 
 \end{equation}
satisfying, by \eqref{expT}, 
 the Taylor expansions
\begin{equation}\label{muexpa}
\begin{aligned}
& \mu_0 (\e) = \umu - \frac{T_2}{T_1} \e^2 + r(\e^3) \, , 
& \mu_\pm (\e) = \umu - \frac{T_2\pm 2 \beta_1}{T_1} \e^2 + r_\pm(\e^3) \, . 
\end{aligned}
\end{equation}
Since $T_1>0 $, the functions $d_\pm(\mu,\e) $ are strictly positive (resp. negative) for $\mu>\mu_\pm(\e)$ (resp.  $\mu<\mu_\pm(\e)$). In addition, since $d_\pm(\mu_0(\e),\e ) =\pm 2\beta(\mu_0(\e),\e)   $  we deduce that for any $\e \in (-\e_1,\e_1) $
\begin{equation}\label{maxalternative}
\begin{cases}
\mbox{if }\beta(\mu_0(\e),\e) >0 & \mbox{then }\ \mu_+(\e) <\mu_0(\e) < \mu_-(\e)\,, \\
\mbox{if }\beta(\mu_0(\e),\e) <0 & \mbox{then }\ \mu_-(\e) <\mu_0(\e) < \mu_+(\e)\,, \\
\mbox{if }\beta(\mu_0(\e),\e) =0 & \mbox{then }\ \mu_0(\e)= \mu_+(\e) = \mu_-(\e)\,.
\end{cases}
\end{equation}
The graphs of these functions look like  in figure \ref{fig:instareg}. 
Thus $\mu_\wedge(\e) \leq \mu_0(\e) \leq \mu_\vee(\e) $ where
\begin{equation}\label{muminmumax}
\mu_\wedge (\e):= \min\big\{\mu_+ (\e),\mu_- (\e)\big\} \quad \text{and}\quad \mu_\vee (\e):= \max\big\{\mu_+ (\e),\mu_- (\e)\big\}\, ,
\end{equation}
and the inequalities are strict if and only if $\beta(\mu_0(\e),\e)\neq 0 $.
 \begin{figure}[h!!]
\centering
\includegraphics[width=8cm]{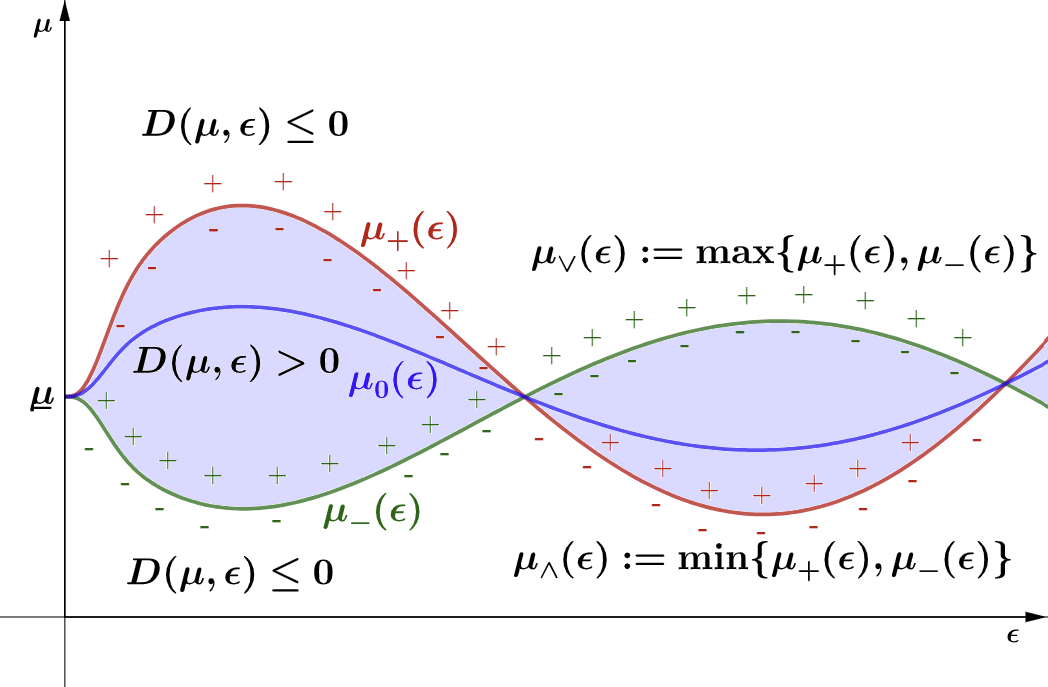}
\caption{ \label{fig:instareg} The instability region around the curve $\mu_0(\e)$  delimited by  the curves $\mu_\vee(\e)$ and $\mu_\wedge(\e)$.}
\end{figure}

 The following result provides a {\it necessary and sufficient} criterion for the existence
of eigenvalues of the matrix $ \tL (\mu,\e) $ 
with non-zero real part.

\begin{teo}\label{instacrit}
{\bf (Criterion of instability)}
Assume that the  $ 2 \times 2 $ Hamiltonian and reversible 
 matrix $\tL(\mu,\e)= \tJ \tB(\mu,\e)$ in \eqref{tocomputematrix1} 
 satisfies Assumption \ref{H}.  
 Then 
 \par \textrm{i)} the matrix $\tL(\mu,\e)$ has eigenvalues with nonzero real part, for $ (\mu,\e) $ close to
 $ (\umu,0) $,  if and only if the analytic function  
\begin{equation}\label{nonzeroan}
 (-\e_1,\e_1)\ni\e \mapsto \beta(\mu_0(\e),\e) \not\equiv 0 
\end{equation}
 is not identically zero, where $ \mu_0 (\e) $ is the analytic function
 defined  in \eqref{impfunc};  
equivalently 
\begin{equation}\label{wholeinstacrit}
 \exists \, n \in \bN \quad \text{ such that }\quad\frac{\;\;\de^n}{\de \e^n }\; \beta(\mu_0(\e),\e)_{| \e = 0 } \neq 0 \, . 
\end{equation}
\par \textrm{ii)} The spectrum of $ \, \tL(\mu,\e) $, for $(\mu,\e) $ close to $ ( \umu, 0 ) $, consists of two  eigenvalues $\lambda^\pm(\mu,\e) $ with opposite non-zero real part 
if and only if $(\mu,\e)$ lies inside 
the  region (see Figure \ref{fig:instareg}) 
\begin{equation}\label{regionR}
\begin{aligned}
R:&=\big\{ (\mu, \e) \in B_{\delta_0}(\umu)\times B_{\e_1}(0)  \ : \  
D(\mu,\e)=-d_+ (\mu,\e) \, d_- (\mu,\e)  > 0 \ \big\} \\ &= \big\{ 
(\mu, \e) \in B_{\delta_0}(\umu)\times B_{\e_1}(0)  \ : \    \mu_\wedge(\e)<\mu< \mu_\vee(\e)\big\} \, , \end{aligned}
\end{equation}
 whereas, for $(\mu,\e)\notin R $, the eigenvalues $\lambda^\pm(\mu,\e)$ are purely imaginary.
 \end{teo}
 \begin{proof}
The eigenvalues of % the Hamiltonian and  reversible matrix 
$ \tL(\mu,\e )$ % in \eqref{tocomputematrix} 
have the form  \eqref{eigenvalues}
and have nonzero real part if and only if  the discriminant 
$  D(\mu,\e) = 4 \beta^2 (\mu,\e) - T^2 (\mu,\e) = - d_+(\mu,\e)d_-(\mu,\e)  $
is positive. 
In view of \eqref{maxalternative} and \eqref{muminmumax}, we deduce item $ii$.
The region $R$ is not empty if and only if condition \eqref{nonzeroan} holds. This proves item $i$.
 \end{proof}

Let us now show sufficient conditions  to verify  \eqref{nonzeroan}.
In view of the expansions \eqref{expb} and \eqref{muexpa} we have 
\begin{equation}\label{beta(mu,e)}
\beta(\mu_0(\e),\e) = \beta_1 \e^2 + \Big(\beta_3 - \beta_2 \frac{T_2}{T_1} \Big)\e^4 + r(\e^5) \ .
\end{equation}
Hence the following conditions 
imply  the instability criterion \eqref{wholeinstacrit} :  
\begin{equation}\label{cases}
\begin{aligned}
&\text{1. {\bf (Non-degenerate case)}} \ \  \frac12 \frac{\;\de^2}{\de \e^2}\; \beta(\mu_0(\e),\e)_{| \e = 0 }  = \beta_1\neq 0\, ;\\
&\text{2. {\bf (First degenerate case)}} \ \ \beta_1=0\text{ and }
\frac{1}{4!} \frac{\;\de^4}{\de \e^4}\; \beta(\mu_0(\e),\e)_{| \e = 0 } = \beta_3-\frac{T_2}{T_1}\beta_2 \neq 0\, . 
\end{aligned}
\end{equation}
We shall prove in Section \ref{section:conti} that, for $p=2$ and considering the  deep water case,    
the coefficient $   \beta_1 $
in the expansion \eqref{matrixentries} of the matrix \eqref{tocomputematrix} 
vanishes, but  $ \beta_3-\frac{T_2}{T_1}\beta_2 \neq 0$, hence we are in the setup of the first degenerate case.  

In order to carefully describe  the  unstable eigenvalues, it is  convenient to 
translate the Floquet exponent 
$ \mu $ around the value $ \mu_0 (\e) $ where the  $ T(\mu_0(\e) ,\e)=0 $ vanishes, cfr. \eqref{impfunc},   namely 
we introduce the new parameter $\nu$  such that
\begin{equation}\label{def:nu}
\mu = \mu_0(\e) + \nu\, , \quad\text{i.e.} \quad \nu := \delta+\umu-\mu_0(\e)\, .
\end{equation}
Accordingly we write the functions   $\mu_\pm(\e)$ in \eqref{muexpa} as
$\mu_{\pm} (\e) = \mu_0(\e) + \nu_\pm(\e) $ 
where
\begin{equation}\label{expnuendsbis}
\nu_\pm(\e) := \mu_\pm(\e)-\mu_0(\e) \stackrel{\eqref{muexpa}}{=}  \mp  \frac{2 \beta_1}{T_1} \e^2 + r(\e^3) \, . 
\end{equation}
Along the proof of 
Theorem \ref{degenerateisola} we need the following expansion. 
\begin{lem}
If $\beta_1=0$, the functions $\nu_\pm(\e) $ in \eqref{expnuendsbis} admit the  expansion
\begin{equation}\label{expnuends}
\nu_\pm(\e) =  \mp \frac{2}{T_1} \left(\beta_3- \beta_2 \frac{T_2}{T_1} \right)\e^4 
+ r(\e^5) \, .
\end{equation}
\end{lem}
\begin{proof}
The function $\nu_+(\e)$ solves 
$$
d_+ ( \mu_0(\e) + \nu_+(\e),\e ) = 
T\big(\mu_0(\e) + \nu_+(\e),\e \big) + 2\beta\big(\mu_0(\e)+\nu_+(\e),\e\big) = 0\, .
$$
Expanding this identity at $\mu= \mu_0(\e)$ we have
$$
\begin{aligned}
&\underbrace{T\big(\mu_0(\e) ,\e \big)}_{=0\textup{ by \eqref{impfunc}}}+ \underbrace{(\pa_\mu T)\big(\mu_0(\e) ,\e \big) }_{= T_1 + r(\e^2)\textup{ by \eqref{expT} and  \eqref{muexpa}}}\nu_+(\e) +\underbrace{ r\big( \nu_+^2(\e)\big)}_{=r(\e^6) \textup{ by \eqref{expnuendsbis} with }\beta_1=0}\\
&+ \underbrace{2\beta\big(\mu_0(\e) ,\e \big)}_{=2 \big(\beta_3 - \beta_2 \frac{T_2}{T_1} \big)\e^4 + r(\e^5)  \textup{ by \eqref{beta(mu,e)}}}+ 2 \underbrace{(\pa_\mu \beta)\big(\mu_0(\e) ,\e \big)}_{= r(\e^2)  \textup{ by \eqref{expb}}} \nu_+(\e) = 0\, ,
\end{aligned}
$$
which gives \eqref{expnuends}.
Analogously one obtains the expansion of $\nu_-(\e)$.
\end{proof}
We define  
\begin{equation}\label{numpiu}
\nu_\wedge(\e) := \min\{\nu_+(\e),\,\nu_-(\e)\} \leq 0\, , \quad 
\nu_\vee(\e) := \max\{\nu_+(\e),\,\nu_-(\e)\} \geq 0 \, .
\end{equation}
Note that $\nu_\wedge(\e)$ and $\nu_\vee(\e) $ are the points where the discriminant $D(\mu_0(\e)+\nu,\e) $ in \eqref{discriminant} of the matrix $\tL(\mu_0(\e)+\nu,\e) $ vanishes and for $\nu_\wedge(\e) <\nu <\nu_\vee(\e)$ the discriminant $D(\mu_0(\e)+\nu,\e) $ is positive. 
We now describe the first degenerate case.
 \begin{teo}[\bf First degenerate case] \label{degenerateisola}
Assume \eqref{expa}-\eqref{expc}  and 
\begin{equation}\label{degeneratecondition}
\beta_1 = 0 \, ,\qquad \beta_3-\beta_2\frac{T_2}{T_1} \neq 0 \, ,
\end{equation}
where $T_1$ and $T_2$ are defined  in \eqref{expT}. 
 Then 
the matrix $\tL(\mu,\e) $  in \eqref{tocomputematrix1} possesses two  unstable
eigenvalues $\lambda^\pm(\mu_0(\e)+\nu,\e)$ for any  $\nu_\wedge(\e) <\nu < \nu_\vee(\e) $ where $ \nu_\wedge(\e)$, $  \nu_\vee(\e) $ are defined in
\eqref{numpiu}. 
The eigenvalues are 
\begin{equation}\label{eigs}
\! \lambda^\pm(\mu_0(\e)+\nu,\e) = \begin{cases}   \frac\im2 S(\mu_0(\e)+\nu,\e) \pm  \frac\im2 \sqrt{ |D(\mu_0(\e)+\nu,\e)|} & \text{if }  \nu \leq \nu_\wedge(\e)\textup{ or }\nu \geq \nu_\vee(\e)\, , \\[1mm]
  \frac\im2 S(\mu_0(\e)+\nu,\e) \pm \frac12 \sqrt{ D(\mu_0(\e)+\nu,\e)}&
   \text{if }\nu_\wedge(\e) <\nu <\nu_\vee(\e)\, , 
 \end{cases} 
\end{equation}
where $ \mu_0 (\e) $ is defined in \eqref{impfunc} and it has
 the form \eqref{muexpa},  
\begin{equation}\label{degenerateexpD}
\begin{aligned}  & D(\mu_0(\e)+\nu,\e) = \\ & 4\Big(\beta_3 - \beta_2 \frac{T_2}{T_1} \Big)^2 \e^8 -T_1^2 \nu^2+ 8 \beta_2 \Big(\beta_3 - \beta_2 \frac{T_2}{T_1}  \Big)\nu\e^6 +r(\e^{9},\nu\e^7,\nu^2\e^2,\nu^3)\, , 
\end{aligned}
\end{equation}
and 
\begin{equation}
\begin{aligned}
& S(\mu_0(\e)+\nu,\e) = \\ 
& 2\omega_*^{(p)} + (\gamma_1-\alpha_1)\nu +  \Big( \gamma_2-\alpha_2 - (\gamma_1-\alpha_1) \frac{T_2}{T_1} \Big) \e^2+ r(\e^3, \nu\e^2, \nu^2)  \, . \label{degenerateexpS}
\end{aligned}
\end{equation}
The maximum absolute value  of the real  part 
of the unstable
 eigenvalues in \eqref{eigs} is 
\begin{equation} \label{degenerateexpnuRe}
\max \text{Re}\, \lambda^\pm(\mu_0(\e)+\nu,\e) = \Big| \beta_3-\beta_2\frac{T_2}{T_1} \Big| \e^4 (1 + r(\e))\, .
\end{equation}
\par {\bf \emph{(Isola)}.} 
Assume in addition that the coefficients in \eqref{expa}-\eqref{expc}
satisfy $\alpha_1\neq \gamma_1$. %and  $\beta_2\neq 0$.
 Then, for any $ \e$ small enough, the pair of unstable eigenvalues $\lambda^\pm(\mu_0(\e)+\nu,\e)$ depicts in the complex $\lambda$-plane, as $\nu $ varies
in the interval $ (\nu_\wedge(\e), \nu_\vee(\e)) $
a closed analytic  curve which
  intersects orthogonally the imaginary axis 
  and encircles a convex region. 
 \end{teo}
 \begin{proof} {\bf (Unstable eigenvalues).}
 The criterion of instability in Theorem \ref{instacrit} is satisfied in view of \eqref{beta(mu,e)} and \eqref{degeneratecondition}.
 By \eqref{eigenvalues}, \eqref{matrixentries} and \eqref{def:nu}, the eigenvalues of $ \tL(\mu,\e) $ 
 % in view of \eqref{eigenvalues}, 
 have the form \eqref{eigs}.    We now prove the expansion  \eqref{degenerateexpD} of the discriminant 
\begin{equation}\label{Dinte}
D(\mu_0(\e)+\nu,\e)= 4\beta^2(\mu_0(\e)+\nu,\e) - T^2(\mu_0(\e)+\nu,\e)\, .
%\big(2\beta(\mu_0(\e)+\nu,\e) - T(\mu_0(\e)+\nu,\e)\big)\big(2\beta(\mu_0(\e)+\nu,\e) - T(\mu_0(\e)+\nu,\e)\big) \, .
\end{equation} 
By \eqref{impfunc} and \eqref{expT}  we  get that 
\begin{equation}\label{Tnuexp}
T(\mu_0(\e)+\nu,\e) = \partial_\mu T (\mu_0(\e),\e)\nu + r(\nu^2) = T_1\nu + r(\nu\e^2,\nu^2) \, .
\end{equation}
By  \eqref{expb} with 
 $ \beta_1 = 0 $,  \eqref{muexpa}, \eqref{def:nu} and \eqref{degeneratecondition},
 \begin{equation}\label{bnuexp}
\beta(\mu_0(\e)+\nu,\e) =  \Big(\beta_3- \beta_2 \frac{T_2}{T_1} \Big) \e^4 + \beta_2 \nu \e^2 + r(\e^5,\nu\e^3,\nu^2\e^2,\nu^3\e)\, .
\end{equation}
Then the expansion  \eqref{degenerateexpD}  follows by \eqref{Dinte} and taking the square of 
\eqref{Tnuexp} and \eqref{bnuexp}.
The expansion of $S(\mu_0(\e)+\nu,\e) $ in \eqref{degenerateexpS} follows from \eqref{traceB}, \eqref{expa}-\eqref{expc}, \eqref{expT} and \eqref{muexpa}.
 The absolute-value maximum of the real parts of the eigenvalues is attained at $\nu=\nu_{\text{Re}}$, with $\nu_{\text{Re}}$ such that $
 (\partial_\mu D) (\mu_0(\e)+\nu_{\text{Re}},\e) = 0$.
By \eqref{degenerateexpD} we have the expansion
\begin{equation}\label{numax}
\nu_{\text{Re}}(\e)  = 4\frac{\beta_2}{T_1^2} \Big(\beta_3-\beta_2 \frac{T_2}{T_1} \Big) \e^6 +r(\e^{7})\,.
\end{equation}
By plugging \eqref{numax} into \eqref{eigs}-\eqref{degenerateexpD} one obtains expansion
 \eqref{degenerateexpnuRe}.

{\bf (Isola).} In view of \eqref{eigs} 
for any fixed  $ \e $ 
small enough the unstable eigenvalues branch off from the imaginary axis at $\nu=\nu_\wedge(\e)$, evolve specularly  as $\nu$ increases and rejoin at $\nu=\nu_\vee(\e)$ thus forming a {\it closed} curve. With the  hypothesis $\alpha_1\neq \gamma_1$ 
the imaginary part of the eigenvalues
$ I(\nu,\e):=  \text{Im}\,\lambda^\pm(\mu_0(\e)+\nu,\e)  = \tfrac12 S\big(\mu_0(\e)+\nu,\e\big) $
 is monotone w.r.t. $\nu\in \big(\nu_\wedge(\e),\nu_\vee(\e)\big) $ because its derivative fulfills
 \begin{equation}
 \label{derivatanonvanishing}
 \pa_\nu I(\nu,\e)  \stackrel{\eqref{degenerateexpS}}{=} \frac{\gamma_1-\alpha_1}{2} + r(\e^2,\nu) \neq 0 \, , \qquad \nu_\wedge(\e) \leq \nu \leq \nu_{\vee}(\e)\, .
 \end{equation}
Thus the map $\nu \mapsto I(\nu,\e)  $ is a diffeomorphism between $\big(\nu_\wedge(\e),\nu_\vee(\e)\big)$ and its image $\big(y_\wedge(\e),y_\vee(\e)\big)$. %where we set 
%$$
%y_\wedge(\e):=\min \big\{ %I(\nu_\wedge(\e),\e)\, %,I(\nu_\vee(\e),\e)\big\}\, , \quad  %y_\vee(\e):=\max \big\{ %I(\nu_\wedge(\e),\e)\, %,I(\nu_\vee(\e),\e)\big\}\, .
%$$
Let us denote by $\nu(y,\e)$  the inverse of $y=I(\nu,\e)$, with $y$ varying in $y_\wedge(\e) < y < y_\vee(\e)$. 
The curves covered by  the two unstable eigenvalues in \eqref{eigs} in the complex plane are the two specular graphs on the imaginary axis
\begin{equation}\label{2graphs}
\begin{aligned}
& \Gamma_r := \big\{ \big( X(y,\e), y\big) \;:\; y_\wedge(\e) <y<y_\vee(\e) \big\}\, , \\
& \Gamma_l := \big\{ \big( -X(y,\e), y\big) \;:\; y_\wedge(\e) <y<y_\vee(\e) \big\}\, ,
\end{aligned}
\end{equation}
where 
\begin{equation}\label{Xdepy}
X(y,\e) := \tfrac12 \sqrt{D\big(\mu(y,\e), \e \big)}\, ,\quad \mu(y,\e):= \mu_0(\e) + \nu(y,\e)\, .
\end{equation}
At the bottom and top of the isola, i.e. at $y=y_\wedge(\e)$ and $y=y_\vee(\e)$, the real parts
 $\pm X(y,\e) $
  of the unstable eigenvalues vanish with derivative that tends to infinity. Indeed
  \begin{equation}\label{upblowing}
 \pa_y X(y,\e) \stackrel{\eqref{Xdepy}}{=}  \pa_y \tfrac12 \sqrt{D(\mu(y,\e),\e)}  =  \dfrac{ (\pa_\mu D)(\mu(y,\e),\e)}{4\sqrt{D(\mu(y,\e),\e)}} (\pa_y \mu)(y,\e) \, ,
 \end{equation}
and, by \eqref{degenerateexpD}, \eqref{expnuends} and \eqref{Xdepy}, we have 
$$
\begin{aligned}
& \lim_{y\to y_\wedge,\, y_\vee } (\pa_\mu D)(\mu(y,\e),\e) =
(\pa_\mu D) (\mu_0(\e)+\nu_\pm(\e),\e)=  
\pm 4\big(\beta_3 - \beta_2 \frac{T_2}{T_1} \big) T_1\e^4 + r(\e^5) \neq 0 \, ,  \\
& 
\lim_{y\to y_\wedge,\, y_\vee} (\pa_y \mu)(y,\e)  = \frac{1}{(\pa_\nu I)(\mu_\pm (\e),\e)} \stackrel{\eqref{derivatanonvanishing}}{\neq} 0 \, ,
 \end{aligned}\, ,
 $$
 and, since $ D(\mu (y,\e),\e) $ tends to 
$ 0 $ as $y \to y_\wedge(\e) $, $ y_\vee(\e)$, we deduce that 
 $|\pa_y X(y,\e)| $ in \eqref{upblowing} tends to $+\infty$. 

Finally, we claim that the region encircled by the two graphs \eqref{2graphs} is convex. It is sufficient to prove that $\pa_{yy} X(y,\e) $ is negative for any 
$y_\wedge(\e)<y<y_\vee (\e)$.  
Indeed, by \eqref{Xdepy}, 
\begin{align}
 \label{concavesso} 
 &\pa_{yy} X(y,\e) = \\
 &{\footnotesize \begin{matrix} \frac{2\big[ (\pa_{\mu}^2D) (\mu(y,\e),\e)( \pa_y \mu(y,\e))^2 + (\pa_\mu D) (\mu(y,\e),\e)\pa_y^2 \mu (y,\e) \big]D(\mu(y,\e),\e)-\big( \pa_\mu D(\mu(y,\e),\e) \pa_y \mu(y,\e)\big)^2}{8\big(D(\mu(y,\e),\e)\big)^{\frac32}} \end{matrix}} \, . \notag
\end{align} 
In view of \eqref{degenerateexpD} and \eqref{expnuends}  we have, for any  $y_\wedge(\e)<y<y_\vee (\e)$, 
\begin{align}\label{estimatesdiscri}
& \pa_\mu^2 D(\mu(y,\e),\e) \leq  - T_1^2  \, , \quad | \pa_\mu D(\mu(y,\e),\e)| \leq 8 \Big| T_1 \Big(\beta_3 - \beta_2 \frac{T_2}{T_1} \Big) \Big| \e^4  \, .
\end{align}
Moreover, by \eqref{derivatanonvanishing}, \eqref{eigs} and \eqref{degenerateexpS} there is $c>0$ such that,
for any 
$ \nu_\wedge(\e)<\nu <\nu_\vee(\e) $, 
$$
|\pa_\nu I(\nu,\e) | \geq \tfrac14 |\gamma_1-\alpha_1| \, , \qquad\quad |\pa_{\nu\nu} I(\nu,\e) | \leq c  \, ,
$$
and therefore, for some $C_1 >0$ and for any $ y_\wedge(\e) <y< y_\vee(\e)$,
\begin{equation}\label{inversionnuwrtyeps}
|\pa_y \mu(y,\e)| = |\pa_y \nu(y,\e)| \leq \frac{C_1}{|\gamma_1-\alpha_1|} \, , \quad |\pa_{yy} \mu(y,\e)| = |\pa_{yy} \nu(y,\e)| \leq \frac{C_1}{|\gamma_1-\alpha_1|} \,  \, .
\end{equation}
By  \eqref{inversionnuwrtyeps}  and \eqref{estimatesdiscri}, we have, for some $\tilde C>0$
\begin{equation}\label{convecavo}
\begin{aligned}
  \pa_\mu^2 D (\mu(y,\e),\e)( \pa_y \mu(y,\e))^2 & + \pa_\mu D(\mu(y,\e),\e)
 \pa_y^2 \mu (y,\e) \\
 & \leq -\frac{\tilde C T_1^2}{(\gamma_1-\alpha_1)^2} + \frac{\tilde C}{|\gamma_1-\alpha_1|} \e^4  < 0 
 \end{aligned}
 \end{equation}
 for $ \e $ small.
By \eqref{concavesso} and \eqref{convecavo}, the function $y\mapsto X(y,\e)$ is concave.
  \end{proof}

 A first approximation $\widetilde\lambda^\pm(\nu,\e) $ of  
 the  eigenvalues 
 $\lambda^\pm(\mu_0(\e)+\nu,\e)$ 
 of Lemma \ref{degenerateisola}, which 
neglects the remainders $r(\nu^3)$ of $D(\mu_0(\e)+\nu,\e)$ in \eqref{degenerateexpD} and $r(\nu^2)$ of $S(\mu_0(\e)+\nu,\e)$ in \eqref{degenerateexpS}, is
\begin{align}\label{truncation} \\[-6mm] \notag 
\begin{cases}
\begin{aligned} x &:= \text{Re}\ \widetilde\lambda^\pm(\nu,\e) \\[2mm]
& := {\footnotesize \begin{matrix} \pm \frac12 \sqrt{4\big(\beta_3 - \beta_2 \frac{T_2}{T_1} \big)^2 \e^8(1+r(\e)) -T_1^2 \nu^2(1+r(\e^2))+ 8 \beta_2 \big(\beta_3 - \beta_2 \frac{T_2}{T_1}  \big)\nu\e^6(1+r(\e))} \end{matrix} }\, ,
\end{aligned}
\\ y:= 
\text{Im}\ \widetilde\lambda^\pm(\nu,\e) :=\omega_*^{(p)} + \big(\frac{\gamma_2-\alpha_2}{2} - \frac{T_2(\gamma_1-\alpha_1)}{2T_1}\big) \e^2  (1+r(\e^2)) +\frac{\gamma_1-\alpha_1}{2}\nu (1+r(\e^2)) \, . 
\end{cases} 
\end{align} 
\normalsize
The functions $\widetilde\lambda^\pm(\nu,\e) $ are defined for $\nu$ in
 the interval $\widetilde\nu_\wedge(\e)\leq\nu\leq\widetilde\nu_\vee(\e) $ where the argument of the square root  in \eqref{truncation}  is non-negative. 
 % We now show that % as $ \nu $ varies, 
 These approximating eigenvalues describe an ellipse in the $(x,y)$-plane.

 \begin{lem} {\bf (Approximating ellipse)}\label{lem:appell} Suppose the coefficients $\alpha_1$ and $\gamma_1$ in \eqref{truncation} are different, i.e. $\gamma_1-\alpha_1\neq 0$. As $\nu$ varies between $\widetilde\nu_\wedge(\e)$ and $\widetilde\nu_\vee(\e)$
the approximating eigenvalues $ \widetilde\lambda^\pm(\nu,\e)  $ in 
\eqref{truncation}  
form an ellipse of equation
\begin{equation}\label{primeisolaeq}
x^2  + \dfrac{T_1^2(1+r(\e^2))}{(\gamma_1-\alpha_1)^2}(y-y_0(\e))^2 =  \big( \beta_3 - \beta_2 \frac{T_2}{T_1} \big)^2 \e^8(1+r(\e)) \, ,
\end{equation}
%in the complex $\lambda$-plane, 
%with $x := \text{Re}\, \lambda$, $y := \text{Im}\, \lambda$
centered at $(0,y_0(\e))$,  
where $y_0(\e)$  is an  analytic function of the form 
\begin{equation}\label{ycentreisola1}
y_0(\e) = \omega_*^{(p)} + \Big(\dfrac{\gamma_2-\alpha_2}{2} - \dfrac{T_2(\gamma_1-\alpha_1)}{2T_1}\Big) \e^2  +r(\e^4)  \, .
\end{equation} 
\end{lem}
\begin{proof}
We invert the second equation in \eqref{truncation} and obtain the function
\begin{equation}\label{expnuye} 
\nu(y,\e) = \dfrac{2}{\gamma_1-\alpha_1}(y-\tilde y(\e))(1+r(\e^2))
\end{equation}
where
$$
\tilde y(\e):=\text{Im}\ \tilde\lambda^\pm(0,\e)=\omega_*^{(p)} + \Big(\dfrac{\gamma_2-\alpha_2}{2} - \dfrac{T_2(\gamma_1-\alpha_1)}{2T_1}\Big) \e^2  (1+r(\e^2)) \,  . 
$$
By plugging the  expansion \eqref{expnuye} for $\nu=\nu(y,\e) $ in the equation for $x^2$, obtained by squaring the first line in \eqref{truncation}, we get the equation of a conic $0=e(x,y):=x^2 - [\text{Re}\; \tilde \lambda^\pm(\nu(y,\e),\e)]^2$, with
\begin{align*}
e(x,y) :&= x^2+  \dfrac{T_1^2}{(\gamma_1-\alpha_1)^2} (y-\tilde y(\e))^2(1+r(\e^2))\\  &  - 4 \dfrac{\beta_2\big( \beta_3 - \beta_2 \frac{T_2}{T_1} \big)}{\gamma_1-\alpha_1}(y-\tilde y(\e))\e^6(1+r(\e))- \Big( \beta_3 - \beta_2 \frac{T_2}{T_1} \Big)^2 \e^8(1+r(\e))\, . \notag
\end{align*}
Then one puts the conic into its canonical form \eqref{primeisolaeq} with $y_0(\e)-\tilde y(\e) = r(\e^6)$.
\end{proof}

\section{Taylor expansion of $\cB(\mu,\e) $, 
$P({\mu,\e})$ and  $\mathfrak{B}(\mu,\e)$}\label{sec:5}

In this section we provide  the Taylor  expansion of the operators
$\cB(\mu,\e) $ in \eqref{WW},
the projectors 
$P({\mu,\e})$
and  the operators $\mathfrak{B}(\mu,\e)$ defined in \eqref{Bgotico}
 around
$ (\umu,0) $. 

\smallskip
\noindent{\bf $\bullet$ Notation.}
For an operator $A=A(\mu,\e; x)$ we denote
% its Taylor coefficients at the point $(\umu,0)$ as
\begin{subequations}\label{notazione}
\begin{equation}
A_{i,j} := A_{i,j}(\umu + \delta,\e; x) :=  \frac{1}{i!j!} \big(\pa^i_\mu\pa^j_\e A \big)(\umu,0;x)\, 
\delta^i \e^j ,\qquad  A_k :=  \sum_{\substack{i+j=k\\ i,j\geq 0}} A_{i,j} \, .\vspace{-2mm}
\end{equation}
We also denote
by $A_{i,j}^{[\kappa]} $ the part of the operator $A_{i,j}$
with Fourier harmonic $e^{\im \kappa x}$, i.e.
\begin{equation}
\begin{aligned}
&
A_{i,j}^{[\kappa]} 
% := A_{i,j}^{[\kappa]} (\umu + \delta,\e; x)
:= \frac{e^{\im \kappa x}}{2\pi} \int_0^{2\pi} A_{i,j}(\umu + \delta,\e; y) e^{-\im \kappa y}\de y\,  , \quad
A_\ell^{[\kappa]} :=  \sum_{\substack{i+j= \ell \\ i,j\geq 0}} A_{i,j}^{[\kappa]}  \, .
\end{aligned}\vspace{-2mm}
\end{equation}
\end{subequations}
It results 
\begin{equation}\label{Pellk*}
[A_\ell^{[\kappa]}]^* = (A^*)_\ell^{[-\kappa]}\, .
\end{equation}
We shall occasionally split $A_{i,j}=A_{i,j}^{[\mathtt{ ev}]} + A_{i,j}^{[\mathtt{odd}]}$ where $A_{i,j}^{[\mathtt{ ev}]}$ is the part of the operator $A_{i,j}$ having only even harmonics, whereas $A_{i,j}^{[\mathtt{odd}]}$ is the part with only odd ones. \smallskip

 We denote by  $\cO(\delta^{m_1}\e^{n_1},\dots,\delta^{m_p}\e^{n_p})$, $ m_j, n_j \in \bN  $, 
analytic functions of $(\delta,\e)$ with values in a Banach space $X$ which satisfy the estimate
 $\|\cO(\delta^{m_j}\e^{n_j})\|_X \leq C \sum_{j = 1}^p |\delta|^{m_j}|\e|^{n_j}$, for some $ C > 0 $ and 
 for small values of $(\delta, \e)$. For any $ k\in\bN $ we denote  by $\cO_k$ an operator mapping $H^1(\bT,\bC^2)$  into $L^2(\bT,\bC^2)$-functions with  size $\e^k,\delta \e^{k-1} ,\dots, \delta^{k-1}\e$ or $\delta^k $.

We directly have the following expansion
recalling \eqref{pfunction}-\eqref{afunction} and  \eqref{grazieStrauss}.

\begin{lem}
The operator ${\cal B}  ({\mu, \e})  $ in 
\eqref{WW} expands as
\begin{equation*} % \label{def.Bj}
{\cal B} ( \umu+\delta,\e ) = {\cal B}_0 + {\cal B}_1 + {\cal B}_2+ {\cal B}_3+{\cal B}_4+\cO_5 \, ,  
\end{equation*}
where 
\begin{subequations}\label{Bsani}
\begin{alignat}{3} 
&{\cal B}_0 &&\quad = \quad {\cal B}_{0,0} \ \ \;=
 &&\begingroup 
\setlength\arraycolsep{3pt}\begin{bmatrix} 1 & - \pa_x - \im \umu \\  \pa_x  + \im \umu&  |D+ \umu|\end{bmatrix}\endgroup \, ,\\
  \label{B1sano} 
  & {\cal B}_{1} &&=  {\cal B}_{0,1} +  {\cal B}_{1,0}=
 \e  &&  \begingroup 
\setlength\arraycolsep{-8pt} \begin{bmatrix} a_1(x) & -p_1(x)(\pa_x+\im  \umu) \\ (\pa_x+\im  \umu)\circ p_1(x) & 0 \end{bmatrix}\endgroup +\delta   \begingroup 
\setlength\arraycolsep{2pt} \begin{bmatrix} 0 & - \im  \\ \im  & \sgn^+(D) \end{bmatrix} \endgroup \, ,  \\
\label{B2sano}
 &{\cal B}_{2} &&=  {\cal B}_{0,2} +  {\cal B}_{1,1}=
 \e^2 &&\begingroup 
\setlength\arraycolsep{-8pt} \begin{bmatrix} a_2(x) & -p_2(x)(\pa_x+\im  \umu) \\ (\pa_x+\im  \umu)\circ p_2(x) & 0 \end{bmatrix} \endgroup + \delta\e  \begingroup 
\setlength\arraycolsep{-2pt} \begin{bmatrix} 0 & -\im p_1(x) \\ \im p_1(x) & 0 \end{bmatrix} \endgroup \,  , \\
 \label{B3sano}
  &{\cal B}_{3} &&=  {\cal B}_{0,3} +  {\cal B}_{1,2}=
 \e^3 &&\begingroup 
\setlength\arraycolsep{-8pt} \begin{bmatrix} a_3(x) & -p_3(x)(\pa_x+\im  \umu) \\ (\pa_x+\im  \umu)\circ p_3(x) & 0 \end{bmatrix}\endgroup + \delta\e^2  \begingroup 
\setlength\arraycolsep{-2pt} \begin{bmatrix} 0 & -\im p_2(x) \\ \im p_2(x) & 0 \end{bmatrix} \endgroup \,  ,\\
  \label{B4sano}
  &{\cal B}_{4} &&=  {\cal B}_{0,4} +  {\cal B}_{1,3}=
 \e^4 &&\begingroup 
\setlength\arraycolsep{-8pt} \begin{bmatrix} a_4(x) & -p_4(x)(\pa_x+\im  \umu) \\ (\pa_x+\im  \umu)\circ p_4(x) & 0  \end{bmatrix}\endgroup + \delta\e^3  \begingroup 
\setlength\arraycolsep{-2pt} \begin{bmatrix} 0 & -\im p_3(x) \\ \im p_3(x) & 0 \end{bmatrix} \endgroup \, ,
\end{alignat}
\end{subequations}
with $p_k (x)$ and $a_k (x)$, $k = 1,\dots,4$, in \eqref{pfunction}-\eqref{afunction}.
\end{lem}
Note that the functions $p_k (x)$ and $a_k (x)$ in \eqref{apexp} have only even (resp. odd) harmonics when $k$ is even (resp. odd). Consequently, with the notation introduced below \eqref{Pellk*}, we have
\begin{equation}\label{cB evodd}
\cB_{i,j}^{[\mathtt{ ev}]} = \begin{cases} \cB_{i,j} &\textup{if }j\textup{ is even}\, , \\ 
0 &\textup{if }j\textup{ is odd}\, , 
\end{cases} \qquad \cB_{i,j}^{[\mathtt{odd}]} = \begin{cases} 0 &\textup{if }j\textup{ is even}\, , \\ 
\cB_{i,j}  &\textup{if }j\textup{ is odd}\, . 
\end{cases} 
\end{equation}
We remark that sum and composition of operators satisfying  \eqref{cB evodd} still satisfy \eqref{cB evodd}. 

Analogously we expand the projectors $P({\mu,\e})$ in  \eqref{Pproj} as
\begin{equation*}
P ( \umu+\delta,\e) =P_0 + P_1 + P_2 + P_3 + \cO_4 \, 
\end{equation*}
where 
\begin{equation}\label{Psani}
\begin{aligned}
 &P_0 := P (\umu,0) \ , \quad P_1 := \mathcal{P} \big[\cB_1\big] \, , \quad P_2 := \mathcal{P} \big[\cB_2\big] + \mathcal{P} \big[\cB_1,\cB_1 \big] \, ,\\
  &P_3 := \mathcal{P} \big[\cB_3\big] + \mathcal{P} \big[\cB_2,\cB_1 \big] + \mathcal{P} \big[\cB_1,\cB_2 \big]  + \mathcal{P} \big[\cB_1,\cB_1,\cB_1 \big] \, ,
 \end{aligned} 
 \end{equation}
and, for any $k \in \bN $, 
\begin{equation}\label{hP}
 \begin{aligned} \mathcal{P} & \big[A_1,\dots, A_k \big] :=\\
 & \frac{(-1)^{k+1} }{2\pi \im } \oint_\Gamma (\cL_{ \umu,0}-\lambda)^{-1} \cJ A_1 (\cL_{ \umu,0}-\lambda)^{-1} \dots \cJ A_k (\cL_{ \umu,0}-\lambda)^{-1} \de\lambda \, ,
 \end{aligned}
\end{equation}
with $\Gamma$ the same circuit of Lemma \ref{lem:Kato1}.
In virtue of \eqref{Psani}-\eqref{hP} and \eqref{cB evodd} we obtain
\begin{equation}
\label{P evodd}
P_{i,j}^{[\mathtt{ ev}]} = \begin{cases} P_{i,j} &\textup{if }j\textup{ is even}\, , \\ 
0 &\textup{if }j\textup{ is odd}\, , 
\end{cases} \qquad P_{i,j}^{[\mathtt{odd}]} = \begin{cases} 0 &\textup{if }j\textup{ is even}\, , \\ 
P_{i,j}  &\textup{if }j\textup{ is odd}\, . 
\end{cases} 
\end{equation}

Now  we provide the expansion of the operators $\mathfrak{B} ({\mu,\e}) $.
Let   $\mathbf{Sym}[A]:= \frac12 A+ \frac12 A^*$.
\begin{lem}[{\bf Expansion of $\mathfrak{B} (\mu,\e) $}] \label{expansionthm}
 The operator $\mathfrak{B} (\mu,\e) $  in \eqref{Bgotico} has the Taylor expansion 
 \begin{equation*}
\mathfrak{B} ({\umu+\delta,\e}) =
\mathfrak{B}_0 +
\mathfrak{B}_1 +
\mathfrak{B}_2 +
\mathfrak{B}_3+ 
\mathfrak{B}_4 +\cO_5
\end{equation*}
where
 \begin{subequations}\label{upto4exp} 
\begin{align}\label{ordini012}
\mathfrak{B}_0 & := P_0^*{\cal B}_0 P_0 \, , 
 \quad 
 \mathfrak{B}_1 := P_0^*{\cal B}_1 P_0 \, , 
 \quad 
 \mathfrak{B}_2 := P_0^*\mathbf{Sym}[{\cal B}_2+{\cal B}_1P_1]  P_0\, , 
 \\  \label{ordine3}
 \mathfrak{B}_3 & := P_0^*\mathbf{Sym}[{\cal B}_3+{\cal B}_2P_1 + {\cal B}_1(\uno-P_0)P_2]P_0\, , 
 \\ \label{ordine4}
  \mathfrak{B}_4 &= P_0^*\mathbf{Sym}[{\cal B}_4+ {\cal B}_3P_1 + {\cal B}_2(\uno-P_0)P_2  + {\cal B}_1(\uno-P_0)P_3 -  {\cal B}_1P_1 P_0 P_2 ] P_0\, , 
\end{align}
\end{subequations}
with ${\cal B}_j$, $j=0,\dots,4$, in \eqref{Bsani} and $P_j$, $j=0,\dots,3$, in \eqref{Psani}.
\end{lem}
\begin{proof}
The proof follows  as Lemma 3.6 of \cite{BMV_ed}, obtaining the same expansions (3.24a)--(3.24c) 
of \cite{BMV_ed}. In the present case the  last operator  in formula (3.24c) of \cite{BMV_ed}, namely $P_0^*\mathbf{Sym}[\mathfrak{N} P_0 P_2 ] P_0$, where  
$\mathfrak{N}:= \frac14\left(P_2^* \cB_0  - \cB_0 P_2  \right)$,   actually vanishes, in view of 
the identity $P_0^* \cB_0 P_2 P_0 = P_0^* P_2^* \cB_0 P_0 $ that we now prove. 
 First we have\footnote{This identity does not hold in \cite{BMV_ed} because of the presence of a generalized eigenvector.
 % $\cL_{0,0}f_0^+ = -f_0^- $.
 } 
 $\cB_0 P_0 = P_0^* \cB_0  = - \im \omega_*^{(p)} \cJ P_0 $, that follows from
$$
(\cB_0 P_0 f , g) = (\cL_{\umu,0} P_0 f, \cJ g) = \im \omega_*^{(p)} (P_0 f, \cJ g) = - \im \omega_*^{(p)} (\cJ P_0 f, g) \, .
$$
This identity, together with  $\cJ P_j = P_j^* \cJ$ (a consequence of $P(\mu,\e)$ being skew-Hamiltonian), gives
$P_0^* \cB_0 P_j P_0 = P_0^* P_j^* \cB_0 P_0 $ for any $j \in \bN$.
\end{proof}
In virtue of \eqref{notazione}, \eqref{cB evodd}, \eqref{P evodd} and \eqref{upto4exp} we obtain
\begin{equation}\label{kB evodd}
\kB_{i,j}^{[\mathtt{ ev}]} = \begin{cases} \kB_{i,j} &\textup{if }j\textup{ is even}\, , \\ 
0 &\textup{if }j\textup{ is odd}\, , 
\end{cases} \qquad \kB_{i,j}^{[\mathtt{odd}]} = \begin{cases} 0 &\textup{if }j\textup{ is even}\, , \\ 
\kB_{i,j}  &\textup{if }j\textup{ is odd}\, . 
\end{cases} 
\end{equation}

\section{Entanglement coefficients}\label{entanglementcoefficients}

In this section we introduce the {\em entanglement coefficients} that represent how the  jets of the operator   $\mathfrak{B}(\mu,\e)$  act on the unperturbed  eigenvector basis \eqref{def:fsigmaj}.

\subsection{Abstract representation formulas}

Take $ \und\mu \in \bR \setminus \bZ $ so that the eigenvectors
$\{f_j^\sigma\}_{j \in \bZ, \sigma  = \pm}$, $f_j^\sigma:=f_j^\sigma(\und \mu)$ in  \eqref{def:fsigmaj}, 
form a complex symplectic basis of $ L^2 (\bT,\bC^2)$.
Our goal is to describe the action of 
the  operators $\cJ \cB_\ell$ and   $ \mathcal{P}[\cB_{\ell_1}, \ldots, \cB_{\ell_s}]$ (recall  \eqref{hP}) on a  vector $f_j^\sigma$ of the basis.
To do so, we introduce  the {\it entanglement coefficients} 
\begin{equation}\label{entcoeff}
\ent{\ell }{\kappa }{j'}{j}{\sigma'}{\sigma}
:= \big({\cal B}_\ell^{[\kappa]} f_j^\sigma, f_{j'}^{\sigma'}  \big)   \, ,\qquad \ell \in \bN^2_0 \, ,\ j',j\in \bZ\, , \ \sigma,\sigma'=\pm\, ,
\end{equation}
where  $\cB_\ell^{[\kappa]}$ is the $\kappa$th Fourier coefficient of the operator $\cB_\ell$ in  \eqref{Bsani} (according to \eqref{notazione}).
We stress that in this section $\ell $ is always a pair $ \ell=(i,j)\in \bN_0^2$.

 The  entanglement coefficients fulfill
\begin{equation}\label{entprop}
\ent{\ell}{\kappa }{j'}{j}{\sigma'}{\sigma} = 0 \ \text{ if } \ j'\neq j+\kappa \, ,\quad\text{ and }\quad \bar{\ent{\ell}{\kappa }{j'}{j}{\sigma'}{\sigma}} =\ \ent{\ell}{-\kappa }{j}{j'}{\sigma}{\sigma'}\, .  
\end{equation} 
 The next lemma provides effective formulas to compute the action of the  operators $\cJ \cB_\ell$ and   $ \mathcal{P}[\cB_{\ell_1}, \ldots, \cB_{\ell_s}]$ on the vector basis.
\begin{lem}
\label{actionofL}
% \red{Let $ \mu \in \bR \setminus \bZ $ so that the eigenvectors
% $\{f_j^\sigma(\mu)\}_{j \in \bZ, \sigma  = \pm}$ in  \eqref{def:fsigmaj}
% form a basis of $ L^2 (\bT,\bC^2)$. }
Let  $\ent{\ell}{\kappa }{j'}{j}{\sigma'}{\sigma}$ denote the entanglement coefficients in \eqref{entcoeff} and
$f_j^\sigma:=f_j^\sigma(\und \mu)$ in  \eqref{def:fsigmaj} with $ \umu \in \bR \setminus \bZ $.
  Then: \\
$(i)$ for any $\ell\in \bN_0^2$ and $j,\kappa\in \bZ$ and $\sigma =\pm $  one has
\begin{equation}\label{Lsacts}
\cJ {\cal B}_\ell^{[\kappa]} f_j^\sigma =
 \sum_{\sigma_1 = \pm} -\im \sigma_1  
 \, 
  \ent{\ell}{\kappa}{j+\kappa}{j}{\sigma_1}{\sigma}\; f_{j+\kappa}^{\sigma_1} 
=  \im \ent{\ell}{\kappa}{j+\kappa}{j}{-}{\sigma} \, f_{j+\kappa}^- -\im \ent{\ell}{\kappa}{j+\kappa}{j}{+}{\sigma}\,  f_{j+\kappa}^+ \, ;
\end{equation}
$(ii)$ for any $q\in \bN$, $ \ell_1,\dots,\ell_q\in \bN_0^2$, $j,\kappa_1,\dots,\kappa_q\in \bZ $ and $\sigma=\pm$, the operator 
$ \mathcal{P}[\cB_{\ell_q}^{[\kappa_q]},\dots,\cB_{\ell_1}^{[\kappa_1]}]  $
defined via \eqref{hP}, satisifies
\begin{align}
\label{hppar}
&\mathcal{P}[\cB_{\ell_q}^{[\kappa_q]},\dots,\cB_{\ell_1}^{[\kappa_1]}]  f_j^\sigma 
\\ \notag
&\quad =\!\!\!\!\!\!\!\!\!\! \sum_{\sigma_1,\dots,\sigma_q  = \pm }\!\!\!\!\!\!\!  \sigma_1\dots\sigma_q\,  
\ent{\ell_q}{\kappa_q}{j_q}{j_{q-1}}{\sigma_q}{\sigma_{q-1}}
\dots 
\;\ent{\ell_2}{\kappa_2}{j_2}{j_1}{\sigma_2}{\sigma_1} 
\ent{\ell_1}{\kappa_1}{j_1}{j}{\sigma_1}{\sigma} 
\;\Res{j, &j_1, &\dots, &j_q}{\sigma, &\sigma_1, &\dots, &\sigma_q} \; f_{j_q}^{\sigma_q} 
\end{align}
where $j_1 := j+\kappa_1$, $j_2 := j_1+\kappa_2$, \dots, $j_q=j_{q-1}+\kappa_q$,  and
\begin{equation}\label{generalresidue}
 \Res{j, &j_1, &\dots, &j_q}{\sigma, &\sigma_1, &\dots, &\sigma_q} := \dfrac{ (-\im\!)^q }{2\pi\im} \oint_\Gamma  \dfrac{ \de \lambda }{(\lambda-\im \omega_j^\sigma)(\lambda-\im \omega_{j_1}^{\sigma_1})\dots (\lambda-\im \omega_{j_q}^{\sigma_q}) }
 \end{equation}
with $\Gamma$ is a circuit winding once around $ \im \omega_*^{(p)}$ counterclockwise and $\omega_{j}^{\pm}:=\omega_{j}^{\pm}(\umu)$ in \eqref{omeghino}. 

The coefficients $\Res{j_0, &j_1, &\dots, &j_q}{\sigma_0, &\sigma_1, &\dots, &\sigma_q} $ are real and invariant by permutations of the indexes, namely
\begin{equation}\label{revres}
\begin{aligned}
 &\Res{j_0, &j_1, &\dots, &j_q}{\sigma_0, &\sigma_1, &\dots, &\sigma_q}  = \bar{ \Res{j_0, &j_1, &\dots, &j_q}{\sigma_0, &\sigma_1, &\dots, &\sigma_q} }\, ,\\
&  \Res{j_0, &j_1, &\dots, &j_q}{\sigma_0, &\sigma_1, &\dots, &\sigma_q} = \Res{j_{\tau(0)}, &j_{\tau(1)}, &\dots, &j_{\tau(q)}}{\sigma_{\tau(0)}, &\sigma_{\tau(1)}, &\dots, &\sigma_{\tau(q)}}\, ,
 \ \  \mbox{ for any  permutation } \tau \ \mbox{of } \{0,1,\ldots,q\}.
 \end{aligned}
\end{equation}
\smallskip
$(iii)$  for any $q\in \bN$, $ \ell_1,\dots,\ell_{q+1}\in \bN_0^2$, $j,j',\kappa_1,\dots,\kappa_{q+1}\in \bZ $ and $\sigma,\sigma'=\pm$, one has 
\begin{subequations}\label{cBactstot}
\begin{align}\label{cBactspar}
&\big( \cB_{\ell_{q+1}}^{[\kappa_{q+1}]}\mathcal{P}[\cB_{\ell_q}^{[\kappa_q]},\dots,\cB_{\ell_1}^{[\kappa_1]}]  f_j^\sigma , f_{j'}^{\sigma'}\big)  \\ \notag 
&\qquad =\!\!\!\!\!\!\!\!\!\! \sum_{\sigma_1,\dots,\sigma_q = \pm }\!\!\!\!\!\!\!  
\sigma_1\dots\sigma_q 
\ent{\ell_{q+1}}{\kappa_{q+1}}{j'}{j_q}{\sigma'}{\sigma_q} 
\ent{\ell_q}{\kappa_q}{j_q}{j_{q-1}}{\sigma_q}{\sigma_{q-1}} 
\dots
\ent{\ell_2}{\kappa_2}{j_2}{j_1}{\sigma_2}{\sigma_1} 
\ent{\ell_1}{\kappa_1}{j_1}{j}{\sigma_1}{\sigma} 
\Res{j, &j_1, &\dots, &j_q}{\sigma, &\sigma_1, &\dots, &\sigma_q}  
\end{align}
with $j_1 := j+\kappa_1$, $j_2 := j_1+\kappa_2$, \dots, $j_q=j_{q-1}+\kappa_q$, and
\begin{align}
 \label{cBactsparbis}
&\big( \cB_{\ell_1}^{[\kappa_1]}  f_{j}^{\sigma} , \mathcal{P}[\cB_{\ell_2}^{[-\kappa_2]},\dots,\cB_{\ell_{q+1}}^{[-\kappa_{q+1}]}] f_{j'}^{\sigma'} \big)  \\ \notag &\qquad =\!\!\!\!\!\!\!\! \sum_{\sigma_1,\dots,\sigma_q = \pm }\!\!\!\!\!\!\!  \sigma_1\dots\sigma_q \, \ent{\ell_{q+1}}{\kappa_{q+1}}{j'}{\xi_q}{\sigma'}{\sigma_q} \ent{\ell_{q}}{\kappa_{q}}{\xi_q}{\xi_{q-1}}{\sigma_q}{\sigma_{q-1}} \dots \ent{\ell_2}{\kappa_2}{\xi_{2}}{\xi_1}{\sigma_{2}}{\sigma_1} \ent{\ell_1}{\kappa_1}{\xi_1}{j}{\sigma_1}{\sigma} \Res{j', &\xi_q, &\dots, &\xi_1}{\sigma', &\sigma_q, &\dots, &\sigma_1}  
\end{align}
\end{subequations}
with $\xi_q := j'-\kappa_{q+1}$, $\xi_{q-1} := \xi_q-\kappa_{q}$, \dots, $\xi_{1} := \xi_2-\kappa_2$.
\end{lem}
\begin{proof}
$(i)$ 
Since the operator ${\cal B}_\ell^{[\kappa]}$ shifts the harmonic $j$ to $j+\kappa$ and $\cJ$ leaves it invariant, there are 
scalars $ \alpha^\pm
\in \bC $ such that  
\begin{equation}\label{Lsactspar}
\cJ {\cal B}_\ell^{[\kappa]} f_j^\sigma = \alpha^- f_{j+\kappa}^- + \alpha^+  f_{j+\kappa}^+ \, .
\end{equation} 
Then 
\begin{equation}\label{troppocaldo}
\big( \cJ {\cal B}_\ell^{[\kappa]} f_j^\sigma, \cJ f_{j+\kappa}^- \big) = \alpha^- \big(  f_{j+\kappa}^-, \cJ f_{j+\kappa}^- \big) + \alpha^+ \big(  f_{j+\kappa}^+, \cJ f_{j+\kappa}^- \big) \stackrel{\eqref{sympbas}}{=} -\im \alpha^- \, . 
\end{equation}
 On the other hand 
\begin{equation}\label{zusivviu}
\big( \cJ {\cal B}_\ell^{[\kappa]} f_j^\sigma, \cJ f_{j+\kappa}^- \big) = \big({\cal B}_\ell^{[\kappa]} f_j^\sigma, f_{j+\kappa}^- \big) \stackrel{\eqref{entcoeff}}{=} \ent{\ell}{\kappa}{j+\kappa}{j}{-}{\sigma}\, .
\end{equation}
Similarly $ \ent{\ell}{\kappa}{j+\kappa}{j}{+}{\sigma} =  \big( \cJ {\cal B}_\ell^{[\kappa]} f_j^\sigma, \cJ f_{j+\kappa}^+ \big) =  \im \alpha^+ $. By  \eqref{Lsactspar}, \eqref{troppocaldo} and \eqref{zusivviu} we get \eqref{Lsacts}.

$(ii)$ By \eqref{def:fsigmaj} we have $(\cL_{\umu,0}-\lambda)^{-1} f_j^\sigma = \frac{1}{\im\omega_j^\sigma-\lambda} f_j^\sigma$. Thus, in view of \eqref{hP},
$$
\mathcal{P}[\cB_{\ell_q}^{[\kappa_q]},\dots,\cB_{\ell_1}^{[\kappa_1]}]  f_j^\sigma = \frac{(-1)^{q+1}}{2\pi\im} \oint_\Gamma  \frac{(\cL_{\umu,0}-\lambda)^{-1}}{\im \omega_j^\sigma -\lambda} \cJ {\cal B}_{\ell_q}^{[\kappa_q]} \dots (\cL_{\umu,0}-\lambda)^{-1}\cJ {\cal B}_{\ell_1}^{[\kappa_1]} f_j^\sigma  \de \lambda\,  \ . 
$$
Now it is enough to use repeatedly the formula above, setting $j_0:=j$ and $\sigma_0:=\sigma $, 
$$
(\cL_{\umu,0}-\lambda)^{-1}\cJ {\cal B}_{\ell_\iota}^{[\kappa_\iota]} f_{j_{\iota-1}}^{\sigma_{\iota-1}}  = \sum_{\sigma_\iota = \pm} -\im \sigma_\iota  \frac{ \ent{\ell_\iota}{\kappa_\iota}{j_\iota}{j_{\iota-1}}{\sigma_\iota}{\sigma_{\iota-1}} }{\im \omega_{j_\iota}^{\sigma_\iota}-\lambda} f_{j_\iota}^{\sigma_\iota} \ \ , \qquad j_\iota := j_{\iota-1} + \kappa_\iota \ ,\quad \iota=1,\dots,q\, ,
$$
(which follows from  \eqref{Lsacts}) to obtain \eqref{hppar}.

Next we prove properties  \eqref{revres}. The second one follows trivially from the definition  \eqref{generalresidue}. For the first one we have, by \eqref{generalresidue}, 
$$ 
\bar{\Res{j_0, &j_1, &\dots, &j_q}{\sigma_0, &\sigma_1, &\dots, &\sigma_q} } =  \frac{\im\!^q}{2\pi\im} \int_0^{2\pi} \dfrac{-\bar{\gamma'}(t)\de t}{(\bar{\gamma}(t)+\im \omega_{j_0}^{\sigma_0})\ldots (\bar{\gamma}(t)+\im \omega_{j_q}^{\sigma_q})}\, ,
$$
for a closed path
$\gamma(t)$ winding around $\im \omega_*^{(p)}$ counterclockwise. Thus, being $-\bar{\gamma}(t)$ reverse oriented, 
$$ 
\bar{\Res{j_0, &j_1, &\dots, &j_q}{\sigma_0, &\sigma_1, &\dots, &\sigma_q} } = -\frac{\im\!^q}{2\pi\im} \oint_\Gamma \dfrac{\de\lambda}{(\im \omega_{j_0}^{\sigma_0}-\lambda)\ldots (\im \omega_{j_q}^{\sigma_q}-\lambda)}  =  \Res{j_0, &j_1, &\dots, &j_q}{\sigma_0, &\sigma_1, &\dots, &\sigma_q} ,
 $$
whence the first identity in \eqref{revres} follows.\\
$(iii)$ Identity \eqref{cBactspar} is a consequence of \eqref{hppar} and \eqref{entcoeff}. By a similar argument one proves \eqref{cBactsparbis}, using also \eqref{entprop} and \eqref{revres}. 
\end{proof}

The following identities will be particularly useful to  identify expressions of the form \eqref{cBactspar}--\eqref{cBactsparbis} that coincide.

\begin{lem}
Let $(j,\sigma)\, ,\;(j',\sigma')$ be such that $\omega_j^\sigma = \omega_{j'}^{\sigma'}$. 
Then
\begin{equation}\label{psmirabilis}
\big( \cB_{\ell_{q+1}}^{[\kappa_{q+1}]}\mathcal{P}[\cB_{\ell_q}^{[\kappa_q]},\dots,\cB_{\ell_1}^{[\kappa_1]}]  f_j^\sigma , f_{j'}^{\sigma'}\big) = \big( \cB_{\ell_1}^{[\kappa_1]}  f_{j}^{\sigma} , \mathcal{P}[\cB_{\ell_2}^{[-\kappa_2]},\dots,\cB_{\ell_{q+1}}^{[-\kappa_{q+1}]}] f_{j'}^{\sigma'} \big) \, .
\end{equation}
\end{lem}
\begin{proof}
If $j'\neq j+\kappa_1+\dots+\kappa_{q+1} $ then
both sides of identity \eqref{psmirabilis} vanish. 
Otherwise $j'= j+\kappa_1+\dots+\kappa_{q+1} $ and, 
recalling the definitions of $j_s, \xi_s$ in \eqref{cBactspar}, \eqref{cBactsparbis}, one has 
$$
\xi_s = j' - \kappa_{s+1} -  \ldots - \kappa_{q+1} = j +\kappa_1 + \ldots + \kappa_s = j_s \ , 
\quad
\forall \, s=1,\dots,q \ .
$$
Using  \eqref{generalresidue}, the second property of 
\eqref{revres} and the assumption 
$\omega_j^\sigma = \omega_{j'}^{\sigma'}$, one has   $\Res{j, &j_1, &\dots, &j_q}{\sigma, &\sigma_1, &\dots, &\sigma_q}=\Res{j', &\xi_q, &\dots, &\xi_1}{\sigma', &\sigma_q, &\dots, &\sigma_1}.$ 
Thus by \eqref{cBactspar}, \eqref{cBactsparbis}, the two sides of \eqref{psmirabilis} are equal. 
\end{proof}

\subsection{Entanglement coefficients for $p=2$}
We conclude this section by giving explicit formulas for the entanglement  coefficients in \eqref{entcoeff} and
the residue term in \eqref{generalresidue} for the particular case $p=2$, where we assume \eqref{choicek}, namely we fix $\umu=\tfrac14 $, $k=0$, $k'=2$ and $\omega_*=\omega_*^{(2)} $. 

We first consider the Fourier series expansions of the
even functions $p_n(x), a_n (x) $, $ n \in \bN $,  defined in \eqref{apexp}, 
\begin{equation}
\label{ajphfe}
\begin{aligned}
p_n (x) &= \frac12 p_n^{[0]}+ \sum_{\kappa\geq1} p_n^{[\kappa]} \cos(\kappa x) = \frac12 p_n^{[0]}+ \sum_{\kappa\geq1} \frac{p_n^{[\kappa]}}{2}  e^{\im \kappa x}+ \frac{p_n^{[-\kappa]}}{2}e^{-\im \kappa x} \, ,\\ a_n(x) &= \frac12 a_n^{[0]}+ \sum_{\kappa\geq1} a_n^{[\kappa]} \cos(\kappa x)= \frac12 a_n^{[0]}+ \sum_{\kappa\geq1} \frac{a_n^{[\kappa]}}{2}  e^{\im \kappa x}+ \frac{a_n^{[-\kappa]}}{2}e^{-\im \kappa x} \, ,
\end{aligned}
\end{equation}
with $p_n^{[-\kappa]} :=p_n^{[\kappa]} $ and $a_n^{[-\kappa]} := a_n^{[\kappa]} $ for any $\kappa \in \bN $.
In view of \eqref{apexp}
 for any 
$ n = 1, \ldots, 4 $ the non zero Fourier coefficients  
$ p_n^{[\kappa]}$, $ a_n^{[\kappa]} $ are 
\begin{equation}
\label{pklakl}
\begin{aligned}
& p_1^{[\pm 1]} =a_1^{[\pm 1]} := -2  \, , \quad
p_2^{[0]} = 3 \, , \quad a_2^{[0]} = 4\, ,\quad 
p_2^{[\pm 2]} = a_2^{[\pm 2]} = -2 \, , \\
&p_3^{[\pm 1]} =  3\, , \quad 
a_3^{[\pm 1]} =  4 \, , \quad 
p_3^{[\pm 3]} = a_3^{[\pm 3]} = -3 \, ,  \\
& p_4^{[0]} = \frac14\, ,\quad a_4^{[0]} = -2   \, , \quad 
p_4^{[\pm 2]} = a_4^{[\pm 2]} = 4   \, , \quad 
p_4^{[\pm 4]} = a_4^{[\pm 4]} = -\frac{16}{3}  \, . 
\end{aligned}
\end{equation} 
In view of \eqref{notazione},
 the operators in \eqref{Bsani} have jets
 \begin{equation}
  \label{calB1} 
  {\cal B}_{1,0}^{[0]}=   \begin{bmatrix}0 & - \im  \\ \im  & \sgn^+ (D) 
  \end{bmatrix}\delta \, ,
  % \quad {\cal B}_{0,1}^{[\pm 1]}= \frac{e^{\pm \im x}}{2}\begingroup \setlength\arraycolsep{-7pt}\begin{bmatrix} a_1^{[1]} & -p_1^{[1]}(\pa_x+\im \umu) \\ p_1^{[1]}(\pa_x+\im \umu\pm\im) & 0 \end{bmatrix}\endgroup \e\, ,
 \end{equation}
and, for any 
$n= 1,\dots,4$ 
(recall 
also \eqref{ajphfe}, \eqref{pklakl})
\begin{equation}\label{paritymatters}
\!\!\!\!\! {\cal B}_{i,n}^{[\kappa]} = {\footnotesize\begin{matrix}
\begin{cases} \dfrac{e^{\im\kappa x}}{2} \begingroup \setlength\arraycolsep{-7pt} \begin{bmatrix} a_n^{[\kappa]} & - p_n^{[\kappa]}(\pa_x+\im \frac14) \\   p_n^{[\kappa]}(\pa_x+\im \frac14+ \im \kappa ) & 0  \end{bmatrix}\endgroup \e^n \, ,
 \   \textup{if } i = 0,  \ \kappa\equiv n \ (\textup{mod }2)\, ,\ |\kappa|\leq n\, , \\[5mm]
  \dfrac{e^{\im\kappa x}}{2} 
\begingroup \setlength\arraycolsep{2pt} \begin{bmatrix} 0 & - \im p_{n-1}^{[\kappa]} \\  \im p_{n-1}^{[\kappa]} & 0  \end{bmatrix}\endgroup \delta \e^{n-1} \, ,\quad\quad\quad\ \ \,
 \textup{if } i=1 , \ \ \kappa\not\equiv n \ (\textup{mod }2)\, ,\ |\kappa|\leq n-1\, , \\
 0\,,  \qquad\qquad\qquad\qquad\qquad\qquad\qquad\qquad\qquad\quad\qquad\qquad\qquad\qquad\ \ \text{otherwise}.
 \end{cases}\end{matrix}}
\end{equation}

\begin{lem}\label{lem:entexp}
For any  $j \in \bZ $ and $\sigma,\sigma'=\pm 1$, the  nonzero
 entanglement coefficients in \eqref{entcoeff} are:
 \begin{itemize}
 \item  If $i=1$, $n=0$ then 
 \end{itemize}
 \begin{subequations}\label{entexp}
 \begin{equation}\label{entexp0}
 \ent{1,0}{0}{j}{j}{\sigma'}{\sigma} =
 \frac{\sqrt{\sigma}\bar{\sqrt{\sigma'}} }{2\Omega_j} \Big(  \sigma\sigma' \sgn^+(j) - (\sigma + \sigma') \Omega_j\Big) \delta 
 \end{equation}
  where $
\Omega_j := \sqrt{| j+\tfrac14 | }  
$;
  \begin{itemize}
\item If $i=0$, 
 $n=1,2,3,4$, $\kappa\in\bZ$ then 
 \end{itemize}
 \begin{equation}\label{entexp1}
 \ent{0,n}{\kappa}{j+\kappa }{j}{\sigma'}{\sigma} =\tfrac14{\footnotesize\begin{matrix}
 \sqrt{\sigma}\bar{\sqrt{\sigma'}} \sqrt{\Omega_{j+\kappa}\Omega_{j}}\Big(a_n^{[\kappa]}-\sigma p_n^{[\kappa]}  \Omega_j \sgn^+(j) -\sigma' p_n^{[\kappa]} \Omega_{j+\kappa}\sgn^+(j+\kappa) \Big)\e^n   \end{matrix}}
 \end{equation}
 if  $\kappa \equiv n \ (\text{mod }2)$ and  $|\kappa|\leq n$, with constants $p_n^{[\kappa]}$, $a_n^{[\kappa]}$ given in \eqref{pklakl}, and vanish otherwise;
 \begin{itemize}
 \item If $i=1$, 
 $n=1,2,3$, $\kappa\in\bZ$ then 
 \end{itemize}
 \begin{equation}\label{entexp2}
\ent{1,n}{\kappa}{j+\kappa }{j}{\sigma'}{\sigma} =  - \frac14 \sqrt{\sigma}\bar{\sqrt{\sigma'}} \dfrac{1}{\sqrt{\Omega_j\Omega_{j+\kappa}}} \Big(\sigma \Omega_{j+\kappa}+ \sigma' \Omega_{j} \Big) p_{n-1}^{[\kappa]} \delta \e^{n-1} 
 \end{equation}
 if $\kappa \not \equiv n \ (\text{mod }2)$ and  $|\kappa|\leq n-1$, with constants $p_n^{[\kappa]}$ given in \eqref{pklakl}, and vanish otherwise.
\end{subequations}
\end{lem}
\begin{proof}
Recall that $f_j^\sigma = f_j^\sigma(\tfrac14) $ are given in \eqref{def:fsigmaj}.
 In view of  \eqref{calB1} we have
$$
{\cal B}_{0,1}^{[0]} f_j^\sigma  
= \frac{e^{\im j x}}{\sqrt{2\Omega_j}}\vet{- \im \sqrt{-\sigma}}{- \im \sqrt{\sigma} \Omega_j + \sqrt{-\sigma} \sgn^+(j)} \delta\, ,
$$
and by \eqref{entcoeff}
$$
 \ent{0,1}{0}{j}{j}{\sigma'}{\sigma}  = \frac{1}{2\Omega_j}\left( \sqrt{-\sigma}\, \bar{\sqrt{-\sigma'}} \,  \sgn^+(j) 
 + \im \sqrt{-\sigma}\,  \bar{\sqrt{\sigma'}} \, \Omega_j - \im \sqrt{\sigma} \bar{\sqrt{-\sigma'}} \Omega_j
 \right)\delta \, .
$$
Then \eqref{entexp0} follows because  for any  $\sigma,\sigma'=\pm$
\begin{equation}\label{idsqrtsigma}
{\footnotesize \begin{matrix}
 \sqrt{-\sigma}\bar{\sqrt{-\sigma'}} = \sigma\sigma'\sqrt{\sigma}\bar{\sqrt{\sigma'}}\, , \ \
 \im\sqrt{-\sigma}\bar{\sqrt{\sigma'}} = -\sigma \sqrt{\sigma}\bar{\sqrt{\sigma'}},\ \  -\im\sqrt{\sigma}\bar{\sqrt{-\sigma'}} = -\sigma' \sqrt{\sigma}\bar{\sqrt{\sigma'}} \end{matrix}} .
\end{equation} 
We now prove \eqref{entexp1}. In this case,
in view of \eqref{paritymatters} and  \eqref{def:fsigmaj}, we obtain
$$
{\cal B}_{0,n}^{[\kappa]} f_j^\sigma = \frac{e^{\im (j+\kappa )x}}{2\sqrt{2\Omega_j}}\vet{-a_n^{[\kappa]}\sqrt{\sigma}\Omega_j-\im\sqrt{-\sigma}p_n^{[\kappa]}(j+\frac14)}{-\im\sqrt{\sigma}p_n^{[\kappa]}\Omega_j(j+ \kappa +\frac14) } \e^n \, ,
$$
thus by \eqref{entcoeff} we have 
$$
\begin{aligned}
\ent{0,n}{\kappa}{j+\kappa}{j}{\sigma'}{\sigma}&=\frac{1}{4\sqrt{\Omega_j\Omega_{j+\kappa}}}\Big( \sqrt{\sigma}\bar{\sqrt{\sigma'}} a_n^{[\kappa]} \Omega_j\Omega_{j+\kappa} + \im \sqrt{-\sigma}\bar{\sqrt{\sigma'}} p_n^{[\kappa]} \Omega_{j+\kappa}(j+\tfrac14)\\ &\qquad \qquad -\im \sqrt{\sigma}\bar{\sqrt{-\sigma'}} p_n^{[k]} \Omega_j(j+\kappa+\tfrac14)\Big)\e^n \, .
\end{aligned}
$$
Formula
 \eqref{entexp1} follows by \eqref{idsqrtsigma}  and  
 $$
 j+\tfrac14 = \Omega_j^2\, \sgn^+(j)\,, \quad j+\kappa+\tfrac14 = \Omega_{j+\kappa}^2 \,
\sgn^+(j+\kappa)\, .
 $$
Similarly we prove \eqref{entexp2}. By \eqref{paritymatters} and \eqref{def:fsigmaj},
$$
\cB_{1,n}^{[\kappa]} f_j^\sigma = - \im p_{n-1}^{[\kappa]} \frac{e^{\im (j+\kappa)x}}{2\sqrt{2\Omega_j}}\vet{\sqrt{-\sigma}}{\sqrt{\sigma}\Omega_j}\delta\e^{n-1}\, ,
$$
thus by \eqref{entcoeff} we have
$$
\ent{1,n}{\kappa}{j+\kappa }{j}{\sigma'}{\sigma} = \frac{p_{n-1}^{[\kappa]} }{4\sqrt{\Omega_j\Omega_{j+\kappa}}} \big( \im \sqrt{-\sigma}\bar{\sqrt{\sigma'}} \Omega_{j+\kappa} - \im \sqrt{\sigma}\bar{\sqrt{-\sigma'}}\Omega_{j} \big)\delta\e^{n-1} \, ,
$$
and by \eqref{idsqrtsigma} we conclude \eqref{entexp2}. 
\end{proof}

We now give some effective formulas to compute the residue term in \eqref{generalresidue}. 
\begin{lem}\label{rem:res}
 Let $j_0,j_1,\dots,j_q\in \bN_0$,  $\sigma_0,\dots,
 \sigma_q=\pm$.  
Then the 
coefficient $\Res{j_0, &j_1, &\dots, &j_q}{\sigma_0, &\sigma_1, &\dots, &\sigma_q}$ in \eqref{generalresidue} fulfills
\\
\texttt{\bf I.} If for any $ \iota =0,\dots,q $ one has $\omega_{j_\iota}^{\sigma_\iota} \neq \omega_* $ (no pole) then
\begin{equation}\label{residuevalue}
 \Res{j_0, &j_1, &\dots, &j_q}{\sigma_0, &\sigma_1, &\dots, &\sigma_q}  = 0 \ ;
\end{equation}
\texttt{\bf II.} If there is one and only one index $ \iota \in\{0,\dots,q \}$ such that $\omega_{j_\iota}^{\sigma_\iota} =\omega_* $  (single pole) then
\begin{equation}\label{residuevalue1}
 \Res{j_0, &j_1, &\dots, &j_q}{\sigma_0, &\sigma_1, &\dots, &\sigma_q}  = \frac{ 1}{ ( \omega_{j_0}^{\sigma_0}-\omega_*)\ldots ( \omega_{j_{\iota-1}}^{\sigma_{\iota-1}}-\omega_*)( \omega_{j_{\iota+1}}^{\sigma_{\iota+1}}-\omega_*)\ldots(\omega_{j_q}^{\sigma_q}-\omega_*) }  \ ;
\end{equation}
\texttt{\bf III.} If there are two and only two indices $\iota_1,\iota_2 \in\{0,\dots,q\}$ such that  $\omega_{\iota_1}^{\sigma_{\iota_1}}=\omega_{\iota_2}^{\sigma_{\iota_2}} = \omega_*$ (double pole) then
\begin{equation}\label{residuevalue2}
   \Res{j_0, &j_1, &\dots, &j_q}{\sigma_0, &\sigma_1, &\dots, &\sigma_q}  = - \Big(\!\!\! \sum_{\substack{m=0\\ m\neq \iota_1,\iota_2}}^q \frac{1}{\omega_{j_m}^{\sigma_m} - \omega_* } \Big)\, \Big(\!\!\! \prod_{\substack{k=0\\ k\neq \iota_1,\iota_2}}^q \frac{1}{\omega_{j_k}^{\sigma_k} - \omega_* }\Big) \, . 
\end{equation}
\texttt{\bf IV.} If $q \geq 1$ and  $\omega_{j_0}^{\sigma_0} = \dots =  \omega_{j_q}^{\sigma_q}  =\omega_*$ (i.e. pole of order $q+1$), then $\Res{j_0, &j_1, &\dots, &j_q}{\sigma_0, &\sigma_1, &\dots, &\sigma_q} = 0$.
\end{lem}
\begin{proof}
Apply the residue theorem to formula \eqref{generalresidue}.
\end{proof}

\begin{rmk}\label{remarkinoid}
From \eqref{hppar} one checks that 
  \begin{subequations}\label{looping}
 \begin{align}\label{outsideloop}
 &(\uno-P_0) \mathcal{P}[\cB_{\ell_q}^{[\kappa_q]},\dots,\cB_{\ell_1}^{[\kappa_1]}]   f_j^\sigma  = \\ \notag
  &{\footnotesize \begin{cases}
 \underset{\sigma_1,\dots,\sigma_{q-1} = \pm }{\sum}   \sigma_1\dots\sigma_{q-1} \ent{\ell_1}{\kappa_1}{j_1}{j}{\sigma_1}{\sigma}  \ent{\ell_2}{\kappa_2}{j_2}{j_1}{\sigma_2}{\sigma_1} \dots \ent{\ell_q}{\kappa_q}{0}{j_{q-1}}{+}{\sigma_{q-1}}\Res{j, &j_1, &\dots, &0}{\sigma, &\sigma_1, &\dots, &+}  f_{0}^{+} \, ,&\text{ if }j+\!\!\underset{{i=1}}{\stackrel{q}{\sum}} \kappa_i = 0\, , \\
\quad -\!\!\!\!\!\!\!\!  \underset{\sigma_1,\dots,\sigma_{q-1} = \pm }{\sum}  \sigma_1\dots\sigma_{q-1} \ent{\ell_1}{\kappa_1}{j_1}{j}{\sigma_1}{\sigma}  \ent{\ell_2}{\kappa_2}{j_2}{j_1}{\sigma_2}{\sigma_1} \dots \ent{\ell_q}{\kappa_q}{2}{j_{q-1}}{-}{\sigma_{q-1}}\Res{j, &j_1, &\dots, &2}{\sigma, &\sigma_1, &\dots, &-}  f_{2}^{-} \, ,&\text{ if }j+\!\!\underset{{i=1}}{\stackrel{q}{\sum}} \kappa_i = 2\, ,\\[4mm]
 \mathcal{P}[\cB_{\ell_q}^{[\kappa_q]},\dots,\cB_{\ell_1}^{[\kappa_1]}]   f_j^\sigma\, , &\text{ otherwise}\, ,
 \end{cases}}
 \end{align}
 \normalsize
 and
  \begin{align}\label{insideloop}
 &P_0 \mathcal{P}[\cB_{\ell_q}^{[\kappa_q]},\dots,\cB_{\ell_1}^{[\kappa_1]}]   f_j^\sigma  = \\ \notag
  & {\footnotesize\begin{cases}
\quad -\!\!\!\!\!\!\!\!  \underset{\sigma_1,\dots,\sigma_{q-1} = \pm }{\sum}  \sigma_1\dots\sigma_{q-1} \ent{\ell_1}{\kappa_1}{j_1}{j}{\sigma_1}{\sigma}  \ent{\ell_2}{\kappa_2}{j_2}{j_1}{\sigma_2}{\sigma_1} \dots \ent{\ell_q}{\kappa_q}{0}{j_{q-1}}{-}{\sigma_{q-1}}\Res{j, &j_1, &\dots, &0}{\sigma, &\sigma_1, &\dots, &-}  f_{0}^{-} \, ,&\text{ if }j+\!\!\underset{{i=1}}{\stackrel{q}{\sum}} \kappa_i= 0\, , \\
 \underset{\sigma_1,\dots,\sigma_{q-1} = \pm }{\sum}  \sigma_1\dots\sigma_{q-1} \ent{\ell_1}{\kappa_1}{j_1}{j}{\sigma_1}{\sigma}  \ent{\ell_2}{\kappa_2}{j_2}{j_1}{\sigma_2}{\sigma_1} \dots \ent{\ell_q}{\kappa_q}{2}{j_{q-1}}{+}{\sigma_{q-1}}\Res{j, &j_1, &\dots, &2}{\sigma, &\sigma_1, &\dots, &+}  f_{2}^{+} \, ,&\text{ if }j+\!\!\underset{{i=1}}{\stackrel{q}{\sum}} \kappa_i = 2\, ,\\[4mm]
 0 \, , &\text{ otherwise}\, .
 \end{cases}}
 \end{align}
 \normalsize
 \end{subequations}
\end{rmk}

\section{Taylor expansion of  $\tB(\mu,\e)$}\label{section:conti}

Let us assume from now on \eqref{choicek}.
The main result of this section 
is the following 
expansion of the matrix $\tB(\mu,\e) $ in \eqref{tocomputematrix} 
which directly  implies Theorems \ref{thm:expansion} and \ref{thm:main}.

\begin{prop}[{\bf Expansion of $\tB(\mu,\e) $}] \label{expbT}
The coefficients $\alpha(\mu,\e)$, $\beta(\mu,\e)$ and $\gamma(\mu,\e)$ of the $ 2 \times 2 $ Hermitian
 matrix $ \tB(\mu,\e) $ in \eqref{tocomputematrix}  admit the  expansions  at  $(\mu,\e) = (\tfrac14+\delta, \e)$
 \begin{align}\label{abc}
&\alpha(\tfrac14+\delta, \e) = 
 \alpha(\tfrac14+\delta, 0)
+\widetilde{ \alpha_1} \e
+  \alpha_2\e^2  
+\widetilde{\alpha_2} \delta\e + 
  r(\e^3,\delta\e^2,\delta^2\e) \,, \notag \\
& \beta(\tfrac14+\delta, \e) =  
\beta(\tfrac14+\delta,0) 
+ 
\widetilde{\beta_1} \e 
+ 
 \beta_1 \e^2+ \widetilde{\beta_2}\delta\e+ \beta_2 \delta\e^2 
+ \widetilde{\beta}_3 \e^3 +  
  \beta_3 \e^4 \\ 
 &\qquad\qquad\qquad +\widetilde{\beta_4} \delta^2\e + r (\e^{5},\delta\e^3,\delta^2\e^2,\delta^3\e)  \,, \notag \\ \notag
&  \gamma(\tfrac14+\delta, \e) =  \gamma(\tfrac14+\delta, 0) +\widetilde{\,\gamma_1\,} \e + \gamma_2 \e^2 +  
\widetilde{\,\gamma_2\,} \delta \e+   r(\e^3,\delta\e^2,\delta^2\e)\, ,
 \end{align}
where 
%the function $\mu\mapsto \im %\beta(\mu,0)    $  vanishes as well as 
% the coefficients
\begin{subequations}\label{TCabc}
\begin{equation}
\begin{alignedat}{2}
\label{zerocoeff}
& \widetilde{\alpha_1} \e := \left( \kB_{0,1} f_2^+, f_2^+ \right)  = 0 \, , 
\qquad
 &&\ \widetilde{\alpha_2} \delta\e  := \left( \kB_{1,1} f_2^+, f_2^+ \right)  = 0\,, \\ 
\im &\widetilde{\beta_1} \e := \left( \kB_{0,1} f_0^-, f_2^+ \right)  = 0\, , \qquad  
&&\im \widetilde{\beta}_3 \e^3 :=\left( \kB_{0,3} f_0^-, f_2^+ \right)  = 0 \, , \\
&\widetilde{\,\gamma_1\,} \e := \left( \kB_{0,1} f_0^-, f_0^- \right)  = 0 
\, , 
\qquad 
&&\ \widetilde{\,\gamma_2\,} \delta \e := \left( \kB_{1,1} f_0^-, f_0^- \right)  = 0 \, , \\
&\widetilde{\,\beta_2\,} \delta\e := \left( \kB_{1,1} f_0^-, f_2^+ \right)  = 0 
\, , 
\qquad 
&&\ \widetilde{\,\beta_4\,} \delta^2\e := \left( \kB_{2,1} f_0^-, f_2^+ \right)  = 0 \, , \\
\end{alignedat}
\end{equation}
and 
\begin{align}\notag
&  \alpha_2 \e^{2} := \left( \kB_{0,2} f_2^+, f_2^+ \right)  =\frac98 \e^{2} \,  ,
 \quad 
\gamma_2 \e^{2} :=  \left( \kB_{0,2} f_0^-, f_0^- \right) = \frac1{16}\e^{2} \, , \\
& 
 \im \beta_1 \e^{2} :=  \left( \kB_{0,2} f_0^-, f_2^+ \right)  = 0  \, , \label{epsilonspento2}  \\
& \notag  \im \beta_2 \delta \e^2  := \left( \kB_{1,2} f_0^-, f_2^+ \right) =  -\dfrac{1}{2\sqrt{3}} \delta \e^2 \, , \quad
\im  \beta_3 \e^4 := \left( \kB_{0,4} f_0^-, f_2^+ \right)
= -\frac{39 \sqrt{3}}{512} \e^4 \, .
\end{align}
\end{subequations}
For any $\mu $ we have 
$ \beta(\mu,0)  = 0 $
and 
$ \alpha(\mu,0)=  
- \omega_{2}^+(\mu)$, 
$ \gamma(\mu,0) =  \omega_{0}^-(\mu) $ 
(cfr. \eqref{bmuzero}). 
\end{prop}

\begin{proof}[Proof of Theorems \ref{thm:main} and \ref{thm:expansion}]
Theorem \ref{thm:expansion} follows by  \eqref{abc}--\eqref{TCabc}  and  using that, 
by \eqref{omeghino}, 
\begin{align*}
& \alpha(\tfrac14+\delta,0) = - \Big( \frac94 + \delta  - \Omega(\tfrac94+\delta) \Big) = - \omega_* + \alpha_1 \delta + r(\delta^2)\ ,  \ \  \alpha_1:=\Omega'(\tfrac94)-1 \stackrel{\eqref{omeghino}}{= } -\tfrac23  \ , 
 \\
&  \gamma(\tfrac14+\delta,0)= \tfrac14 + \delta + \Omega(\tfrac14+\delta) = \omega_*  + \gamma_1 \delta  + r(\delta^2)\ , \quad 
\gamma_1:=\Omega'(\tfrac14)+1 \stackrel{\eqref{omeghino}}{=} 2  \ .
\end{align*}
Furthermore Assumption \ref{H} 
is fulfilled because $T_1:= \alpha_1 + \gamma_1 =  \frac43 $. Then we are in the  first degenerate case of \eqref{cases}  since $\beta_1 = 0$, 
$T_2 = \frac{19}{16}$ and 
$ \beta_3-\beta_2\frac{T_2}{T_1}=\frac{37 \sqrt{3}}{512} 
$. 
Finally $\gamma_1-\alpha_1=\tfrac83 \neq 0$, 
then the 
abstract Theorem \ref{degenerateisola} applies and 
proves 
Theorem \ref{thm:main} together with \eqref{isola.intro}. The expansions in \eqref{mupm} descend from \eqref{muexpa} and \eqref{expnuends}. 
\end{proof}

The rest of the section is devoted to the proof of Proposition 
\ref{expbT}.  We first  prove  \eqref{zerocoeff}.
\begin{lem}
The coefficients $\widetilde{\alpha}_1$, $\widetilde{\alpha}_2$, $\widetilde{\beta}_1$, $\widetilde{\beta}_2$, $\widetilde{\beta}_3$, $\widetilde{\beta}_4$, $\widetilde{\gamma}_1$, $\widetilde{\gamma}_2$  in \eqref{zerocoeff} vanish.
\end{lem}
\begin{proof} Each scalar product in \eqref{zerocoeff} splits into two terms accordingly to the splitting of the operators in even and odd harmonics defined below \eqref{Pellk*}. The term with odd harmonics vanishes because the difference between the harmonics of $f_0^-$ and $f_2^+$ is $2$.
Thus the coefficients are respectively  given by
$ \big( \kB_{0,1}^{[\mathtt{ ev}]} f_2^+, f_2^+ \big)$,  
$ \big( \kB_{1,1}^{[\mathtt{ ev}]} f_2^+, f_2^+ \big)$,  
$ 
\big( \kB_{0,1}^{[\mathtt{ ev}]} f_0^-, f_2^+ \big) $, $ 
\big( \kB_{1,1}^{[\mathtt{ ev}]} f_0^-, f_2^+ \big) $,  $ 
\big( \kB_{2,1}^{[\mathtt{ ev}]} f_0^-, f_2^+ \big) $, 
$ \big( \kB_{0,3}^{[\mathtt{ ev}]} f_0^-, f_2^+ \big)$, 
$ \big( \kB_{0,1}^{[\mathtt{ ev}]} f_0^-, f_0^- \big) $,  
$\big( \kB_{1,1}^{[\mathtt{ ev}]} f_0^-, f_0^- \big) $, 
and vanish because, by \eqref{kB evodd}, the operators 
$\kB_{0,1}^{[\mathtt{ ev}]}$, $\kB_{1,1}^{[\mathtt{ ev}]}$, $\kB_{0,1}^{[\mathtt{ ev}]}$, 
$\kB_{0,3}^{[\mathtt{ ev}]}$, $\kB_{0,1}^{[\mathtt{ ev}]}$, $ \kB_{1,1}^{[\mathtt{ ev}]} $ and $\kB_{2,1}^{[\mathtt{ev}]}$ are zero.
\end{proof}

We now compute the remaining coefficients in \eqref{epsilonspento2}. 
We give their  algebraic expression in terms of the entanglement coefficients, and their numerical values exploiting Mathematica. The code can be found
at the link in footnote\footnote{
\url{https://git-scm.sissa.it/amaspero/first-isola-of-modulational-instability-of-stokes-waves-in-deep-water}}.
We start with the quadratic terms.

\smallskip
\noindent{\bf Computation of $\alpha_2$}.  In view of \eqref{epsilonspento2}, \eqref{ordini012}, since $P_{1}^{[\pm1]}=\mathcal P[\cB_{1}^{[\pm 1]}]$ where by \eqref{paritymatters}  $ \cB_{0,1}^{[\kappa]}$ (and  similarly $P_{0,1}^{[\kappa]}$) is nonzero only for $\kappa=\pm 1$, one has
\begin{align*}
\alpha_2\e^2 & =  \left( \kB_{0,2} f_2^+, f_2^+ \right)  = 
 \left( \cB_{0,2} f_2^+, f_2^+ \right) 
 +
 \frac12 \left( \cB_{0,1} P_{0,1}  f_2^+, f_2^+ \right) 
  +
\frac12   \left( \cB_{0,1} f_2^+,  P_{0,1} f_2^+ \right) \\
   & = 
   	\underbrace{ \big( {\cal B}_{0,2}^{[0]} f_2^+, f_2^+\big)}_{ =\mathrm{Aa}  }
	 + 
   \frac12
  \underbrace{ \big( {\cal B}_{0,1}^{[+1]} P_{0,1}^{[-1]}f_2^+, f_2^+\big)}_{=:\mathrm{Ab} } 
    +
    \frac12 
\underbrace{    \big( {\cal B}_{0,1}^{[-1]} P_{0,1}^{[+1]}f_2^+, f_2^+\big)}_{=:\mathrm{Ac} }
\\ &\quad +
 \frac12
 \underbrace{\big( {\cal B}_{0,1}^{[+1]} f_2^+, P_{0,1}^{[+1]} f_2^+\big) }_{\stackrel{\eqref{psmirabilis}}{=}\mathrm{Ac}}
 +
 \frac12 
\underbrace{ \big( {\cal B}_{0,1}^{[-1]} f_2^+, P_{0,1}^{[-1]} f_2^+\big) }_{\stackrel{\eqref{psmirabilis}}{=}\mathrm{Ab}}
= \mathrm{Aa} + \mathrm{Ab} + \mathrm{Ac}
   \end{align*}
(for the last two terms we also used  \eqref{Pellk*}).  One then has, also by using \eqref{entprop},
   \begin{align}\label{A}
   & \mathrm{Aa} \stackrel{\eqref{entcoeff}}{=}  {\tB_{0,2}^{[0]}}_{2,2}^{+,+}  = -\frac{15}{8} \e^2 \ ;\\ \notag
   & \mathrm{Ab} = \big( {\cal B}_{0,1}^{[+1]} P_{0,1}^{[-1]}f_2^+, f_2^+\big)\stackrel{\eqref{psmirabilis}}{=} 
  \big( {\cal B}_{0,1}^{[-1]} f_2^+, P_{0,1}^{[-1]} f_2^+\big)     \stackrel{\eqref{cBactstot}, \eqref{residuevalue1}}{=}  \sum_{\sigma =\pm}   \frac{ \sigma |\sent{0,1}{-1}{1}{2}{\sigma}{+}|^{2}}{\omega_1^{\sigma}- \omega_{*}}
   = - \frac{15}{8}\e^2 ; 
   \\ \notag
   & \mathrm{A c} 
   = \big( {\cal B}_{0,1}^{[-1]} P_{0,1}^{[+1]}f_2^+, f_2^+\big)
   \stackrel{\eqref{psmirabilis}}{=}  \big( {\cal B}_{0,1}^{[+1]} f_2^+, P_{0,1}^{[+1]} f_2^+\big) 
   \stackrel{\eqref{cBactstot}, \eqref{residuevalue1}}{=}  \sum_{\sigma =\pm}  \frac{ \sigma |\sent{0,1}{+1}{3}{2}{\sigma}{+}|^{2}}{\omega_3^{\sigma}- \omega_{*}} = \frac{39}{8}\e^2 \ . 
   \end{align}
We conclude that $\alpha_2 = \frac98 $ as  claimed in \eqref{epsilonspento2}. \smallskip

\noindent{\bf Computation of $\beta_1$}. In view of \eqref{epsilonspento2}, \eqref{ordini012}, since $P_{1}^{[\pm1]}=\mathcal P[\cB_{1}^{[\pm 1]}]$ where by \eqref{paritymatters}  $ \cB_{0,1}^{[\kappa]}$ (and  similarly $P_{0,1}^{[\kappa]}$) is nonzero only for $\kappa=\pm 1$, one has
\begin{equation*}
\begin{aligned}
\im \beta_1\e^2 & =   \left( \kB_{0,2} f_0^-, f_2^+ \right) \\
& = \underbrace{ \big({\cal B}_{0,2}^{[+2]} f_0^-\,,\,f_{2}^+\big)}_{=:\mathrm{Ba}} + 
\frac12
\underbrace{ \big({\cal B}_{0,1}^{[+1]}P_{0,1}^{[+1]} f_0^-\, , \,f_{2}^+\big)}_{=:\mathrm{Bb}} + 
\frac12 
\underbrace{
\big({\cal B}_{0,1}^{[+1]} f_0^-\, , \, P_{0,1}^{[-1]}f_{2}^+\big)}_{\stackrel{\eqref{psmirabilis}}{=}\mathrm{Bb}}\, ,
\end{aligned} 
\end{equation*}
(for the last term we also used  \eqref{Pellk*}). We have 
\begin{align}\label{B}
\mathrm{Ba} &= \big({\cal B}_{0,2}^{[+2]} f_0^-\,,\,f_{2}^+\big)  \stackrel{\eqref{entcoeff}}{=}  {\tB_{0,2}^{[+2]}}_{2,0}^{+,-} \stackrel{\eqref{entexp}}{=} 0\\ \notag
\mathrm{Bb} &= 
 \big({\cal B}_{0,1}^{[+1]}P_{0,1}^{[+1]} f_0^-\, , \,f_{2}^+\big)
 \stackrel{\eqref{psmirabilis}}{=}
\big({\cal B}_{0,1}^{[+1]} f_0^-\, , \, P_{0,1}^{[-1]}f_{2}^+\big) \\ \notag
&\stackrel{\eqref{cBactstot}}{=} -\frac{{\tB_{0,1}^{[+1]}}_{1,0}^{-,-}{\tB_{0,1}^{[+1]}}_{2,1}^{+,-}}{\omega_1^--\omega_*}+ \frac{{\tB_{0,1}^{[+1]}}_{1,0}^{+,-}{\tB_{0,1}^{[+1]}}_{2,1}^{+,+}}{\omega_1^+-\omega_*} \stackrel{\eqref{entexp}}{=}0\, .
\end{align}
We conclude that $\beta_1 = 0$ as  claimed in \eqref{epsilonspento2}. 
\begin{rmk}
In the finite depth case the coefficient $\beta_1\neq 0$. In infinite depth the degeneracy
$ \beta_1 = 0 $ descends from  peculiar identities of the expansion of the Stokes wave. These are  coherent with the structure of the completely integrable 
quartic Birkhoff normal form of
the pure gravity water waves in deep water  proved in  \cite{ZakD,CW,CS2,BFP}, see also \cite{Wu3,DIP}.
\end{rmk}

\noindent{\bf Computation of $\gamma_2$}. Proceeding in a similar way as for the computation of $\alpha_{2}$, we get 
\begin{align*}
\gamma_2\e^2 & =  \left( \kB_{0,2} f_0^-, f_0^- \right) 
 = 
\underbrace{\big( {\cal B}_{0,2}^{[0]} f_0^-, f_0^-\big) }_{=: \mathrm{Ca}}
+
 \frac12
 \underbrace{\big( {\cal B}_{0,1}^{[+1]} P_{0,1}^{[-1]}f_0^-, f_0^-\big)}_{=: \mathrm{Cb}}
  +
  \frac12
  \underbrace{ \big( {\cal B}_{0,1}^{[-1]} P_{0,1}^{[+1]}f_0^-, f_0^-\big)}_{=: \mathrm{Cc}}
  \\ &\quad + 
  \frac12
  \underbrace{\big( {\cal B}_{0,1}^{[+1]} f_0^-, P_{0,1}^{[+1]} f_0^-\big) }_{\stackrel{\eqref{psmirabilis}}{=} \mathrm{Cc}}
  +
  \frac12 
  \underbrace{\big( {\cal B}_{0,1}^{[-1]} f_0^-, P_{0,1}^{[-1]} f_0^-\big)}_{\stackrel{\eqref{psmirabilis}}{=} \mathrm{Cb}} = 
   \mathrm{Ca} +  \mathrm{Cb} +  \mathrm{Cc}\, ,
  \end{align*}
where
\begin{align}
& \mathrm{Ca} = \big( {\cal B}_{0,2}^{[0]} f_0^-, f_0^-\big) = {\tB_{0,2}^{[0]}}_{0,0}^{-,-} =\frac78 \e^2 \label{C} \\
& \mathrm{Cb} = \big( {\cal B}_{0,1}^{[+1]} P_{0,1}^{[-1]}f_0^-, f_0^-\big)
= \big( {\cal B}_{0,1}^{[-1]} f_0^-, P_{0,1}^{[-1]} f_0^-\big)
 =
\dfrac{\big|{\tB_{0,1}^{[-1]}}_{-1,0}^{+,-}\big|^2}{\omega_{-1}^+ - \omega_*}- \dfrac{\big|{\tB_{0,1}^{[-1]}}_{-1,0}^{-,-}\big|^2}{\omega_{-1}^- - \omega_*} = -\frac{3}{16}\e^2
    \notag \\
& \mathrm{Cc} = 
\big( {\cal B}_{0,1}^{[-1]} P_{0,1}^{[+1]}f_0^-, f_0^-\big) 
=
\big( {\cal B}_{0,1}^{[+1]} f_0^-, P_{0,1}^{[+1]} f_0^-\big) 
=\dfrac{\big|{\tB_{0,1}^{[+1]}}_{1,0}^{+,-}\big|^2}{\omega_1^+ - \omega_*}-\dfrac{\big|{\tB_{0,1}^{[+1]}}_{1,0}^{-,-}\big|^2}{\omega_1^- - \omega_*} = -\frac{5}{8}\e^2 \, . \notag 
\end{align}
We conclude that $\gamma_2 = \frac{1}{16} $ as  claimed in \eqref{epsilonspento2}. \smallskip

We now proceed with the cubic term $\gamma_2 $.\\
\noindent{\bf Computation of $\beta_2$}. 
We shall use the following  identities.
\begin{lem}
For any $ (j,\sigma) \in \{(0,-), (2,+) \}$ we have 
\begin{align}
\label{beta2.aux1}
& P_0 \mathcal{P}[\cB_{0,2}^{[-2]}] f_2^+  = {P_0 \mathcal{P}[\cB_{0,2}^{[0]}] f_2^+} = P_0 \mathcal{P}[\cB_{0,2}^{[+2]}] f_0^-  ={P_0 \mathcal{P}[\cB_{0,2}^{[0]}] f_0^-} =0 \, , \\
\label{beta2.aux2}
& P_0 \mathcal{P}[\cB_{1,1}^{[\pm 1]}] f_j^\sigma = P_0  \mathcal{P}[\cB_{0,1}^{[\pm 1]}, \cB_{1,0}^{[0]}] f_j^\sigma = P_0  \mathcal{P}[\cB_{1,0}^{[0]}, \cB_{0,1}^{[\pm 1]}] f_j^\sigma = 0 \, . 
\end{align}

\end{lem}
\begin{proof}
Identities  \eqref{beta2.aux1} follow from formula \eqref{insideloop} with $q = 1$ and $(j,\sigma) = (2,+)$ or $(0,-)$,  since in both cases, by \eqref{generalresidue},  the residue coefficients $\tR_{j, 0}^{\sigma, -}$ and $\tR_{j,2}^{\sigma, +} $ vanish.
Identities  \eqref{beta2.aux2} follow  from \eqref{insideloop} too, with $j + \kappa_0 + \dots + k_q \neq 0,2$.
\end{proof}

We claim that the  coefficient $\beta_2$ is a  linear combination of the following terms:  
\begin{align}\label{addendsbeta2}
 \Theta\;  &: = \big({\cal B}_{1,2}^{[+2]} f_0^-\,,\,f_{2}^+\big)   \, ,\\ \notag
 \mathrm{I} &: =\big({\cal B}_{1,1}^{[+1]}\mathcal{P}[\cB_{0,1}^{[+1]}] f_0^-\, , \,f_{2}^+\big) =  \big({\cal B}_{0,1}^{[+1]} f_0^-\, , \, \mathcal{P}[{\cal B}_{1,1}^{[-1]}] f_{2}^+\big)   \, ,\\ \notag
  \mathrm{II}\; & := \big({\cal B}_{0,2}^{[+2]}\mathcal{P}[\cB_{1,0}^{[0]}] f_0^-\, , \,f_{2}^+\big)=  \big( \cB_{1,0}^{[0]} f_0^-\, , \,(\uno-P_0) \mathcal{P}[{\cal B}_{0,2}^{[-2]}] f_{2}^+\big)  \, ,\\
\mathrm{III\,a} & := \big({\cal B}_{0,1}^{[+1]}  \mathcal{P}[\cB_{1,1}^{[+1]}]  f_0^-\, , \,f_{2}^+\big) = \big( \cB_{1,1}^{[+1]}  f_0^-\, , \, \mathcal{P}[{\cal B}_{0,1}^{[-1]} ] f_{2}^+\big) \notag \, ,\\
 \notag
 \mathrm{III\,b}\; &:=\big({\cal B}_{0,1}^{[+1]}  \mathcal{P}[\cB_{1,0}^{[0]},\cB_{0,1}^{[+1]}]  f_0^-\, , \,f_{2}^+\big) = \big({\cal B}_{0,1}^{[+1]}  f_0^-\, , \,  \mathcal{P}[\cB_{1,0}^{[0]},\cB_{0,1}^{[-1]}]  f_{2}^+\big) \, , \\
\notag
 \mathrm{III\,c}\; & := \big({\cB}_{1,0}^{[0]}  f_0^-\, , \, (\uno-P_0) \mathcal{P} [{\cal B}_{0,1}^{[-1]} ,\cB_{0,1}^{[-1]}]  f_{2}^+\big) \, , \\
\notag
  \mathrm{III\,d}\; & := 
  \ \big({\cal B}_{0,1}^{[+1]}  \mathcal{P}[\cB_{0,1}^{[+1]},\cB_{1,0}^{[0]}]  f_0^-\, , \,f_{2}^+\big) =
 \mathrm{III\,c}\; +\big({\cB}_{1,0}^{[0]}  f_0^-\, , \, P_0 \mathcal{P} [{\cal B}_{0,1}^{[-1]} ,\cB_{0,1}^{[-1]}]  f_{2}^+\big) \, , \\
 \notag
 \mathrm{IV\,a}  & := \big({\cal B}_{1,0}^{[0]} (\uno-P_0)  \mathcal{P}[\cB_{0,2}^{[+2]}]  f_0^-\, , \,f_{2}^+\big) = \big(  \cB_{0,2}^{[+2]}  f_0^-\, , \,\mathcal{P}[ {\cal B}_{1,0}^{[0]}]f_{2}^+\big) \, ,\\ \notag
 \mathrm{IV\,b}\; & := \big({\cal B}_{1,0}^{[0]} (\uno-P_0)   \mathcal{P}[\cB_{0,1}^{[+1]},\cB_{0,1}^{[+1]}]  f_0^-\, , \,f_{2}^+\big) 
 \, ,   \\ \notag
 \mathrm{IV\, c} & :=  \big(\cB_{0,1}^{[+1]} f_0^-, \mathcal{P}[\cB_{0,1}^{[-1]},\cB_{1,0}^{[0]}] f_2^+ \big)  =  \mathrm{IV\,b}\;+  \big({\cal B}_{1,0}^{[0]} P_0   \mathcal{P}[\cB_{0,1}^{[+1]},\cB_{0,1}^{[+1]}]  f_0^-\, , \,f_{2}^+\big) \,,
 \end{align}
where the identities come from
 \eqref{psmirabilis} and \eqref{beta2.aux1}.
The claim follows by \eqref{ordine3},  \eqref{kB evodd}, \eqref{cB evodd}, \eqref{P evodd} together with
\begin{equation}
\label{Psplittano}
\begin{aligned}
&P_{0,1}^{[+1]}= \mathcal{P}[\cB_{0,1}^{[+1]}]\, ,\quad P_{0,2}^{[\pm 2]} =  \mathcal{P}[\cB_{0,2}^{[\pm 2]}]+\mathcal{P}[\cB_{0,1}^{[\pm 1]},\cB_{0,1}^{[\pm 1]}]\, ,  \\
&P_{1,1}^{[\pm 1]}=  \mathcal{P}[\cB_{1,1}^{[\pm 1]}]+\mathcal{P}[\cB_{0,1}^{[\pm 1]},\cB_{1,0}^{[0]}]+\mathcal{P}[\cB_{1,0}^{[0]},\cB_{0,1}^{[\pm 1]}]\, ,
\end{aligned} 
\end{equation}
whence, setting  $\widehat P_2 := (\uno-P_0) P_2 P_0$, one has
\begin{align} \notag
\im \beta_2 \delta\e^2 &= \big( \kB_{1,2}^{[2]} f_0^-, f_2^+ \big) \notag \\
& = 
\underbrace{ \big({\cal B}_{1,2}^{[+2]} f_0^-\,,\,f_{2}^+\big)}_{= \Theta}
  +
\frac12 \underbrace{\big({\cal B}_{1,1}^{[+1]}P_{0,1}^{[+1]} f_0^-\, , \,f_{2}^+\big) }_{= \mathrm{I}}
+
 \frac12 \underbrace{ \big({\cal B}_{1,1}^{[+1]} f_0^-\, , \, P_{0,1}^{[-1]}f_{2}^+\big) }_{=\mathrm{III a}} \\ \notag 
&+ 
\frac12 
\underbrace{\big({\cal B}_{0,2}^{[+2]}P_{1,0}^{[0]} f_0^-\, , \,f_{2}^+\big) }_{=\mathrm{II}}
+ 
\frac12 \underbrace{ \big({\cal B}_{0,2}^{[+2]} f_0^-\, , \, P_{1,0}^{[0]}f_{2}^+\big)}_{= \mathrm{ IV a }}
+\ \ 
\frac12 
\underbrace{\big({\cal B}_{0,1}^{[+1]}\widehat P_{1,1}^{[+1]} f_0^-\, , \,f_{2}^+\big) }_{
= \mathrm{ III a}  + \mathrm{ III b} + \mathrm{ III d} \mbox{ by } \eqref{Psplittano},  \eqref{beta2.aux2}
}
\\ \notag
&
+
 \frac12
 \underbrace{ \big({\cal B}_{0,1}^{[+1]} f_0^-\, , \, \widehat P_{1,1}^{[-1]}f_{2}^+\big)}_{=\mathrm{I} + \mathrm{III b + IV\, c}  \mbox{ by } \eqref{Psplittano}, \eqref{beta2.aux2}} \,
+\frac12 
\underbrace{\big({\cal B}_{1,0}^{[0]}\widehat P_{0,2}^{[+2]} f_0^-\, , \,f_{2}^+\big) }_{=\mathrm{IV a} + \mathrm{IV b} }
+ \frac12
\underbrace{\big({\cal B}_{1,0}^{[0]} f_0^-\, , \, \widehat P_{0,2}^{[-2]}f_{2}^+\big)}_{=\mathrm{II} + \mathrm{III c} \mbox{ by } \eqref{Psplittano}
} \\
& = \Theta +  \mathrm{I}+\mathrm{II} + \mathrm{III \, a} + \mathrm{III\, b} + \frac12 (\mathrm{III \, c} + \mathrm{III \, d} ) + \mathrm{IV a} + \frac12 (\mathrm{IV\,b} + \mathrm{IV \, c}) \ . 
\label{beta2.sum}
\end{align}
By \eqref{addendsbeta2},
   \eqref{cBactstot}, Lemma \ref{rem:res} and  finally \eqref{entexp} we obtain 
 \begin{align}\notag
& \Theta =  \ent{1,2}{+2}{2}{0}{+}{-}= -\frac{\im}{\sqrt{3}} \delta \e^2 \, , \ \
\mathrm{I} =  \frac{\ent{0,1}{+1}{1}{0}{+}{-}\ent{1,1}{+1}{2}{1}{+}{+}}{\omega_1^+-\omega_*} - \frac{\ent{0,1}{+1}{1}{0}{-}{-}\ent{1,1}{+1}{2}{1}{+}{-}}{\omega_1^--\omega_*} \,  = \frac{\im }{4 \sqrt{3}}\delta \e^2   \, , \\   \notag
&\mathrm{II}  =     \frac{\ent{1,0}{0}{0}{0}{+}{-}\ent{0,2}{+2}{2}{0}{+}{+}}{\omega_0^+-\omega_*} =\im  \frac{\sqrt{3}}{4} \delta \e^2 \, , \\ \notag
&\mathrm{ III a}  =  
\frac{\ent{1,1}{+1}{1}{0}{+}{-}\ent{0,1}{+1}{2}{1}{+}{+}}{\omega_1^+-\omega_*}-\frac{\ent{1,1}{+1}{1}{0}{-}{-}\ent{0,1}{+1}{2}{1}{+}{-}}{\omega_1^--\omega_*} =  \im \frac{\sqrt{3}}{4} \delta \e^2  \, , \\
\notag
& \mathrm{IIIb} =\frac{\ent{0,1}{+1}{1}{0}{-}{-}\ent{1,0}{0}{1}{1}{-}{-}\ent{0,1}{+1}{2}{1}{+}{-}}{(\omega_1^--\omega_*)^2 } 
 -\frac{\ent{0,1}{+1}{1}{0}{-}{-}\ent{1,0}{0}{1}{1}{+}{-}\ent{0,1}{+1}{2}{1}{+}{+}}{(\omega_1^+-\omega_*)(\omega_1^--\omega_*) }\\ \notag
&\qquad + \frac{\ent{0,1}{+1}{1}{0}{+}{-}\ent{1,0}{0}{1}{1}{+}{+}\ent{0,1}{+1}{2}{1}{+}{+}}{(\omega_1^+-\omega_*)^2 } 
 -\frac{\ent{0,1}{+1}{1}{0}{+}{-}\ent{1,0}{0}{1}{1}{-}{+}\ent{0,1}{+1}{2}{1}{+}{-}}{(\omega_1^+-\omega_*)(\omega_1^--\omega_*) }
 = \im  \frac{\sqrt{3}}{2}  \delta \e^2
 \, , \ \  \\ \notag
 & \mathrm{IIIc}
=
 \frac{\ent{1,0}{0}{0}{0}{+}{-}\ent{0,1}{+1}{1}{0}{+}{+}\ent{0,1}{+1}{2}{1}{+}{+}}{(\omega_0^+-\omega_*)(\omega_1^+-\omega_*)} 
 -\frac{\ent{1,0}{0}{0}{0}{+}{-}\ent{0,1}{+1}{1}{0}{-}{+}\ent{0,1}{+1}{2}{1}{+}{-}}{(\omega_0^+-\omega_*)(\omega_1^--\omega_*)} 
 = - \im \frac{5 \sqrt{3}}{16} \delta \e^2  \, ,\\ \notag
  &\mathrm{III\,d} = 
 \mathrm{III\,c}\; +  \frac{\ent{1,0}{0}{0}{0}{-}{-}\ent{0,1}{+1}{1}{0}{+}{-}\ent{0,1}{+1}{2}{1}{+}{+}}{(\omega_1^+-\omega_*)^2}    -\frac{\ent{1,0}{0}{0}{0}{-}{-}\ent{0,1}{+1}{1}{0}{-}{-}\ent{0,1}{+1}{2}{1}{+}{-}}{(\omega_1^--\omega_*)^2}
    =  - \im \frac{15 \sqrt{3}}{16}  \delta \e^2 
  \, ,\\ 
\notag
& \mathrm{IV a}    = -\frac{ \ent{0,2}{+2}{2}{0}{-}{-}\ent{1,0}{0}{2}{2}{+}{-} }{\omega_2^--\omega_*} =  - \im \frac{1}{4 \sqrt{3}} \delta \e^2\, ,\ \  \\
\notag
 & \mathrm{IV b}   =  \frac{\ent{0,1}{+1}{1}{0}{-}{-}\ent{1}{+1}{2}{1}{-}{-}\ent{1,0}{0}{2}{2}{+}{-}}{(\omega_1^--\omega_*)(\omega_2^--\omega_*)}  - \frac{\ent{0,1}{+1}{1}{0}{+}{-}\ent{0,1}{+1}{2}{1}{-}{+}\ent{1,0}{0}{2}{2}{+}{-}}{(\omega_1^+-\omega_*)(\omega_2^--\omega_*)} = 
  - \im \frac{5}{16 \sqrt{3}} \delta \e^2 \, ,\ \  \\
  \notag
 &\mathrm{IV c} = 
  \mathrm{IV b } + \frac{\ent{0,1}{+1}{1}{0}{-}{-}\ent{0,1}{+1}{2}{1}{+}{-}\ent{1,0}{0}{2}{2}{+}{+}}{(\omega_1^--\omega_*)^2} -\frac{\ent{0,1}{+1}{1}{0}{+}{-}\ent{0,1}{+1}{2}{1}{+}{+}\ent{1,0}{0}{2}{2}{+}{+}}{(\omega_1^+-\omega_*)^2}
= -\im \frac{5 \sqrt{3}}{16} \delta \e^2\,.
\end{align}
 By inserting the resulting values in \eqref{beta2.sum} we obtain that $\beta_2 = -\frac{1}{2\sqrt{3}}$ as claimed in \eqref{epsilonspento2}.
 \smallskip

We now conclude with the quartic term $\beta_3 $.\\
\noindent{\bf Computation of $\beta_3$}. 
In this part, to simplify notation, we shall simply denote  
\begin{equation}\label{notazionefinale}
\cB_{0,n} \equiv \cB_n\, ,\qquad P_{0,n} \equiv P_n\, ,\qquad \ent{0,n}{\kappa}{j}{j'}{\sigma}{\sigma'} \equiv \ent{n}{\kappa}{j}{j'}{\sigma}{\sigma'}\, , 
\end{equation}
since all the operators that enter in the computation of $\beta_3$ have this form. 
 By \eqref{ordine4} 
we have
\begin{align}\label{tB4}
 \im \beta_3 \e^4 &=  \big(\mathfrak{B}_{0,4} f_0^-,f_2^+ \big) =    \big({\cal B}_{4} f_0^-,f_2^+ \big)  + \frac12 \big({\cal B}_{3} P_1 f_0^-,f_2^+ \big) + \frac12 \big({\cal B}_{3}  f_0^-, P_1 f_2^+ \big)\\ 
 &+\frac12 \big({\cal B}_{2} \widehat P_{2} f_0^-,f_2^+ \big)+\frac12 \big({\cal B}_{2}  f_0^-, \widehat P_{2} f_2^+ \big)+\frac12 \big({\cal B}_{1}\widehat P_{3} f_0^-,  f_2^+ \big)+\frac12 \big({\cal B}_{1} f_0^-,\widehat P_{3}  f_2^+ \big) \notag \\
 &-\frac12\big({\cal B}_{1} P_{1} P_0 P_{2} f_0^-,f_2^+ \big)-\frac12 \big({\cal B}_{1}  f_0^-,P_{1} P_0 P_{2} f_2^+ \big) \, ,  \notag
\end{align} 
 where, for brevity, we denote
 \begin{equation}\label{P2hatP3hat}
     \widehat P_{2} := (\uno-P_0) P_{2} P_0\, ,\quad \textup{and}\quad \widehat P_{3} := (\uno-P_0) P_{3} P_0\, .
 \end{equation}
  Each  scalar product in \eqref{tB4} is a   sum of terms of the form (recall \eqref{notazione})
  $$
\big( \cB_{\ell_0}^{[\kappa_0]} \, \mathcal{P}[\cB_{\ell_1}^{[\kappa_1]}, \ldots, \cB_{\ell_p}^{[\kappa_p]}]  f_0^-,  
\mathcal{P}[\cB_{\ell_1}^{[\kappa_{p+1}]}, \ldots, \cB_{\ell_{p+q}}^{[\kappa_{p+q}}] f_2^+ 
\big) 
$$
which vanishes  unless 
$\kappa_0 + \kappa_1 + \ldots +\kappa_p - \kappa_{p+1}-\kappa_{p+q} =  2$.
To compute the non zero scalar products, we 
 represent the action of the nine operators  in \eqref{tB4} --i.e. ${\cal B}_4$, ${\cal B}_3P_1$, $P_1^*{\cal B}_3$, ${\cal B}_2\widehat P_2$, $\widehat P_2^*{\cal B}_2$, ${\cal B}_2\widehat P_2$, $\widehat P_2^*{\cal B}_2$, $ {\cal B}_1 P_1 P_0 P_2$, $ P_2^* P_0^* P_1^* {\cal B}_1  $-- on the eigenfunctions $f_{2}^+,\,f_{0}^- $ in \eqref{def:fsigmaj} through the following graphs, where we only highlight   the paths connecting $f_0^-$ to $f_2^+$, with the corresponding order in $\delta^i \e^j$ (recall formulas 	\eqref{Bsani}, \eqref{Psani}, and also \eqref{cB evodd}, \eqref{P evodd}, \eqref{Pellk*}):

\footnotesize
 \begin{equation}\label{harmonicgraphB4}{ \footnotesize 
 \begin{array}{ccc}
 \begin{tikzcd}[column sep = large, row sep=large]
f_{2}^+  \arrow[r, hookrightarrow] & {\color{purple} 2} \arrow[r,dash]   \arrow[d, dash]
& {\color{purple} 2} \arrow[d, dash] \\
\boxed{{\cal B}_4} 
&1 \arrow[r,dash] \arrow[d, dash]
& 1 \arrow[d, dash]\\
f_{0}^- \arrow[r, hookrightarrow] &{\color{purple} 0} \arrow[r, dash]    \arrow[uur,cyan,bend left,"\e^4\!\!\!\!\!\!\!\!\!\!\!\!\!\!\!\!\!\!\!\!"]   
& {\color{purple} 0} 
\end{tikzcd}  & \qquad &
\begin{tikzcd}
f_{2}^+  \arrow[r, hookrightarrow] & {\color{purple} 2} \arrow[d, dash]  \arrow[r,dash] 
& {\color{purple} 2} \arrow[r,red]  \arrow[d, dash] & {\color{purple} 2} \arrow[d, dash] \\
\boxed{{\cal B}_2\widehat P_2,\widehat P_2^*{\cal B}_2} 
 &1  \arrow[r,dash]  \arrow[d, dash]     &1 \arrow[ur,blue]  \arrow[r,dash]    \arrow[d, dash] 
& 1 \arrow[d, dash]\\
f_{0}^- \arrow[r, hookrightarrow] &{\color{purple} 0} \arrow[r, red]  \arrow[ur, blue, "\delta\e"] \arrow[uur,red,bend left,"\e^2\!\!\!\!\!"] 
& {\color{purple} 0} \arrow[r, dash]   \arrow[uur,red,bend right]  & {\color{purple} 0} 
\end{tikzcd}
\end{array} } \normalsize
\end{equation}
$$
{\footnotesize
\begin{array}{ccc}
\begin{tikzcd}
f_{2}^+  \arrow[r, hookrightarrow] & {\color{purple} 2} \arrow[d, dash]  \arrow[r,dash] 
& {\color{purple} 2} \arrow[r,dash]   \arrow[d, dash] & {\color{purple} 2} \arrow[d, dash] \\
\boxed{{\cal B}_3P_1,\widehat P_3^*{\cal B}_1} 
 &1  \arrow[r,dash]  \arrow[d, dash]     &1 \arrow[ur, violet]  \arrow[r,dash]    \arrow[d, dash] 
& 1 \arrow[d, dash]\\
f_{0}^- \arrow[r, hookrightarrow] &{\color{purple} 0}  \arrow[d, dash]   \arrow[r, green,"\delta"]  \arrow[ur, orange,"\e"] \arrow[dr, orange] 
& {\color{purple} 0} \arrow[r, dash] \arrow[d, dash]   \arrow[uur,brown,"\delta\e^2\!\!\!"]  & {\color{purple} 0}   \arrow[d, dash]  \\
& -1  \arrow[r, dash]   & -1  \arrow[r, dash] \arrow[uuur,violet,bend right,"\e^3"]  & -1 
\end{tikzcd}
&  &
\begin{tikzcd}
& 3 \arrow[r,dash] \arrow[d,dash] & 3  \arrow[r,dash] \arrow[d,dash] \arrow[dr,orange,"\e"]  & 3 \arrow[d,dash]\\ 
f_{2}^+  \arrow[r, hookrightarrow] 
& {\color{purple} 2} \arrow[r,dash]  \arrow[d, dash]& {\color{purple} 2} \arrow[d, dash]  \arrow[r,green,"\delta"]  & {\color{purple} 2} \arrow[d, dash] \\
\boxed{{\cal B}_1\widehat P_3, P_1^*{\cal B}_3} 
    &1   \arrow[r,dash]    \arrow[d, dash] &1 \arrow[ur, orange]  \arrow[r,dash]  \arrow[d, dash]  
& 1 \arrow[d, dash]\\
f_{0}^- \arrow[r, hookrightarrow] 
& {\color{purple} 0} \arrow[r, dash]  \arrow[ur, violet] \arrow[uur,brown,"\delta\e^2"] \arrow[uuur,violet,bend left,"\e^3"] &{\color{purple} 0} \arrow[r, dash]    & {\color{purple} 0} 
\end{tikzcd}
\end{array}}
$$
$$ {\footnotesize
\begin{array}{cc} 
\begin{tikzcd}
&  & 3  \arrow[r,dash] \arrow[d,dash] & 3  \arrow[r,dash] \arrow[d,dash] \arrow[dr,orange]  & 3 \arrow[d,dash]\\ 
f_{2}^+  \arrow[r, hookrightarrow] 
& {\color{purple} 2} \arrow[r,dash]  \arrow[d, dash]& {\color{purple} 2} \arrow[r,dash] \arrow[ur,orange,"\e"] \arrow[dr,orange] \arrow[r,green,"\delta"] \arrow[d, dash] & {\color{purple} 2} \arrow[d, dash]  \arrow[r,green]  & {\color{purple} 2} \arrow[d, dash] \\
\boxed{\cB_1 P_1 P_0 P_2 } 
    &1   \arrow[r,dash]    \arrow[d, dash] &1   \arrow[r,dash]    \arrow[d, dash] &1 \arrow[ur, orange]  \arrow[r,dash]  \arrow[d, dash]  
& 1 \arrow[d, dash]\\
f_{0}^- \arrow[r, hookrightarrow] 
& {\color{purple} 0} \arrow[r, red]   \arrow[uur,red,bend left,"\e^2\!\!\!\!\!\!\!\!\!\!\!\!\!\!\!"] &{\color{purple} 0} \arrow[r, dash] \arrow[ur,orange] &{\color{purple} 0} \arrow[r, dash]    & {\color{purple} 0} 
\end{tikzcd}  &
\begin{tikzcd}
f_{2}^+  \arrow[r, hookrightarrow] & {\color{purple} 2} \arrow[d, dash]  \arrow[r,dash] & {\color{purple} 2} \arrow[d, dash]  \arrow[r,dash] 
& {\color{purple} 2} \arrow[r,red, "\e^2"]   \arrow[d, dash] & {\color{purple} 2} \arrow[d, dash] \\
\boxed{ P_2^* P_0^* P_1^* \cB_1 } 
 &1  \arrow[r,dash]  \arrow[d, dash]   &1  \arrow[r,dash]  \arrow[d, dash]    \arrow[ur,orange] \arrow[dr,orange]   &1   \arrow[r,dash]    \arrow[d, dash] 
& 1 \arrow[d, dash]\\
f_{0}^- \arrow[r, hookrightarrow] &{\color{purple} 0}  \arrow[d, dash]   \arrow[r, green,"\delta"]  \arrow[ur, orange,"\e"] \arrow[dr, orange] 
& {\color{purple} 0} \arrow[r, green] \arrow[d, dash]   & {\color{purple} 0}  \arrow[r,dash] \arrow[d, dash] \arrow[uur, bend right, red] & {\color{purple} 0}  \\
& -1  \arrow[r, dash] & -1 \arrow[ur,orange] \arrow[r, dash]   & -1    & 
\end{tikzcd}
\end{array}}
$$
\normalsize
In particular, by \eqref{Psani}, the only nontrivial terms of order $\e^4$ involve, besides \eqref{Psplittano},
the terms
% $P_2^{[0]}$, $P_3^{[\pm 1]}$, $P_3^{[\pm 3]}$ that expand as
\begin{align}
\label{nodelta}
&% P_1^{[\pm1]}=\mathcal{P}[\cB_1^{[\pm]}] \,, \quad
P_2^{[0]} = \mathcal{P}[\cB_2^{[0]}]+\mathcal{P}[\cB_1^{[+1]},\cB_1^{[-1]}]+\mathcal{P}[\cB_1^{[-1]},\cB_1^{[+1]}]\, ,\\ \notag
& P_3^{[\pm 1]} = \mathcal{P}[\cB_1^{[\mp 1]},\cB_1^{[\pm 1]},\cB_1^{[\pm 1]}]+\mathcal{P}[\cB_1^{[\pm 1]},\cB_1^{[\mp 1]},\cB_1^{[\pm 1]}]+\mathcal{P}[\cB_1^{[\pm 1]},\cB_1^{[\pm 1]},\cB_1^{[\mp 1]}]     \\
\notag
&\ +\mathcal{P}[\cB_2^{[0]},\cB_1^{[\pm 1]}]+\mathcal{P}[\cB_1^{[\pm 1]},\cB_2^{[0]}] +\mathcal{P}[\cB_2^{[\pm 2]},\cB_1^{[\mp 1]}] +\mathcal{P}[\cB_1^{[\mp 1]},\cB_2^{[\pm 2]}]+\mathcal{P}[\cB_3^{[\pm 1]}]\,  ,\\
\notag
& P_3^{[\pm 3]} = \mathcal{P}[\cB_3^{[\pm 3]}]+\mathcal{P}[\cB_2^{[\pm 2]},\cB_1^{[\pm 1]}]+\mathcal{P}[\cB_1^{[\pm 1]},\cB_2^{[\pm 2]}]+\mathcal{P}[\cB_1^{[\pm 1]},\cB_1^{[\pm 1]},\cB_1^{[\pm 1]}]\, .
\end{align}
Indeed, using in \eqref{qui}-\eqref{equi} that  $P_0P_3^{[\pm 1]}P_0= P_0P_3^{[\pm 3]}P_0=0$, 
\begin{align} \label{beta3formula}
&\im \e^4 \beta_3 := 
\underbrace{  \big({\cal B}_4^{[+2]} f_0^-\,,\,f_{2}^+\big) }_{= \mathit \Theta} + \frac12 \underbrace{ \big({\cal B}_3^{[+1]}P_1^{[+1]} f_0^-\, , \,f_{2}^+\big) }_{= \mathit{I}}
+ \frac12
\underbrace{ \big({\cal B}_3^{[+3]}P_1^{[-1]} f_0^-\, , \,f_{2}^+\big)}_{=\mathit{II}} \\ \notag
&\quad  
+\frac12
\underbrace{ \big({\cal B}_3^{[+1]} f_0^-\, , \, P_1^{[-1]}f_{2}^+\big) }_{=\mathit{V \, a}}
+\frac12 
\underbrace{\big({\cal B}_3^{[+3]} f_0^-\, , \, P_1^{[+1]}f_{2}^+\big)}_{=\mathit{VI \, a}}
\\  \notag
 &\quad 
   +\frac12 \underbrace{  \big({\cal B}_2^{[0]}(\uno-P_0) P_2^{[+2]} f_0^-\, , \,f_{2}^+\big) }_{ = \mathit{III\,a} + \mathit{III \, b} \, \mbox{(by } \eqref{Psplittano} \mbox{)}}
   + \frac12
   \underbrace{\big({\cal B}_2^{[+2]} (\uno-P_0) P_{2}^{[0]} f_0^-\, , \,f_{2}^+\big)  }_{= \mathit{IV \, a} + \mathit{IV \,b} + \mathit{IV\,c}  \mbox{ (by } \eqref{nodelta} \mbox{)}}
   \\ \notag
&\quad +\frac12 
\underbrace{\big({\cal B}_2^{[0]} f_0^-\, , \,(\uno-P_0)
 P_2^{[-2]}f_{2}^+\big)}_{= \mathit{IV\, a}+\mathit{V\, c} \mbox{ (by } \eqref{Psplittano} \mbox{)} }
 +\frac12 
 \underbrace{\big({\cal B}_2^{[+2]} f_0^-\, , \, (\uno-P_0)
   P_{2}^{[0]}f_{2}^+\big) }_{=\mathit{III\,a} + \mathit{V\,e} + \mathit{VI \, c}  \mbox{ (by } \eqref{nodelta} \mbox{)} }  \\ 
\label{qui}
&\quad  
 + \frac12
 \underbrace{ \big({\cal B}_1^{[+1]}  P_3^{[+1]} f_0^-\, , \,f_{2}^+\big)}_{= \mathit{V\,a} + \mathit{V\,b}+ 
\widehat{ \mathit{V \, c} }+ \mathit{V \, d} + \widehat{\mathit{V \, e}}+
\mathit{V\,f} + \mathit{V\,g} + \mathit{V\,h}   
 \mbox{ (by }  \eqref{nodelta} \mbox{)} }
+ \frac12 
\underbrace{\big({\cal B}_1^{[-1]} P_3^{[+3]} f_0^-\, , \,f_{2}^+\big)}_{ = \mathit{VI\,a} + \mathit{VI \,b} + \widehat{\mathit{VI\,c}} + \mathit{VI\,d} \mbox{ (by } \eqref{nodelta} \mbox{)}}
\\ \label{equi}
& \quad  +\frac12 
\underbrace{\big({\cal B}_1^{[+1]} f_0^-\, , \,   P_3^{[-1]}f_{2}^+\big)}_{ = \mathit{I} +\mathit{V\, b}
+ \widehat{\mathit{III \, b}} + \mathit{VI\,b} + 
\widehat{\mathit{IV\,c}} + 
\mathit{V\,g} + \mathit{V\,f} +  \mathit{VI\,d}
 \mbox{ (by } \eqref{nodelta} \mbox{)} }
+\frac12
\underbrace{
 \big({\cal B}_1^{[-1]} f_0^-\, , \,  P_3^{[-3]} f_{2}^+\big) }_{ = 
 \mathit{II} + 
 \mathit{V\,d}+
 \widehat{\mathit{IV\,b} } + \mathit{V \, h}
  \mbox{ (by } \eqref{nodelta} \mbox{)} }
\\ 
&  \displaystyle \left. \begin{aligned}
&  - \frac12 \big({\cal B}_1^{[+1]} P_1^{[+1]} P_0 P_2^{[0]} f_0^-\, , \, f_{2}^+\big) - \frac12 \big({\cal B}_1^{[-1]} P_1^{[+1]} P_0 P_2^{[+2]} f_0^-\, , \, f_{2}^+\big)  \\
& - \frac12 \big({\cal B}_1^{[+1]} P_1^{[-1]} P_0 P_2^{[+2]} f_0^-\, , \, f_{2}^+\big) - \frac12 \big({\cal B}_1^{[+1]}  f_0^-\, , \, P_1^{[-1]} P_0  P_{2}^{[0]} f_{2}^+\big)\\
& - \frac12 \big({\cal B}_1^{[-1]}  f_0^-\, , \, P_1^{[-1]} P_0 P_2^{[-2]} f_{2}^+\big) - \frac12 \big({\cal B}_1^{[+1]}  f_0^-\, , \, P_1^{[+1]} P_0 P_2^{[-2]} f_{2}^+\big)
\end{aligned}\right\}=: \mathbf{L}\, , \label{def:ellone} 
\end{align}
namely
\begin{align}\notag
 &\im \e^4 \beta_3 = {\mathit \Theta}+ \mathit{I} + \mathit{II}+\mathit{III\, a}+\tfrac12 \mathit{III\, b}+\tfrac12 \widehat{\mathit{III\, b}}+\mathit{IV\, a}+\tfrac12 \mathit{IV\, b}+\tfrac12 \widehat{\mathit{IV\, b}}+\tfrac12 \mathit{IV\, c}+\tfrac12 \widehat{\mathit{IV\, c}} \\  \notag
&\quad +\mathit{V\, a} + \mathit{V\, b}+\tfrac12\mathit{V\, c}+\tfrac12 \widehat{\mathit{V\, c}}+ \mathit{V\, d}+\tfrac12\mathit{V\, e}+\tfrac12 \widehat{\mathit{V\, e}}+\mathit{V\, f} + \mathit{V\, g}+ \mathit{V\, h}\\ \label{beta3expfin} 
&\quad +\mathit{VI\, a} + \mathit{VI\, b}+\tfrac12\mathit{VI\, c}+\tfrac12 \widehat{\mathit{VI\, c}}+ \mathit{VI\, d}+ \mathbf{L}
\end{align}
with, by \eqref{psmirabilis},
\begin{align}
\notag
{\mathit \Theta} &:=\big({\cal B}_4^{[+2]} f_0^-\,,\,f_{2}^+\big)\, ,\\
\notag
\mathit{I} &:= \big({\cal B}_3^{[+1]}P_1^{[+1]} f_0^-\, , \,f_{2}^+\big) 
\stackrel{\eqref{Psplittano},\eqref{psmirabilis}}{=} \big(\cB_1^{[+1]} f_0^-\, , \, \mathcal{P}[{\cal B}_3^{[-1]}] f_{2}^+\big)\, ,\\ \notag
\mathit{II} &:= \big({\cal B}_3^{[+3]}P_1^{[-1]} f_0^-\, , \,f_{2}^+\big) 
\stackrel{\eqref{Psplittano}}{=}
 \big(\cB_1^{[-1]} f_0^-\, , \, \mathcal{P}[{\cal B}_3^{[-3]}]f_{2}^+\big)  \, ,\\ \notag
\mathit{III\, a} &:= \big({\cal B}_2^{[0]} (\uno-P_0) \mathcal{P}[\cB_2^{[+2]}] f_0^-\, , \,f_{2}^+\big) 
\stackrel{\eqref{beta2.aux1}}{=} \big({\cal B}_2^{[+2]} f_0^-\, ,(\uno-P_0) \mathcal{P}[{\cal B}_2^{[0]}]\,f_{2}^+\big) \, , \\  \notag
\mathit{III\, b} &:= \big({\cal B}_2^{[0]} (\uno-P_0) \mathcal{P}[\cB_1^{[+1]},\cB_1^{[+1]}] f_0^-\, , \,f_{2}^+\big) \, , \\ 
\notag
\widehat{\mathit{III\, b}} & :=  \big({\cal B}_1^{[+1]} f_0^-\, , \mathcal{P}[\cB_1^{[-1]},{\cal B}_2^{[0]}]\,f_{2}^+\big) = \mathit{III\, b}
+
\big({\cal B}_2^{[0]} P_0 \mathcal{P}[\cB_1^{[+1]},\cB_1^{[+1]}] f_0^-\, , \,f_{2}^+\big)
 \, ,\\ \notag
\mathit{IV\, a} &:= \big({\cal B}_2^{[+2]} (\uno-P_0)  \mathcal{P}[\cB_2^{[0]}] f_0^-\, , \,f_{2}^+\big)  \stackrel{\eqref{beta2.aux1}}{=} { \big(\cB_2^{[0]} f_0^-\, , \, (\uno-P_0) \mathcal{P}[\cB_2^{[-2]}]f_2^+ \big)} \, ,\\ \notag
\mathit{IV\, b} &:= \big({\cal B}_2^{[+2]} (\uno-P_0) \mathcal{P}[\cB_1^{[+1]},\cB_1^{[-1]}] f_0^-\, , \,f_{2}^+\big) \, ,\\ \notag
\mathit{IV\, c} &:= \big({\cal B}_2^{[+2]} (\uno-P_0) \mathcal{P}[\cB_1^{[-1]},\cB_1^{[+1]}] f_0^-\, , \,f_{2}^+\big) \, ,  \\ 
\notag
\widehat{\mathit{IV\,b}}&:=  \big({\cal B}_1^{[-1]} f_0^-\, , \mathcal{P}[\cB_1^{[-1]},{\cal B}_2^{[-2]}]\,f_{2}^+\big)=\mathit{IV\, b}+{\big({\cal B}_2^{[+2]} P_0 \mathcal{P}[\cB_1^{[+1]},\cB_1^{[-1]}] f_0^-\, , \,f_{2}^+\big)} \, ,\\ \notag
\widehat{\mathit{IV \, c}} &:= \big({\cal B}_1^{[+1]} f_0^-\, , \,   \mathcal{P}[\cB_1^{[+1]},\cB_2^{[-2]}] f_{2}^+\big) =\mathit{IV \, c} +\big({\cal B}_2^{[+2]} P_0 \mathcal{P}[\cB_1^{[-1]},\cB_1^{[+1]}] f_0^-\, , \,f_{2}^+\big)\, ,\\ \notag 
\mathit{V\,a} &:= \big(\cB_1^{[+1]}\mathcal{P}[\cB_3^{[+1]}]f_0^-\, , \, f_2^+ \big) \stackrel{\eqref{psmirabilis}}{=}\big(\cB_3^{[+1]}f_0^-\, , \, \mathcal{P}[\cB_1^{[-1]}]f_2^+ \big) \, , \\ \notag
\mathit{V\,b} &:= \big(\cB_1^{[+1]}\mathcal{P}[\cB_2^{[0]},\cB_1^{[+1]}] f_0^-\, , \, f_2^+ \big) = \big(\cB_1^{[+1]}f_0^-\, , \, \mathcal{P}[\cB_2^{[0]},\cB_1^{[-1]}]f_2^+ \big) \, , \\ \notag
\mathit{V\,c} &:= \big(\cB_2^{[0]} f_0^-\, , \, (\uno-P_0) \mathcal{P}[\cB_1^{[-1]},\cB_1^{[-1]}]f_2^+ \big) ,\\ \notag
 \widehat{\mathit{V\,c} } & :=  \big(\cB_1^{[+1]}\mathcal{P}[\cB_1^{[+1]},\cB_2^{[0]}] f_0^-\, , \, f_2^+ \big) = \mathit{V\,c}+  \big(\cB_2^{[0]} f_0^-\, , \,P_0\mathcal{P}[\cB_1^{[-1]},\cB_1^{[-1]}]f_2^+ \big)  \,, \\ \notag
\mathit{V\,d} &:= \big(\cB_1^{[+1]}\mathcal{P}[\cB_2^{[+2]},\cB_1^{[-1]}] f_0^-\, , \, f_2^+ \big) = \big(\cB_1^{[-1]} f_0^-\, , \, \mathcal{P}[\cB_2^{[-2]},\cB_1^{[-1]}] f_2^+ \big) \, , \\ \notag
\mathit{V\,e} &:= \big(\cB_2^{[+2]} f_0^-\, , \, (\uno-P_0) \mathcal{P}[\cB_1^{[+1]},\cB_1^{[-1]}] f_2^+ \big) \, , \\ \notag
\widehat{\mathit{V\, e}} &:= \big(\cB_1^{[+1]}\mathcal{P}[\cB_1^{[-1]},\cB_2^{[+2]}] f_0^-\, , \, f_2^+ \big)  \stackrel{\eqref{psmirabilis}}{=} \mathit{V\,e}+
\big(\cB_2^{[+2]} f_0^-\, , \, P_0 \mathcal{P}[\cB_1^{[+1]},\cB_1^{[-1]}] f_2^+ \big)  \, ,\\ \notag
\mathit{V\, f} &:= \big(\cB_1^{[+1]}\mathcal{P}[\cB_1^{[-1]},\cB_1^{[+1]},\cB_1^{[+1]}] f_0^-\, , \, f_2^+ \big) = {\big(\cB_1^{[+1]} f_0^-\, , \, \mathcal{P}[\cB_1^{[-1]},\cB_1^{[+1]},\cB_1^{[-1]}] f_2^+ \big)} \, , \\ \notag
\mathit{V\, g} &:= \big(\cB_1^{[+1]}\mathcal{P}[\cB_1^{[+1]},\cB_1^{[-1]},\cB_1^{[+1]}] f_0^-\, , \, f_2^+ \big) =   {\big(\cB_1^{[+1]} f_0^-\, , \,\mathcal{P}[\cB_1^{[+1]},\cB_1^{[-1]},\cB_1^{[-1]}] f_2^+ \big)} \, , \\ \notag
\mathit{V\, h} &:= \big(\cB_1^{[+1]}\mathcal{P}[\cB_1^{[+1]},\cB_1^{[+1]},\cB_1^{[-1]}] f_0^-\, , \, f_2^+ \big) ={ \big(\cB_1^{[-1]} f_0^-\, , \, \mathcal{P}[\cB_1^{[-1]},\cB_1^{[-1]},\cB_1^{[-1]}] f_2^+ \big)} \, , \\ \notag
\mathit{VI\, a} &:= \big(\cB_1^{[-1]}\mathcal{P}[\cB_3^{[+3]}]f_0^-\, , \, f_2^+ \big) 
\stackrel{\eqref{psmirabilis}}{=}
 \big(\cB_3^{[+3]}f_0^-\, , \, \mathcal{P}[\cB_1^{[+1]}] f_2^+ \big) \, , \\ \notag
\mathit{VI\, b} &:= \big(\cB_1^{[-1]}\mathcal{P}[\cB_2^{[+2]},\cB_1^{[+1]}]f_0^-\, , \, f_2^+ \big) = \big(\cB_1^{[+1]}f_0^-\, , \, \mathcal{P}[\cB_2^{[-2]},\cB_1^{[+1]}] f_2^+ \big) \, ,  \\ \notag
\mathit{VI\, c} &:= \big(\cB_2^{[+2]}f_0^-\, , \, (\uno-P_0) \mathcal{P}[\cB_1^{[-1]},\cB_1^{[+1]}] f_2^+ \big) \, , \\ 
\notag
\widehat{\mathit{VI\, c}} &:=
\big(\cB_1^{[-1]}\mathcal{P}[\cB_1^{[+1]},\cB_2^{[+2]}]f_0^-\, , \, f_2^+ \big)  
 = 
\mathit{VI\, c} +   \big(\cB_2^{[+2]}f_0^-\, , \, P_0 \mathcal{P}[\cB_1^{[-1]},\cB_1^{[+1]}] f_2^+ \big)
 \, ,\\ \notag 
\mathit{VI\, d} &:= \big(\cB_1^{[-1]}\mathcal{P}[\cB_1^{[+1]},\cB_1^{[+1]},\cB_1^{[+1]}] f_0^-\, , \, f_2^+ \big) = \big(\cB_1^{[+1]} f_0^-\, , \, \mathcal{P}[\cB_1^{[-1]},\cB_1^{[-1]},\cB_1^{[+1]}] f_2^+ \big) \, .
\end{align} 
By Lemma \ref{actionofL},  Remark \ref{remarkinoid}, Lemma \ref{rem:res} and \eqref{entprop}
we represent these expressions in terms of the entanglement coefficients; we then compute them by means of Lemma \ref{lem:entexp}. 
We have, recalling the notation $\ent{0,n}{\kappa}{j}{j'}{\sigma}{\sigma'} \equiv \ent{n}{\kappa}{j}{j'}{\sigma}{\sigma'}$,
\begin{align}\label{beta3.coeff}
& {\mathit \Theta} =\ent{4}{+2}{2}{0}{+}{-} = 0\, ,\\
\notag
& \mathit{I}
 = \frac{\ent{1}{+1}{1}{0}{+}{-}\ent{3}{+1}{2}{1}{+}{+}}{\omega_1^+-\omega_*}-\frac{\ent{1}{+1}{1}{0}{-}{-}\ent{3}{+1}{2}{1}{+}{-}}{\omega_1^--\omega_*} =\im \frac{5  \sqrt{3}}{32} \e^4 \, , \\
 \notag
 & \mathit{II} = \frac{\ent{1}{-1}{-1}{0}{+}{-}\ent{3}{+3}{2}{-1}{+}{+}}{\omega_{-1}^+-\omega_*}-\frac{\ent{1}{-1}{-1}{0}{-}{-}\ent{3}{+3}{2}{-1}{+}{-}}{\omega_{-1}^--\omega_*} =- \frac{9 \im \sqrt{3}}{32}  \e^4  \, ,\\
 \notag
 & \mathit{III\, a} \stackrel{\eqref{beta2.aux1}}{=}   - \frac{\ent{2}{+2}{2}{0}{-}{-}\ent{2}{0}{2}{2}{+}{-}}{\omega_2^--\omega_*} 
 =\im \frac{3\sqrt{3}}{8} \e^4\, ,\\
 \notag
 & \mathit{III \, b} =  \frac{\ent {1}{+1}{1}{0}{-}{-}\ent{1}{+1}{2}{1}{-}{-}\ent{2}{0}{2}{2}{+}{-}}{(\omega_1^--\omega_*)(\omega_2^--\omega_*)} -\frac{\ent{1}{+1}{1}{0}{+}{-}\ent{1}{+1}{2}{1}{-}{+}\ent{2}{0}{2}{2}{+}{-}}{(\omega_1^+-\omega_*)(\omega_2^--\omega_*)} = \im \frac{15  \sqrt{3}}{32}  \e^4 \, , \\
 \notag
 &
  {\widehat{\mathit{III\, b}}-\mathit{III\, b}} = \frac{\ent{1}{+1}{1}{0}{-}{-}\ent{1}{+1}{2}{1}{+}{-}\ent{2}{0}{2}{2}{+}{+}}{(\omega_1^--\omega_*)^2} -\frac{\ent{1}{+1}{1}{0}{+}{-}\ent{1}{+1}{2}{1}{+}{+}\ent{2}{0}{2}{2}{+}{+}}{(\omega_1^+-\omega_*)^2}= -\im\frac{75  \sqrt{3}}{128}\e^4\, ,\\
  \notag
 & \mathit{IV \, a} \stackrel{\eqref{beta2.aux1}}{=} \frac{\ent{2}{0}{0}{0}{+}{-}\ent{2}{+2}{2}{0}{+}{+}}{\omega_0^+-\omega_*} =-\im \frac{\sqrt 3}{8} \e^4  \, , \\
 \notag
 & \mathit{IV\, b} = \frac{\ent{1}{-1}{-1}{0}{+}{-}\ent{1}{+1}{0}{-1}{+}{+}\ent{2}{+2}{2}{0}{+}{+}}{(\omega_{-1}^+-\omega_*)(\omega_0^+-\omega_*)}-\frac{\ent{1}{-1}{-1}{0}{-}{-}\ent{1}{+1}{0}{-1}{+}{-}\ent{2}{+2}{2}{0}{+}{+}}{(\omega_{-1}^--\omega_*)(\omega_0^+-\omega_*)} = 
 \im \frac{ 3 \sqrt{3}}{64} \e^4  \, ,\\
 \notag
&  
 {\widehat{\mathit{IV\, b}}-\mathit{IV\, b}} := \frac{\ent{1}{-1}{-1}{0}{+}{-}\ent{1}{+1}{0}{-1}{-}{+}\ent{2}{+2}{2}{0}{+}{-}}{(\omega_{-1}^+-\omega_*)^2}-\frac{\ent{1}{-1}{-1}{0}{-}{-}\ent{1}{+1}{0}{-1}{-}{-}\ent{2}{+2}{2}{0}{+}{-}}{(\omega_{-1}^--\omega_*)^2} = 0\, ,\\
  \notag 
&  {\mathit{IV\, c}} = \frac{\ent{1}{+1}{1}{0}{+}{-}\ent{1}{-1}{0}{1}{+}{+}\ent{2}{+2}{2}{0}{+}{+}}{(\omega_1^+-\omega_*)(\omega_0^+-\omega_*)} - \frac{\ent{1}{+1}{1}{0}{-}{-}\ent{1}{-1}{0}{1}{+}{-}\ent{2}{+2}{2}{0}{+}{+}}{(\omega_1^--\omega_*)(\omega_0^+-\omega_*)} = \im\frac{5  \sqrt{3}}{64} \e^4\, ,\\ 
\notag
&   {\widehat{\mathit{IV\, c}}-\mathit{IV\, c}} = \frac{\ent{1}{+1}{1}{0}{+}{-}\ent{1}{-1}{0}{1}{-}{+}\ent{2}{+2}{2}{0}{+}{-}}{(\omega_1^+-\omega_*)^2} - \frac{\ent{1}{+1}{1}{0}{-}{-}\ent{1}{-1}{0}{1}{-}{-}\ent{2}{+2}{2}{0}{+}{-}}{(\omega_1^--\omega_*)^2} = 0 \, ,\\ 
\notag
&\mathit{V\,a} = \frac{\ent{3}{+1}{1}{0}{+}{-}\ent{1}{+1}{2}{1}{+}{+}}{\omega_1^+-\omega_*}- \frac{\ent{3}{+1}{1}{0}{-}{-}\ent{1}{+1}{2}{1}{+}{-}}{\omega_1^--\omega_*} = -\im \frac{5 \sqrt{3}}{32} \e^4 \, ,\\ 
\notag
& \mathit{V\,b} = \frac{\ent{1}{+1}{1}{0}{+}{-}\ent{2}{0}{1}{1}{+}{+}\ent{1}{+1}{2}{1}{+}{+}}{(\omega_1^+-\omega_*)^2} - \frac{\ent{1}{+1}{1}{0}{-}{-}\ent{2}{0}{1}{1}{+}{-}\ent{1}{+1}{2}{1}{+}{+}}{(\omega_1^--\omega_*)(\omega_1^+-\omega_*)} \\ \notag
&+ \frac{\ent{1}{+1}{1}{0}{-}{-}\ent{2}{0}{1}{1}{-}{-}\ent{1}{+1}{2}{1}{+}{-}}{(\omega_1^--\omega_*)^2} - \frac{\ent{1}{+1}{1}{0}{+}{-}\ent{2}{0}{1}{1}{-}{+}\ent{1}{+1}{2}{1}{+}{-}}{(\omega_1^--\omega_*)(\omega_1^+-\omega_*)}  = 
-\im \frac{25 \sqrt{3}}{128} \e^4 \, ,\\ 
\notag
&  {\mathit{V\,c}} = \frac{\ent{2}{0}{0}{0}{+}{-}\ent{1}{+1}{1}{0}{+}{+}\ent{1}{+1}{2}{1}{+}{+} }{(\omega_0^+-\omega_*)(\omega_1^+-\omega_*)}-\frac{\ent{2}{0}{0}{0}{+}{-}\ent{1}{+1}{1}{0}{-}{+}\ent{1}{+1}{2}{1}{+}{-} }{(\omega_0^+-\omega_*)(\omega_1^--\omega_*)} = \im\frac{5  \sqrt{3}}{32} \e^4\, ,\\ 
\notag
&   {\widehat{\mathit{V\, c}}-\mathit{V\, c}} := \frac{\ent{2}{0}{0}{0}{-}{-}\ent{1}{+1}{1}{0}{+}{-}\ent{1}{+1}{2}{1}{+}{+} }{(\omega_1^+-\omega_*)^2}-\frac{\ent{2}{0}{0}{0}{-}{-}\ent{1}{+1}{1}{0}{-}{-}\ent{1}{+1}{2}{1}{+}{-} }{(\omega_1^--\omega_*)^2} = -\im\frac{35  \sqrt{3}}{128}\e^4 \, ,\\ \notag
&  {\mathit{V\,d}} = \frac{\ent{1}{-1}{-1}{0}{+}{-}\ent{2}{+2}{1}{-1}{+}{+}\ent{1}{+1}{2}{1}{+}{+}}{(\omega_{-1}^+-\omega_*)(\omega_1^+-\omega_*)}-\frac{\ent{1}{-1}{-1}{0}{-}{-}\ent{2}{+2}{1}{-1}{+}{-}\ent{1}{+1}{2}{1}{+}{+}}{(\omega_{-1}^--\omega_*)(\omega_1^+-\omega_*)}\\ \notag 
& +\frac{\ent{1}{-1}{-1}{0}{-}{-}\ent{2}{+2}{1}{-1}{-}{-}\ent{1}{+1}{2}{1}{+}{-}}{(\omega_{-1}^--\omega_*)(\omega_1^--\omega_*)}-\frac{\ent{1}{-1}{-1}{0}{+}{-}\ent{2}{+2}{1}{-1}{-}{+}\ent{1}{+1}{2}{1}{+}{-}}{(\omega_{-1}^+-\omega_*)(\omega_1^--\omega_*)} = \im \frac{15 \sqrt{3}}{64} \e^4 \, , \\ \notag
&  {\mathit{V\,e}} = \frac{\ent{2}{+2}{2}{0}{-}{-}\ent{1}{-1}{1}{2}{-}{-}\ent{1}{+1}{2}{1}{+}{-}}{(\omega_2^--\omega_*)(\omega_1^--\omega_*)}-\frac{\ent{2}{+2}{2}{0}{-}{-}\ent{1}{-1}{1}{2}{+}{-}\ent{1}{+1}{2}{1}{+}{+}}{(\omega_2^--\omega_*)(\omega_1^+-\omega_*)} = \im\frac{  15 \sqrt{3}}{64} \e^4\, , \\ \notag
&   {\widehat{\mathit{V\, e}}-\mathit{V\, e}} = \frac{\ent{2}{+2}{2}{0}{+}{-}\ent{1}{-1}{1}{2}{-}{+}\ent{1}{+1}{2}{1}{+}{-}}{(\omega_1^--\omega_*)^2}-\frac{\ent{2}{+2}{2}{0}{+}{-}\ent{1}{-1}{1}{2}{+}{+}\ent{1}{+1}{2}{1}{+}{+}}{(\omega_1^+-\omega_*)^2}= 0 \, ,\\ \notag
 & {\mathit{V\,f}} =-2\frac{\ent{1}{+1}{1}{0}{+}{-}\ent{1}{+1}{2}{1}{+}{+}\ent{1}{-1}{1}{2}{+}{+}\ent{1}{+1}{2}{1}{+}{+}}{(\omega_1^+-\omega_*)^3}-2\frac{\ent{1}{+1}{1}{0}{-}{-}\ent{1}{+1}{2}{1}{+}{-}\ent{1}{-1}{1}{2}{-}{+}\ent{1}{+1}{2}{1}{+}{-}}{(\omega_1^--\omega_*)^3}\\  \notag
&+ \frac{(\omega_1^++\omega_1^--2\omega_*)\ent{1}{+1}{1}{0}{-}{-}\ent{1}{+1}{2}{1}{+}{-}\ent{1}{-1}{1}{2}{+}{+}\ent{1}{+1}{2}{1}{+}{+} }{(\omega_1^+-\omega_*)^2(\omega_1^--\omega_*)^2} \\  \notag 
&+ \frac{(\omega_1^++\omega_1^--2\omega_*)\ent{1}{+1}{1}{0}{+}{-}\ent{1}{+1}{2}{1}{+}{+}\ent{1}{-1}{1}{2}{-}{+}\ent{1}{+1}{2}{1}{+}{-} }{(\omega_1^+-\omega_*)^2(\omega_1^--\omega_*)^2} \\ \notag
&+ \frac{\ent{1}{+1}{1}{0}{+}{-}\ent{1}{+1}{2}{1}{-}{+}\ent{1}{-1}{1}{2}{-}{-}\ent{1}{+1}{2}{1}{+}{-}}{(\omega_1^+-\omega_*)(\omega_2^--\omega_*)(\omega_1^--\omega_*)}  - \frac{\ent{1}{+1}{1}{0}{+}{-}\ent{1}{+1}{2}{1}{-}{+}\ent{1}{-1}{1}{2}{+}{-}\ent{1}{+1}{2}{1}{+}{+}}{(\omega_1^+-\omega_*)^2(\omega_2^--\omega_*)} \\ \notag 
&+ \frac{\ent{1}{+1}{1}{0}{-}{-}\ent{1}{+1}{2}{1}{-}{-}\ent{1}{-1}{1}{2}{+}{-}\ent{1}{+1}{2}{1}{+}{+}}{(\omega_1^--\omega_*)(\omega_2^--\omega_*)(\omega_1^+-\omega_*)} -  \frac{\ent{1}{+1}{1}{0}{-}{-}\ent{1}{+1}{2}{1}{-}{-}\ent{1}{-1}{1}{2}{-}{-}\ent{1}{+1}{2}{1}{+}{-}}{(\omega_1^--\omega_*)^2(\omega_2^--\omega_*)} = - \im \frac{75  \sqrt{3}}{256}  \e^4 \, ,
\\ \notag
&  {\mathit{V\, g}} = 2\frac{\ent{1}{+1}{1}{0}{+}{-}\ent{1}{-1}{0}{1}{-}{+}\ent{1}{+1}{1}{0}{+}{-}\ent{1}{+1}{2}{1}{+}{+}}{(\omega_1^+-\omega_*)^3}+2\frac{\ent{1}{+1}{1}{0}{-}{-}\ent{1}{-1}{0}{1}{-}{-}\ent{1}{+1}{1}{0}{-}{-}\ent{1}{+1}{2}{1}{+}{-}}{(\omega_1^--\omega_*)^3} \\ \notag
&-\frac{(\omega_1^++\omega_1^--2\omega_*)\ent{1}{+1}{1}{0}{+}{-}\ent{1}{-1}{0}{1}{-}{+}\ent{1}{+1}{1}{0}{-}{-}\ent{1}{+1}{2}{1}{+}{-}}{(\omega_1^+-\omega_*)^2(\omega_1^--\omega_*)^2}\\ \notag
&-\frac{(\omega_1^++\omega_1^--2\omega_*)\ent{1}{+1}{1}{0}{-}{-}\ent{1}{-1}{0}{1}{-}{-}\ent{1}{+1}{1}{0}{+}{-}\ent{1}{+1}{2}{1}{+}{+}}{(\omega_1^+-\omega_*)^2(\omega_1^--\omega_*)^2} \\ \notag
&+\frac{\ent{1}{+1}{1}{0}{+}{-}\ent{1}{-1}{0}{1}{+}{+}\ent{1}{+1}{1}{0}{+}{+}\ent{1}{+1}{2}{1}{+}{+}}{(\omega_1^+-\omega_*)^2(\omega_0^+-\omega_*)} - \frac{\ent{1}{+1}{1}{0}{-}{-}\ent{1}{-1}{0}{1}{+}{-}\ent{1}{+1}{1}{0}{+}{+}\ent{1}{+1}{2}{1}{+}{+}}{(\omega_1^+-\omega_*)(\omega_0^+-\omega_*)(\omega_1^--\omega_*)}  \\ \notag
&+ \frac{\ent{1}{+1}{1}{0}{-}{-}\ent{1}{-1}{0}{1}{+}{-}\ent{1}{+1}{1}{0}{-}{+}\ent{1}{+1}{2}{1}{+}{-}}{(\omega_1^--\omega_*)^2(\omega_0^+-\omega_*)} - \frac{\ent{1}{+1}{1}{0}{+}{-}\ent{1}{-1}{0}{1}{+}{+}\ent{1}{+1}{1}{0}{-}{+}\ent{1}{+1}{2}{1}{+}{-}}{(\omega_1^+-\omega_*)(\omega_0^+-\omega_*)(\omega_1^--\omega_*)}=\im \frac{25  \sqrt{3}}{256} \e^4   \, ,  \\ \notag
&  {\mathit{V\, h}}= \frac{(\omega_{-1}^++\omega_1^+-2\omega_*)\ent{1}{-1}{-1}{0}{+}{-}\ent{1}{+1}{0}{-1}{-}{+}\ent{1}{+1}{1}{0}{+}{-}\ent{1}{+1}{2}{1}{+}{+}}{(\omega_{-1}^+-\omega_*)^2(\omega_{1}^+-\omega_*)^2}\\ \notag
&+ \frac{(\omega_{-1}^-+\omega_1^--2\omega_*) \ent{1}{-1}{-1}{0}{-}{-}\ent{1}{+1}{0}{-1}{-}{-}\ent{1}{+1}{1}{0}{-}{-}\ent{1}{+1}{2}{1}{+}{-}}{(\omega_{-1}^--\omega_*)^2(\omega_{1}^--\omega_*)^2} \\ \notag
&- \frac{(\omega_{-1}^-+\omega_1^+-2\omega_*) \ent{1}{-1}{-1}{0}{-}{-}\ent{1}{+1}{0}{-1}{-}{-}\ent{1}{+1}{1}{0}{+}{-}\ent{1}{+1}{2}{1}{+}{+}}{(\omega_{-1}^--\omega_*)^2(\omega_1^+-\omega_*)^2}\\ \notag
&- \frac{(\omega_{-1}^++\omega_1^--2\omega_*) \ent{1}{-1}{-1}{0}{+}{-}\ent{1}{+1}{0}{-1}{-}{+}\ent{1}{+1}{1}{0}{-}{-}\ent{1}{+1}{2}{1}{+}{-}}{(\omega_{-1}^+-\omega_*)^2(\omega_1^--\omega_*)^2}  \\ \notag
&+\frac{\ent{1}{-1}{-1}{0}{+}{-}\ent{1}{+1}{0}{-1}{+}{+}\ent{1}{+1}{1}{0}{+}{+}\ent{1}{+1}{2}{1}{+}{+}}{(\omega_{-1}^+-\omega_*)(\omega_0^+-\omega_*)(\omega_{1}^+-\omega_*)} - \frac{\ent{1}{-1}{-1}{0}{-}{-}\ent{1}{+1}{0}{-1}{+}{-}\ent{1}{+1}{1}{0}{+}{+}\ent{1}{+1}{2}{1}{+}{+}}{(\omega_{-1}^--\omega_*)(\omega_0^+-\omega_*)(\omega_1^+-\omega_*)} \\ \notag
&+ \frac{\ent{1}{-1}{-1}{0}{-}{-}\ent{1}{+1}{0}{-1}{+}{-}\ent{1}{+1}{1}{0}{-}{+}\ent{1}{+1}{2}{1}{+}{-}}{(\omega_{-1}^--\omega_*)(\omega_0^+-\omega_*)(\omega_1^--\omega_*)} - \frac{\ent{1}{-1}{-1}{0}{+}{-}\ent{1}{+1}{0}{-1}{+}{+}\ent{1}{+1}{1}{0}{-}{+}\ent{1}{+1}{2}{1}{+}{-}}{(\omega_{-1}^+-\omega_*)(\omega_0^+-\omega_*)(\omega_1^--\omega_*)} =0
 \, , \\ \notag
 & \mathit{VI\, a} :=   \frac{\ent{3}{+3}{3}{0}{+}{-}\ent{1}{-1}{2}{3}{+}{+}}{\omega_3^+-\omega_*}-\frac{\ent{3}{+3}{3}{0}{-}{-}\ent{1}{-1}{2}{3}{+}{-}}{\omega_3^--\omega_*} = 0\, ,\\ \notag
& \mathit{VI\, b} = \frac{\ent{1}{+1}{1}{0}{+}{-}\ent{2}{+2}{3}{1}{+}{+}\ent{1}{-1}{2}{3}{+}{+}}{(\omega_1^+-\omega_*)(\omega_3^+-\omega_*)} - \frac{\ent{1}{+1}{1}{0}{-}{-}\ent{2}{+2}{3}{1}{+}{-}\ent{1}{-1}{2}{3}{+}{+}}{(\omega_1^--\omega_*)(\omega_3^+-\omega_*)} \\ \notag
&+ \frac{\ent{1}{+1}{1}{0}{-}{-}\ent{2}{+2}{3}{1}{-}{-}\ent{1}{-1}{2}{3}{+}{-}}{(\omega_1^--\omega_*)(\omega_3^--\omega_*)} - \frac{\ent{1}{+1}{1}{0}{+}{-}\ent{2}{+2}{3}{1}{-}{+}\ent{1}{-1}{2}{3}{+}{-}}{(\omega_1^+-\omega_*)(\omega_3^--\omega_*)} = 0\, 
 \, ,\\ \notag
& \mathit{VI\, c}= \frac{\ent{2}{+2}{2}{0}{-}{-}\ent{1}{+1}{3}{2}{-}{-}\ent{1}{-1}{2}{3}{+}{-}}{(\omega_2^--\omega_*)(\omega_3^--\omega_*)} - \frac{\ent{2}{+2}{2}{0}{-}{-}\ent{1}{+1}{3}{2}{+}{-}\ent{1}{-1}{2}{3}{+}{+}}{(\omega_2^--\omega_*)(\omega_3^+-\omega_*)} = -\im \frac{39 \sqrt{3}}{64}\e^4 \, ,\\ \notag
&   {\widehat{\mathit{VI\, c}}-\mathit{VI\, c}} := \frac{\ent{2}{+2}{2}{0}{+}{-}\ent{1}{+1}{3}{2}{-}{+}\ent{1}{-1}{2}{3}{+}{-}}{(\omega_3^--\omega_*)^2} - \frac{\ent{2}{+2}{2}{0}{+}{-}\ent{1}{+1}{3}{2}{+}{+}\ent{1}{-1}{2}{3}{+}{+}}{(\omega_3^+-\omega_*)^2} = 0 \, ,\\ \notag
&  {\mathit{VI\, d}} =\frac{\ent{1}{+1}{1}{0}{-}{-}\ent{1}{+1}{2}{1}{-}{-}\ent{1}{+1}{3}{2}{+}{-}\ent{1}{-1}{2}{3}{+}{+}}{(\omega_1^--\omega_*)(\omega_2^--\omega_*)(\omega_3^+-\omega_*)} -\frac{\ent{1}{+1}{1}{0}{+}{-}\ent{1}{+1}{2}{1}{-}{+}\ent{1}{+1}{3}{2}{+}{-}\ent{1}{-1}{2}{3}{+}{+}}{(\omega_1^+-\omega_*)(\omega_2^--\omega_*)(\omega_3^+-\omega_*)} \\ \notag 
&\quad +\frac{\ent{1}{+1}{1}{0}{+}{-}\ent{1}{+1}{2}{1}{-}{+}\ent{1}{+1}{3}{2}{-}{-}\ent{1}{-1}{2}{3}{+}{-}}{(\omega_1^+-\omega_*)(\omega_2^--\omega_*)(\omega_3^--\omega_*)} -\frac{\ent{1}{+1}{1}{0}{-}{-}\ent{1}{+1}{2}{1}{-}{-}\ent{1}{+1}{3}{2}{-}{-}\ent{1}{-1}{2}{3}{+}{-}}{(\omega_1^--\omega_*)(\omega_2^--\omega_*)(\omega_3^--\omega_*)} \\ \notag
&\quad +\frac{(\omega_1^-+\omega_3^+-2\omega_*)\ent{1}{+1}{1}{0}{-}{-}\ent{1}{+1}{2}{1}{+}{-}\ent{1}{+1}{3}{2}{+}{+}\ent{1}{-1}{2}{3}{+}{+}}{(\omega_1^--\omega_*)^2(\omega_3^+-\omega_*)^2}\\ \notag
&\quad-\frac{(\omega_1^++\omega_3^+-2\omega_*)\ent{1}{+1}{1}{0}{+}{-}\ent{1}{+1}{2}{1}{+}{+}\ent{1}{+1}{3}{2}{+}{+}\ent{1}{-1}{2}{3}{+}{+}}{(\omega_1^+-\omega_*)^2(\omega_3^+-\omega_*)^2} \\ \notag 
&\quad +\frac{(\omega_1^++\omega_3^--2\omega_*)\ent{1}{+1}{1}{0}{+}{-}\ent{1}{+1}{2}{1}{+}{+}\ent{1}{+1}{3}{2}{-}{+}\ent{1}{-1}{2}{3}{+}{-}}{(\omega_1^+-\omega_*)^2(\omega_3^--\omega_*)^2}\\ \notag
&\quad-\frac{(\omega_1^-+\omega_3^--2\omega_*)\ent{1}{+1}{1}{0}{-}{-}\ent{1}{+1}{2}{1}{+}{-}\ent{1}{+1}{3}{2}{-}{+}\ent{1}{-1}{2}{3}{+}{-}}{(\omega_1^--\omega_*)^2(\omega_3^--\omega_*)^2} = \im \frac{195  \sqrt{3}}{256} \e^4
\, .
\end{align}
Finally we compute the last term 
$\mathbf{L}$ in \eqref{def:ellone}. 
By \eqref{Psani}, \eqref{nodelta} and \eqref{beta2.aux1}
\begin{equation}\label{Lcalc}
\begin{aligned}
&P_0 P_2^{[+2]} f_0^- = P_0 \mathcal{P}[\cB_1^{[+1]},\cB_1^{[+1]}] f_0^- = \zeta_1 f_2^+\e^2 \, ,\\
&P_0 P_2^{[-2]} f_2^+ =   P_0 \mathcal{P}[\cB_1^{[-1]},\cB_1^{[-1]}] f_2^+  = \zeta_2 f_0^-\e^2\, , \\
&P_0  P_{2}^{[0]} f_0^- =
% P_0 \mathcal{P}[\cB_1^{[+1]},\cB_1^{[-1]}] f_0^- +  P_0 \mathcal{P}[\cB_1^{[-1]},\cB_1^{[+1]}] f_0^- =
\zeta_3 f_0^-\e^2 \, , \quad
P_0  P_{2}^{[0]} f_2^+ =
% P_0 \mathcal{P}[\cB_1^{[+1]},\cB_1^{[-1]}] f_2^+ +  P_0 \mathcal{P}[\cB_1^{[-1]},\cB_1^{[+1]}] f_2^+=
\zeta_4 f_2^+ \e^2 \, , 
\end{aligned}
\end{equation}
with $\zeta_3,\zeta_4\in \bC$ and, again by \eqref{insideloop} and \eqref{residuevalue2}, 
\begin{equation}\label{P0P2coeffs}
\begin{aligned}
&\zeta_1 := \frac{\ent{1}{+1}{1}{0}{-}{-}\ent{1}{+1}{2}{1}{+}{-}}{(\omega_1^--\omega_*)^2}-\frac{\ent{1}{+1}{1}{0}{+}{-}\ent{1}{+1}{2}{1}{+}{+}}{(\omega_1^+-\omega_*)^2} = \im \frac{5 \sqrt{3}}{16} \\
&  \zeta_2 :=  \frac{\ent{1}{-1}{1}{2}{+}{+}\ent{1}{-1}{0}{1}{-}{+}}{(\omega_1^+-\omega_*)^2} - \frac{\ent{1}{-1}{1}{2}{-}{+}\ent{1}{-1}{0}{1}{-}{-}}{(\omega_1^--\omega_*)^2}\, = \im \frac{5 \sqrt{3}}{16}  \ .
%&\zeta_2 := \frac{\ent{1}{-1}{-1}{0}{+}{-}\ent{1}{+1}{0}{-1}{-}{+}}{(\omega_{-1}^+-\omega_*)^2}-\frac{\ent{1}{-1}{-1}{0}{-}{-}\ent{1}{+1}{0}{-1}{-}{-}}{(\omega_{-1}^--\omega_*)^2}+  \frac{\ent{1}{+1}{1}{0}{+}{-}\ent{1}{-1}{0}{1}{-}{+}}{(\omega_{1}^+-\omega_*)^2}-\frac{\ent{1}{+1}{1}{0}{-}{-}\ent{1}{-1}{0}{1}{-}{-}}{(\omega_{1}^--\omega_*)^2} \, , 
%\\
%&\zeta_3 := \frac{\ent{1}{-1}{1}{2}{-}{+}\ent{1}{+1}{2}{1}{+}{-}}{(\omega_{1}^--\omega_*)^2}-\frac{\ent{1}{-1}{1}{2}{+}{+}\ent{1}{+1}{2}{1}{+}{+}}{(\omega_{1}^+-\omega_*)^2} + \frac{\ent{1}{+1}{3}{2}{-}{+}\ent{1}{-1}{2}{3}{+}{-}}{(\omega_{3}^--\omega_*)^2}- \frac{\ent{1}{+1}{3}{2}{+}{+}\ent{1}{-1}{2}{3}{+}{+}}{(\omega_{3}^+-\omega_*)^2} \, ,
\end{aligned}
\end{equation}
Note that $\zeta_1 = -\bar\zeta_2 $ by \eqref{entprop}. By \eqref{Lcalc} 
the term $\mathbf{L}$ in \eqref{def:ellone} is given by
\begin{align}\notag
&-\frac{2 \mathbf{L}}{\e^2}=  
 \zeta_3
  \underbrace{\big({\cal B}_1^{[+1]} P_1^{[+1]} f_0^-\, , \, f_{2}^+\big)}_{=\mathrm{Bb}=0}
  + \zeta_1
  \underbrace{ \big({\cal B}_1^{[-1]} P_1^{[+1]} f_2^+ , \, f_{2}^+\big)}_{=\mathrm{Ac}}
  + \zeta_1 
   \underbrace{\big({\cal B}_1^{[+1]} P_1^{[-1]}  f_2^+\, , \, f_{2}^+\big)}_{=\mathrm{Ab}}\\ \label{lastterm}
& + \bar{\zeta_4}
\underbrace{ \big({\cal B}_1^{[+1]}  f_0^-\, , \, P_1^{[-1]}  f_{2}^+\big)}_{=\mathrm{Bb} = 0}  +\bar{\zeta_2}
\underbrace{ \big({\cal B}_1^{[-1]}  f_0^-\, , \, P_1^{[-1]}  f_{0}^-\big)}_{=\mathrm{Cb}} 
+\bar{\zeta_2} 
\underbrace{\big({\cal B}_1^{[+1]}  f_0^-\, , \, P_1^{[+1]} f_{0}^-\big)}_{=\mathrm{Cc}}  \, ,
\end{align}
with $\mathrm{Ab}$, $\mathrm{Ac}$, $\mathrm{Bb}$, $\mathrm{Cb}$ and $\mathrm{Cc}$ defined and computed in \eqref{A}, \eqref{B}, \eqref{C} (recall also \eqref{notazionefinale}). We then have 
\begin{equation}\label{Lfinaleecco}
    \mathbf{L} = - \im\frac{305}{512}  \sqrt{3} \e^4\, .
\end{equation}
By   \eqref{beta3expfin}, \eqref{beta3.coeff} and \eqref{Lfinaleecco} we conclude that $\beta_3 = -\dfrac{39 \sqrt{3}}{512}$ as stated in \eqref{epsilonspento2}. This completes the proof of Proposition \ref{expbT}. \medskip

%------
% Insert acknowledgments and information
% regarding funding at the end of the last
% section, i.e., right before the bibliography.
%------
\footnotesize{
\noindent {\bf Acknowledgments.}
We thank W. Strauss, B. Deconinck and V. Hur for several insightful conversations.
This work is supported by~ PRIN 2020 (2020XB3EFL001) “Hamiltonian and dispersive PDEs”, PRIN 2022 (2022HSSYPN)   "TESEO - Turbulent Effects vs Stability in Equations from Oceanography".
P.\ Ventura is supported by the ERC STARTING GRANT 2021 ``Hamiltonian Dynamics, Normal Forms and Water Waves" (HamDyWWa), Project Number: 101039762.
 Views and opinions expressed are however those of the authors only and do not necessarily reflect those of the European Union or the European Research Council. Neither the European Union nor the granting authority can be held responsible for them.
}

\vspace{-1em}

\end{document}